\documentclass{amsart}
\usepackage{latexsym}
\usepackage{amsmath}
\usepackage{amssymb}
\usepackage{amsthm}
\usepackage{amscd}
\usepackage{enumerate} 
\usepackage{amssymb} 
\usepackage{mathrsfs}
\usepackage[all]{xy}
\usepackage{color}
\usepackage{graphicx}
\usepackage{mathtools}
\usepackage{comment}
\usepackage{multirow}
\usepackage[pdftex]{hyperref}
\usepackage[foot]{amsaddr}
\hypersetup{colorlinks=true}
\usepackage{subcaption}
\usepackage{tikz}
\usetikzlibrary{positioning}
\usepackage{capt-of}
\newsavebox{\measurebox}

\title{Pathologies and liftability of Du Val del Pezzo surfaces in positive characteristic}

\author{Tatsuro Kawakami\textsuperscript{1}}
\email{\textsuperscript{1}kawakami@ms.u-tokyo.ac.jp}
\address{\textsuperscript{1}Graduate School of Mathematical Sciences, University of Tokyo, 3-8-1 Komaba, Meguro-ku, Tokyo 153-8914, Japan}
\author{Masaru Nagaoka\textsuperscript{2}}
\email{\textsuperscript{2}m-nagaoka@imi.kyushu-u.ac.jp}
\address{\textsuperscript{2}Institute of Mathematics for Industry, Kyushu University, 744 Motooka, Nishi-ku, Fukuoka 819-0395, Japan}

\def\phi{\varphi}
\def\epsilon{\varepsilon}
\def\tilde{\widetilde}

\def\mapsto{\longmapsto}

\def\log{\operatorname{log}}

\def\Spec{\operatorname{Spec}}

\def\Supp{\operatorname{Supp}}

\def\Aut{\operatorname{Aut}}
\def\PGL{\operatorname{PGL}}
\def\MW{\operatorname{MW}}
\def\Dyn{\operatorname{Dyn}}

\newcommand{\Q}{\mathbb{Q}} 
\newcommand{\C}{\mathbb{C}} 
 
\newcommand{\Z}{\mathbb{Z}}
\newcommand{\PP}{\mathbb{P}}
\newcommand{\FF}{\mathbb{F}}
\newcommand{\ZZ}{\mathbb{Z}}

\newcommand{\sO}{\mathcal{O}}

\theoremstyle{plain}
\newtheorem{thm}{Theorem}[section] 
\newtheorem{cor}[thm]{Corollary}
\newtheorem{prop}[thm]{Proposition}

\newtheorem{lem}[thm]{Lemma}
\theoremstyle{definition} 
\newtheorem{defn}[thm]{Definition}

\theoremstyle{remark}
\newtheorem{rem}[thm]{Remark}

\newtheorem{defn and notation}[thm]{Definition and Notation}

\newtheorem*{clproof}{Proof of Claim}
\newtheorem{cln}{Claim}

\begin{document}
\tolerance = 9999

\maketitle
\markboth{Tatsuro Kawakami and Masaru Nagaoka}{Du Val del Pezzo surfaces in positive characteristic}

\begin{abstract}
In this paper, we study pathologies of Du Val del Pezzo surfaces defined over an algebraically closed field of positive characteristic by relating them to their non-liftability to the ring of Witt vectors.
More precisely, 
we investigate
the condition (NB): all the anti-canonical divisors are singular, 
(ND): there are no Du Val del Pezzo surfaces over the field of complex numbers with the same Dynkin type, Picard rank, and anti-canonical degree, 
(NK): there exists an ample $\Z$-divisor which violates the Kodaira vanishing theorem for $\Z$-divisors, and
(NL): the pair $(Y, E)$ does not lift to the ring of Witt vectors, where $Y$ is the minimal resolution and $E$ is its reduced exceptional divisor. 
As a result, for each of these conditions, we determine all the Du Val del Pezzo surfaces which satisfy the given one.
\end{abstract}

\tableofcontents

\section{Introduction}
We say that $X$ is a Du Val del Pezzo surface if $X$ is a normal projective surface whose anti-canonical divisor is ample and which has at worst Du Val singularities, i.e., 2-dimensional canonical singularities.
By \textit{the Dynkin type of $X$}, we mean the corresponding Dynkin diagrams of singularities on $X$. 
For example, we say that $X$ is of type $3A_1+D_4$ if $X$ has three $A_1$-singularities and one $D_4$-singularity.  In this case, we also write $\Dyn(X) = 3A_1+D_4$ and $X=X(3A_1+D_4)$. 

In positive characteristic, it has become clear that many pathological phenomena occur on Du Val del Pezzo surfaces. 
For example, Keel-M\textsuperscript{c}Kernan \cite[end of Section 9]{KM} constructed a Du Val del Pezzo surface $X(7A_1)$ of Picard rank one and degree $K_X^2=2$ in characteristic two. This Dynkin type does not appear in characteristic zero (see \cite[Theorem 2, Table (II)]{Fur} or \cite[Theorem 1.1]{Bel}).
Furthermore, Cascini-Tanaka \cite[Proposition 4.3 (iii)]{CT18} pointed out that anti-canonical members of $X(7A_1)$ are all singular. Since the complete linear system of the anti-canonical divisor of $X(7A_1)$ is base point free, this gives a counterexample to Bertini's theorem in positive characteristic.
Cascini-Tanaka \cite[Theorem 4.2 (6)]{CT19} also showed that there exists an ample $\Z$-divisor $A$ on $X(7A_1)$ such that $H^1(X, \sO_X(-A))\neq 0$, which gives a counterexample to the Kodaira vanishing theorem for ample $\Z$-divisors on klt surfaces. 

On the other hand, the question of whether a variety admits a lifting to the ring of Witt vectors $W(k)$ is often related to pathological phenomena in positive characteristic.
Since all Du Val del Pezzo surfaces considered by themselves are always liftable (see Remark \ref{lift remark}), it is more useful to consider the notion of log liftability.  
\begin{defn}[\textup{cf.~Definition \ref{loglift:Def} and Lemma \ref{equiv}}]\label{Intro:loglift:Def}
We say that a Du Val del Pezzo surface $X$ is \textit{log liftable over the ring of Witt vectors $W(k)$} if for the minimal resolution $\pi \colon Y \to X$, the pair of $Y$ and its reduced exceptional divisor $E_\pi$ lifts to $W(k)$.
\end{defn}
Indeed, it is known that the Keel-M\textsuperscript{c}Kernan's surface $X(7A_1)$ is not log liftable over $W_2(k)$ (see \cite[Proposition 11.1]{Langer19} and \cite[Proposition 4.1]{Lan}).

The aim of this paper is to investigate the relationship between these pathological phenomena and non-log liftability of Du Val del Pezzo surfaces over $W(k)$.
For simplicity of notation, we define the following conditions.
\begin{defn}\label{notation}
For a Du Val del Pezzo surface $X$ over an algebraically closed field $k$ of characteristic $p>0$, we say that $X$ satisfies:
\begin{itemize}
    \item (ND) if there does not exist any Du Val del Pezzo surface $X_{\C}$ over the field of complex numbers $\C$ with the same Dynkin type, the same Picard rank, and the same degree as $X$.
    \item (NB) if anti-canonical members of $X$ are all singular.
    \item (NK) if $H^1(X, \sO_X(-A))\neq0$ for some ample $\Z$-divisor $A$ on $X$.
    \item (NL) if $X$ is not log liftable over $W(k)$.
\end{itemize}
\end{defn}

Our main results consist of three theorems.
One is the following, which shows that (NK) $\Rightarrow$ (NL) and (ND) $\Rightarrow$ (NL) $\Rightarrow$ (NB).

\begin{thm}\label{smooth, Intro}
Let $X$ be a Du Val del Pezzo surface over an algebraically closed field $k$ of characteristic $p>0$.
Then the following hold.
\begin{enumerate}\renewcommand{\labelenumi}{$($\textup{\arabic{enumi}}$)$}
\item{If a general anti-canonical member is smooth, then $X$ is log liftable over $W(k)$.}
\item{If $X$ is log liftable over $W(k)$, then there exists a Du Val del Pezzo surface over $\C$ with the same Dynkin type, the same Picard rank, and the same degree as $X$.}
\item{If $X$ is log liftable over $W(k)$, then $H^1(X, \sO_X(-A))=0$ for every ample $\Z$-divisor $A$.}
\end{enumerate}
\end{thm}

The second main theorem is Theorem \ref{sing}, which classifies Du Val del Pezzo surfaces satisfying (NB), the weakest condition among four pathological phenomena in Definition \ref{notation}.

\begin{thm}\label{sing}
Let $X$ be a Du Val del Pezzo surface over an algebraically closed field $k$ of characteristic $p>0$.
Suppose that anti-canonical members of $X$ are all singular.
Then the following hold.
\begin{enumerate}
    \item[\textup{(0)}] $K_X^2 \leq 2$ and $p=2$ or $3$.
    \item[\textup{(1)}] When $K_X^2=1$ and $p=2$ (resp.~$p=3$), the Dynkin type of $X$ is $E_8$, $D_8$, $A_1+E_7$, $2D_4$, $2A_1+D_6$, $4A_1+D_4$, or $8A_1$ (resp.~$E_8$, $A_2+E_6$, or $4A_2$). In particular, the Picard rank of $X$ is equal to one.
    \item[\textup{(2)}] When $K_X^2=2$, we have $p=2$ and the Dynkin type of $X$ is $E_7$, $A_1+D_6$, $3A_1+D_4$, or $7A_1$.
    Furthermore, the anti-canonical morphism $\phi_{|-K_X|}\colon X\to \PP_k^2$ is purely inseparable and hence $X$ is homeomorphic to $\PP^2_k$. In particular, the Picard rank of $X$ is equal to one.
    \item[\textup{(3)}] 
    The isomorphism class of $X$ is uniquely determined by its Dynkin type except when the Dynkin type is $2D_4, 4A_1+D_4$, or $8A_1$.
    In these cases, the isomorphism classes of del Pezzo surfaces of type $8A_1$ (resp.~each of types $2D_4$ and $4A_1+D_4$) correspond to the closed points of $\mathcal{D}_n / \PGL(n+1, \FF_2)$ with $n=2$ (resp.~$n=1$), where $\mathcal{D}_n \subset \PP^n_k$ is the complement of the union of all the hyperplane sections defined over the prime field $\FF_2$ of $k$.
\end{enumerate}

Summarizing these statements, we obtain Table \ref{table:sing}.

\begin{table}[htbp]
\caption{}
  \begin{tabular}{|c|c|c|c|c|c|c|c|c|c|c|} 
  \hline
  \multicolumn{3}{|c||}{Degree} & \multicolumn{5}{|c|}{$K_X^2=1$}            \\
  \hline 
\multicolumn{3}{|c||}{Dynkin type}                 &$E_8$   & $A_2+E_6$ & \multicolumn{1}{c|}{$4A_2$} & $D_8$ & $A_1+E_7$            \\ \hline 
\multicolumn{3}{|c||}{Characteristic}              &$p=2, 3$& \multicolumn{2}{|c|}{$p=3$}             &\multicolumn{2}{|c|}{$p=2$}   \\ \hline
\multicolumn{3}{|c||}{No. of isomorphism classes}  & $1$    & $1$       & \multicolumn{1}{c|}{$1$}    & $1$   & $1$                  \\ \hline \cline{1-8} 
\multicolumn{4}{|c||}{$K_X^2=1$}                                             & \multicolumn{4}{|c|}{$K_X^2=2$}\\\cline{1-8}
$2D_4$    & $2A_1+D_6$ & $4A_1+D_4$& \multicolumn{1}{c||}{$8A_1$}       & $E_7$& $A_1+D_6$ & $3A_1+D_4$& $7A_1$ \\ \cline{1-8}
\multicolumn{4}{|c||}{$p=2$}                                            & \multicolumn{4}{|c|}{$p=2$}           \\ \cline{1-8}
$\infty$  & $1$        &$\infty$   & \multicolumn{1}{c||}{$\infty$ }     & $1$  & $1$       & $1$       & $1$    \\ \cline{1-8}
  \end{tabular}
  \label{table:sing}
\end{table}
\end{thm}
\begin{rem}
The list of root bases in the lattice $\mathbb{E}_{9-d}$ gives the list of possible Dynkin types of Du Val del Pezzo surfaces of degree $d$.
When $k= \C$ and $d=2$ (resp.~$d=1$), it is also shown that any root bases are realized as the Dynkin type of a Du Val del Pezzo surfaces of degree $d$ except $7A_1$ (resp.~$4A_1+D_4, 8A_1$, and $7A_1$) (see \cite[Chapter 8]{Dol} and the references given there for more details).
Theorems \ref{smooth, Intro} and \ref{sing} show that $7A_1$ (resp.~each of $4A_1+D_4$ and $8A_1$) is realized as the Dynkin type of a Du Val del Pezzo surface of degree two (resp.~one) only in characteristic two.
On the other hand, $7A_1$ cannot be realized as the Dynkin type of a Du Val del Pezzo surface of degree one in any characteristic.
\end{rem}
\begin{rem}
Theorem \ref{sing} describes isomorphism classes of several rational quasi-elliptic surfaces and, combining Ito's results \cite{Ito1, Ito2}, we obtain the complete classification of isomorphism classes of rational quasi-elliptic surfaces (see Corollary \ref{Itorem} for more details). In the process of the proof of Theorem \ref{sing}, we also describe automorphism group scheme structures of all rational quasi-elliptic surfaces (see Corollary \ref{cor:q-ellaut} and Remark \ref{autorem}).
\end{rem}

The last main theorem is Theorem \ref{pathologies}. It clarifies which Du Val del Pezzo surfaces satisfying (NB) also satisfy additional pathological phenomena.
As a consequence, we conclude that  (NK) $\Rightarrow$ (ND) $\Rightarrow$ (NL) $\Rightarrow$ (NB) and none of the opposite directions hold.

\begin{thm}\label{pathologies}
Let $X$ be a Du Val del Pezzo surface over an algebraically closed field $k$ of characteristic $p>0$.
Then the following hold.
 \begin{enumerate}
 \item [\textup{(1)}] The surface $X$ is not log liftable over $W(k)$ if and only if $(p, \Dyn(X))=(3, 4A_2)$, $(2, 4A_1+D_4)$, $(2, 8A_1)$, or $(2, 7A_1)$.
 \item [\textup{(2)}] There exists no Du Val del Pezzo surface $X_{\C}$ over $\C$ with the same Dynkin type, the same Picard rank, and the same degree as $X$ if and only if $(p, \Dyn(X))=(2, 4A_1+D_4)$, $(2, 8A_1)$, or $(2, 7A_1)$.
 \item [\textup{(3)}] There exists an ample $\Z$-divisor on $X$ such that $H^1(X, \sO_X(-A))\neq0$ if and only if $(p, \Dyn(X))=(2, 8A_1)$ or $(2, 7A_1)$.
 \end{enumerate}
\end{thm}

\begin{rem}
The vanishing of the second cohomology of the tangent sheaf is important because this implies that there are no local-to-global obstructions to deformations (cf.~\cite[Theorem 4.13]{LN}). 
Theorem \ref{pathologies} asserts that there exists a Du Val del Pezzo surface $X$ with $H^2(X, T_X) \neq 0$ when $p=2$ or $3$. We remark that this does not happen when $p>3$ since a general anti-canonical member is smooth in this case (see the proof of Proposition \ref{NBtoNL} for more details).
\end{rem}

\subsection{Structure of the paper}
%%%%%%%%%%%%%%%%%%%%%%%%%%%%%%%%%%%%%%%%%%%%%%%%%%%%%%%%%%%%%%%%%%%%%%%%%%%%%%%%%%%%%%%%%%%%%%%%%%%%%%%%%%%%%%%

This paper is structured as follows.
In Section \ref{sec:pre}, we recall some facts on liftability of pairs, Du Val del Pezzo surfaces, and rational quasi-elliptic surfaces.
In Section \ref{sec:smooth}, we prove Theorem \ref{smooth, Intro}.
The main idea of the proof is to utilize a smooth anti-canonical member of a Du Val del Pezzo surface $X$ to show the vanishing of $H^2(X, T_X)$,
which is the obstruction of log lifting over $W(k)$.
After that, we construct the desired surface over $\C$ from the generic fiber of the log lifting.
Theorem \ref{smooth, Intro} (3) is an easy consequence of Hara's vanishing theorem \cite[Corollary 3.8]{Hara98}.
In Sections \ref{sec:sing} and \ref{sec:singisom}, we prove Theorem \ref{sing} by using the description of the configuration of negative rational curves on rational quasi-elliptic surfaces by Ito \cite{Ito1, Ito2}. In these sections, we also determine the automorphism groups of Du Val del Pezzo surfaces satisfying (NB).

In the rest of the paper, we prove Theorem \ref{pathologies} as follows.
We only have to consider Du Val del Pezzo surfaces satisfying (NB) due to Theorem \ref{smooth, Intro}.
Fix such a surface $X$.
We note that the assertion (2) now follows from the list of Du Val del Pezzo surfaces over $\C$ of Picard rank one \cite[Theorem 2, Table (II)]{Fur}.

In Section \ref{sec:singliftable}, we show the assertion (1).
The main difficulty of the proof is to show non-log liftability of $X(4A_2)$ over $W(k)$, which does not satisfy (ND).
For the proof, we observe that log liftability of $X(4A_2)$ over $W(k)$ would give the extremal rational elliptic surface over $\C$ constructed from the Hesse pencil $\Lambda=\{s(x^3+y^3+z^3)+t(xyz)\}$.
Then we obtain a contradiction by showing that $W(k)$ must contain the cube root of unity, which appears as the singular member locus of $\Lambda$.
In Section \ref{sec:KV}, we show the assertion (3). 
By \cite[Theorem 4.8]{Kaw2} and Theorem \ref{smooth, Intro} (3), it suffices to consider the case where $X$ is of type $7A_1$, $8A_1$, or $4A_1+D_4$.
\cite[Theorem 4.2 (6)]{CT19} shows that $X(7A_1)$ satisfies (NK).
Choosing an ample $\Z$-divisor based on [\textit{ibid}.], we show that $X(8A_1)$ also satisfies (NK).
On the other hand, we utilize a birational map between $X(4A_1+D_4)$ and the Du Val del Pezzo surface of type $7A_1$ to conclude that $X(4A_1+D_4)$ does not satisfy (NK).
We have thus proved Theorem \ref{pathologies}.

\subsection{Related results}
Log liftability was originally considered by Cascini-Tanaka-Witaszek \cite{CTW}, in which they proved that surfaces of del Pezzo type are either log liftable over $W(k)$ or globally $F$-regular in large characteristic. After that Lacini \cite{Lac} classified klt del Pezzo surfaces of Picard rank one in characteristic $p>3$, from which he deduced that klt del Pezzo surfaces of Picard rank one in $p>5$ are log liftable over $W(k)$. Arvidsson-Bernasconi-Lacini \cite{ABL} generalized his result to the case of klt projective surfaces $X$ with Iitaka dimension $\kappa(X, K_X)=-\infty$. The first author \cite{Kaw21} showed that normal projective surfaces $X$ in large characteristic with Iitaka dimension $\kappa(X, K_X)\leq0$ are log liftable to characteristic zero. 
The second author \cite{Nag21} determined klt del Pezzo surfaces of Picard rank one in $p=5$ which are not log liftable over $W(k)$.

\subsubsection{Classification of del Pezzo surfaces in positive characteristic}
Lacini \cite{Lac} classified klt del Pezzo surfaces of Picard rank one in characteristic $p>3$. Inspired by his work, the authors \cite{KN20} determined Du Val del Pezzo surfaces of Picard rank one in $p=2$ or $3$.
Martin and Stadlmayr \cite{M-S} classified smooth weak del Pezzo surfaces with non-zero global vector fields in any characteristic.

\subsection{Notation}
We work over an algebraically closed field $k$ of characteristic $p>0$.
A \textit{variety} means an integral separated scheme of finite type over $k$. 
A \textit{curve} (resp.~a \textit{surface}) means
a variety of dimension one (resp.~two). 
We call two-dimensional canonical singularities \textit{Du Val singularities}.
We always require quasi-elliptic surfaces to be relatively minimal.
Throughout this paper, we also use the following notation:
\begin{itemize}
\item $\FF_q$: the finite field of order $q$.
\item $W(k)$ (resp.~$W_n(k)$): the ring of Witt vectors (resp.~the ring of Witt vectors of length $n$).
\item $E_f$: the reduced exceptional divisor of a birational morphism $f$.
\item $\rho(X)$: the Picard rank of a projective variety $X$.
\item $T_X\coloneqq \mathcal{H}om_{\sO_X}(\Omega_X, \sO_X)$: the tangent sheaf of a normal variety $X$
\item $T_Y(-\log E) \coloneqq \mathcal{H}om_{\sO_Y}(\Omega_Y(\log E), \sO_Y)$: the logarithmic tangent bundle of a smooth variety $Y$ and a simple normal crossing divisor $E$ on $Y$.
\item $\Aut X$: the automorphism group of a variety $X$. 
\item $\MW(Z)$: the Mordell-Weil group of a genus one fibration $f \colon Z \to \PP^1_k$ defined by $|-K_Z|$.
\item $\Dyn(X)$ : the Dynkin type of $X$.
\end{itemize}

\section{Preliminaries}
\label{sec:pre}

\subsection{Liftability of pairs}
In this subsection, we review a general theory of liftability to Noetherian irreducible schemes.
\begin{defn}
Let $T$ be a Noetherian irreducible scheme, $Y$ be a smooth separated scheme over $T$, and $E \coloneqq \sum_{i=1}^r E_i$ be a reduced divisor on $Y$, where each $E_i$ is an irreducible component of $E$. We say that $E$ is \textit{simple normal crossing over $T$} if, for any subset $J \subseteq \{1, \ldots, r\}$ such that $\bigcap_{i \in J} E_i\neq \emptyset$, the scheme-theoretic intersection $\bigcap_{i \in J} E_i$ is smooth over $T$ of relative dimension $\dim Y_{\eta}-|J|$, where $Y_{\eta}$ is the generic fiber of $Y\to T$.
\end{defn}

\begin{defn}\label{d-liftable}
Let $\alpha \colon S \to T$ be a morphism between Noetherian irreducible schemes.
Let $Y$ be a smooth projective scheme over $S$ and $E$ a simple normal crossing divisor over $S$ on $Y$ with $E = \sum_{i=1}^r E_i$ the irreducible decomposition. 
We say that the pair $(Y,E)$ \textit{lifts to $T$ via $\alpha$} if 
there exist 
\begin{itemize}
	\item a smooth and projective morphism $\mathcal{Y} \to T$ and
	\item effective divisors $\mathcal{E}_1, \dots, \mathcal E_r$ on $\mathcal{Y}$ such that $\sum _{i=1}^r \mathcal{E}_i$ is simple normal crossing over $T$
\end{itemize}
such that the base change of the schemes $\mathcal{Y}, \mathcal{E}_1,\cdots,\mathcal{E}_r$ by $\alpha \colon S \to T$ are isomorphic to $Y, E_1,\cdots, E_r$ respectively. 
When $T$ is the spectrum of a local ring $(R, m)$ and $\alpha$ is induced by $R/m \cong k$, we also say that $(Y, E)$ \textit{lifts to $R$} for short.
\end{defn}

The following theorem is a log version of \cite[Theorem 8.5.9]{FAG}. It seems to be well-known for experts, but we include the sketch of the proof for the convenience of the reader.
\begin{thm}\label{loglift:criterion 1}
Let $Y$ be a smooth projective variety 
and $E$ a simple normal crossing divisor on $Y$. If $H^2(Y, T_{Y}(-\log\,E))=H^2(Y, \sO_Y)=0$, then $(Y ,E)$ lifts to every Noetherian complete local ring $(R, m)$ with the residue field $k$.
\end{thm}
\begin{proof}
We denote $R/m^n$ by $R_n$.
Let $(Y^n, E^n)$ be a lifting of $(Y, E)$ over $\Spec\,R_n$. We first see that $(Y^n, E^n)$ is liftable to $\Spec\,R_{n+1}$. Since $E^n$ is simple normal crossing over $\Spec\,R_n$, we can take an affine open covering $\{U_i\}$ of $Y^n$ such that $(U_i, E|_{U_i})$ lifts to $\Spec\,R_{n+1}$.
Then for each $i$ and any open subset $U$ of $U_{i}$, the set of equivalence classes of such liftings is a torsor under the action of $\mathcal{H}om(\Omega_{U}(\log E), m^{n-1}\sO_{U})$. We refer to the arguments of \cite[Section 8]{EV} for the details. Then by a similar argument as in \cite[Theorem 8.5.9 (b)]{FAG}, 
the obstruction for the lifting of $(Y^n, E^n)$ over $\Spec\,R_{n+1}$ is contained in $H^2(Y, T_{Y}(-\log\,E))\otimes m^{n}/m^{n+1}$.
Thus the vanishing of $H^2(Y, T_{Y}(-\log\,E))$ gives a lifting of $Y$ and $E_{i}$ over $\Spec R$ as formal schemes.
Since $H^2(Y, \sO_Y)=0$, they are algebraizable and we get a projective scheme $\mathcal{Y}$ over $\Spec\,R$ and a closed subscheme $\mathcal{E} \coloneqq \sum_{i=1}^r \mathcal{E}_i$ on $\mathcal{Y}$ such that 
$\mathcal{Y}\otimes_{R} R_n =Y^n$ and $\mathcal{E}_i\otimes_{R} R_n=E^n_i$ for each $n$ and $i$ by \cite[Corollary 8.5.6 and Corollary 8.4.5]{FAG}.
We take a subset $J \subseteq \{1, \ldots, r\}$. Since $(\bigcap_{i \in J} \mathcal{E}_i)\otimes_{R} R_n=\bigcap_{i \in J} {E}^n_i$ is smooth over $\Spec\,R_n$ for all $n>0$ and $\mathcal{Y}$ is projective over $\Spec\,R$, \cite[Chapitre 0, Proposition (10.2.6)]{EGAIII} and \cite[Th\'eor\`eme 12.2.4 (iii)]{EGAIV3} show that $\bigcap_{i \in J} \mathcal{E}_i$ is smooth of relative dimension $\dim \mathcal{Y}_{\eta}-|J|$ except when $\bigcap_{i \in J} \mathcal{E}_i = \emptyset$, where $\mathcal{Y}_{\eta}$ is the generic fiber. Therefore $(\mathcal{Y}, \mathcal{E}=\sum_{i=1}^r \mathcal{E}_i)$ is a lifting of $(Y, E)$ over $\Spec\,R$.
\end{proof}

We use the following Hara's vanishing theorem in Propositions \ref{prop:W2lift} and \ref{lem:KV-3}.

\begin{thm}[\textup{cf.~\cite[Corollary 3.8]{Hara98}}]\label{loglift vanishing}
Let $Y$ be a smooth projective variety and $E$ a simple normal crossing divisor on $Y$ such that $(Y, E)$ lifts to $W_2(k)$.
Let $A$ be an ample $\Q$-divisor the support of whose fractional part is contained in $E$.
Suppose that $p\geq \dim\,Y$.
Then $H^{j}(Y, \Omega_Y^{i}(\log \, E)\otimes\sO_Y(-\lceil A \rceil))=0$ if $i+j<\dim\, Y$.
\end{thm}
\begin{proof}
When $p>\dim\,Y$, the assertion follows from \cite[Corollary 3.8]{Hara98} and the Serre duality.
We remark that even when $p=\dim\,Y$, the assertion holds. This is because, in the proof of \cite[Corollary 3.8]{Hara98}, the assumption that $p>\dim\,Y$ is only used for the quasi-isomorphism
$\bigoplus_{i}\Omega_Y^i(\log\, E)[-i]\simeq F_{*}\Omega_Y^{\bullet}(\log E)$, which holds even in $p=\dim\,Y$ as in \cite[10.19 Proposition]{EV}.
\end{proof}

As we will see in Remark \ref{lift remark}, all Du Val del Pezzo surfaces lift to every Noetherian complete local ring with the residue field $k$.
For this reason, we will mainly consider the following notion of liftability.

\begin{defn}\label{loglift:Def}
Let $X$ be a normal projective surface. Fix a Noetherian irreducible scheme $T$ and a morphism $\alpha \colon \Spec k \to T$.
We say that $X$ is \textit{log liftable over $T$ via $\alpha$} (or \textit{log liftable over $R$ via $\alpha$} when $T= \Spec R$) if the pair $(Z, E_{f})$ lifts to $T$ via $\alpha$ for some log resolution $f \colon Z \to X$.
When $T$ is the spectrum of a local ring $(R, m)$ and $\alpha$ is induced by $R/m \cong k$, 
We also say that $X$ is \textit{log liftable over $R$} for short.
\end{defn}

\begin{rem}
We remark that log liftability over a scheme smooth and separated over $\Z$ is equivalent to log liftability over $W(k)$ by \cite[Proposition 2.5]{ABL}.
\end{rem}

Let us see conditions equivalent to log liftability. 
We note that the minimal resolution of a Du Val singularity is a log resolution.

\begin{lem}\label{equiv}
Let $X$ be a normal projective surface and $R$ a Noetherian complete local ring with the residue field $k$. 
Let $\pi \colon Y\to X$ be a log resolution and $f \colon Z \to X$ a log resolution which factors through $\pi$.
Then the following holds.
\begin{enumerate}
    \item[\textup{(1)}] Suppose that $(Z, E_{f})$ lifts to $R$. 
    Then $(Y, E_{\pi})$ lifts to $R$ as a formal scheme. 
    If $H^2(Z, \sO_Z)=0$ in addition, then $(Y, E_{\pi})$ lifts to $R$.
    \item[\textup{(2)}] Suppose that $(Y, E_{\pi})$ lifts to $R$ and $R$ is regular. 
    Then $(Z, E_{f})$ lifts to $R$.
\end{enumerate}
\end{lem}
\begin{proof}
The assertion (1) is \cite[Proposition 4.3 (1)]{AZ}. 
We show (2). Since the morphism $Z\to Y$ is a birational morphism of smooth projective surfaces, this is a composition of blow-ups at a smooth point.
Since $R$ is complete and regular, the essentially same argument as \cite[Proposition 2.9]{ABL} shows the liftability of $(Z, E_{f})$.
 \end{proof}

In particular, Definition \ref{Intro:loglift:Def} agrees with Definition \ref{loglift:Def}.

\begin{prop}\label{prop:W2lift}
Let $X$ be a normal projective surface.
Suppose that one of the following conditions holds.
\begin{enumerate}
    \item[\textup{(1)}] $-K_X$ is ample $\Q$-Cartier, the minimal resolution $\pi\colon Y\to X$ is a log resolution, and there exists a log resolution $f\colon Z\to X$ such that $(Z, E_f)$ lifts to $W_2(k)$.
    \item[\textup{(2)}] $H^2(X, T_X)=0$ and $H^2(X, \sO_X)=0$.
\end{enumerate}
Then, for every log resolution $f'\colon Z'\to X$, the pair $(Z',E_{f'})$ lifts to every Noetherian complete local ring with residue field $k$.
\end{prop}
\begin{proof}
We first show (1).
We take a $\pi$-exceptional effective $\Q$-divisor $F$ such that $\pi^{*}(-K_X)-F$ is ample and $\lceil \pi^{*}(-K_X)-F\rceil=\lceil \pi^{*}(-K_X)\rceil$.
Since $\pi$ is minimal, we have $-K_Y \geq \lceil \pi^{*}(-K_X) \rceil=\lceil \pi^{*}(-K_X)-F\rceil$.
Let $f'\colon Z'\to X$ be a log resolution.
Then $f'$ decomposes into $g\colon Z'\to Y$ and the minimal resolution $\pi\colon Y\to X$.
We have the injective morphism
\begin{align*}
g_{*}(\Omega_{Z'}(\log\,E_{f'})\otimes \sO_{Z'}(K_{Z'}))\hookrightarrow&
(g_{*}(\Omega_{Z'}(\log\,E_{f'})\otimes \sO_{Z'}(K_{Z'})))^{**}\\
=&\Omega_Y(\log\,E_{\pi})\otimes\sO_Y(K_Y)
\end{align*}
and then the Serre duality yields
\begin{align*}
H^2(Z', T_{Z'}(-\log\, E_{f'}))\cong&H^0(Z', \Omega_{Z'}(\log\,E_{f'})\otimes \sO_{Z'}(K_{Z'}))\\
\hookrightarrow&H^0(Y, \Omega_Y(\log\,E_{\pi})\otimes \sO_Y(K_Y))\\
                                 \hookrightarrow&H^0(Y, \Omega_Y(\log\,E_{\pi})\otimes \sO_Y(-\lceil \pi^{*}(-K_X)-F\rceil)).
\end{align*}
 Since $(Z, E_f)$ lifts to $W_2(k)$ by assumption, so does $(Y, E_{\pi})$ by Lemma \ref{equiv} (1), and hence the last cohomology vanishes by Theorem \ref{loglift vanishing}.
Together with 
\[
H^2(Z', \sO_{Z'})\cong H^0(Z', \sO_{Z'}(K_{Z'}))\hookrightarrow H^0(X, \sO_X(K_X))=0,\] 
we obtain the liftability of $(Z',E_{f'})$ by Theorem \ref{loglift:criterion 1}.

In the case of (2), we have
\begin{align*}
H^2(Z', T_{Z'}(-\log\, E_{f'}))\cong&H^0(Z', \Omega_{Z'}(\log\,E_{f'})\otimes \sO_{Z'}(K_{Z'}))\\
\hookrightarrow&H^0(X, (\Omega_X\otimes \sO_X(K_X))^{**})\\
                                 \cong&H^2(X, T_X)=0,
\end{align*}
and the rest proof is similar to (1). 
\end{proof}

As a consequence of Proposition \ref{prop:W2lift} (1), log liftability of a del Pezzo surface $X$ with rational singularities over $W(k)$ is equivalent to log liftability of $X$ over $W_2(k)$, and one of the conditions implies log liftability of $X$ over every Noetherian complete local ring with the residue field $k$.

\subsection{Du Val del Pezzo surfaces}
In this subsection, we gather some basic results of Du Val del Pezzo surfaces.
\begin{defn}\label{GdelPezzo}
Let $X$ be a normal projective surface.
We say that $X$ is a \textit{Du Val del Pezzo surface} if $-K_X$ is ample and $X$ has only Du Val singularities.
\end{defn}
\begin{rem}\label{lift remark}
Let $X$ be a normal projective surface with only rational singularities with Iitaka dimension $\kappa(\tilde{X}, K_{\tilde{X}})=-\infty$, where $\tilde{X}\to X$ is a resolution. Let us see that $X$ lifts to every Noetherian complete local ring $R$ with the residue field $k$.

First, we recall that $\tilde{X}$ lifts to $R$ (see \cite[8.5.26]{FAG}). Then $X$ is formally liftable to $R$ by \cite[Proposition 4.3(1)]{AZ}, and the formal lifting is algebraizable since $H^2(X, \sO_X)=0$. 
In particular, all Du Val del Pezzo surfaces lift to $R$.
\end{rem}

\begin{lem}\label{basic}
Let $X$ be a Du Val del Pezzo surface of degree $d \coloneqq K_X^2$.
Then the following hold.
\begin{enumerate}\renewcommand{\labelenumi}{$($\textup{\arabic{enumi}}$)$}
\item{$\dim|-K_X|=d$.}
\item{$|-K_X|$ has no fixed part.}
\item{A general anti-canonical member is a locally complete intersection curve with arithmetic genus one. Moreover, if $p>3$, then a general anti-canonical member is smooth.}
\item{If $d\geq3$, then $|-K_X|$ is very ample. }
\item{If $d\geq2$, then $|-K_X|$ is base point free.} 
\end{enumerate}
\end{lem}
\begin{proof}
We refer to \cite[Propositions 2.10, 2.12, and 2.14]{BT} and \cite[Proposition 4.6]{Kaw2} for the proof.
\end{proof}

\subsection{Quasi-elliptic surfaces}

In this subsection, we compile the results on rational quasi-elliptic surfaces by Ito \cite{Ito1, Ito2}, which we will use in Sections \ref{sec:sing} and \ref{sec:singisom}.

\begin{thm}[{\cite[Theorems 3.1--3.3]{Ito1}}]\label{thm:q-ell3}
Suppose $p=3$.
Then the following hold.
\begin{enumerate}
    \item[\textup{(1)}] The configurations of reducible fibers of rational quasi-elliptic surfaces and their Mordell-Weil groups are listed in Table \ref{q-ell3}, where we use Kodaira's notation. 
    \item[\textup{(2)}] Rational quasi-elliptic surfaces of each type (1), (2), and (3) uniquely exist.
    \item[\textup{(3)}] Sections on rational quasi-elliptic surfaces are disjoint from each other.
    Moreover, the dual graphs of negative rational curves in rational quasi-elliptic surfaces are as in Figure \ref{fig:dual(1)-(3)} and Table \ref{tab:type(3)}, where black nodes (resp.\ white nodes) correspond to $(-1)$-curves (resp.\ $(-2)$-curves). 
\end{enumerate}
\end{thm}

    \begin{table}[htbp]
    \centering
\caption{}
  \begin{tabular}{|c|c|c|} \hline
       Type & Reducible fibers &$\mathrm{MW}(Z)$ \\ \hline \hline
       \textup{(1)} & $\textup{II}^*$          & $\{1\}$  \\ \hline
       \textup{(2)} & $\textup{IV}^*, \textup{IV}$      & $\ZZ/3\ZZ$  \\ \hline
       \textup{(3)} & four $\textup{IV}$       & $(\ZZ/3\ZZ)^{2}$  \\ \hline
  \end{tabular}
  \label{q-ell3}
\end{table}

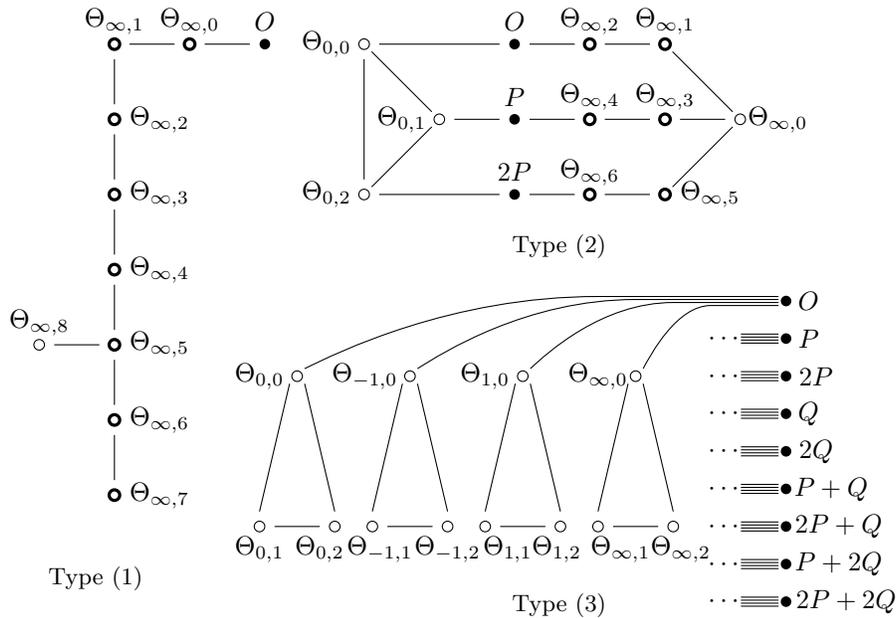
\begin{figure}[htbp]
\captionsetup[subfigure]{labelformat=empty}
    \begin{tabular}{cc}
        \multirow{2}{*}{
      \begin{minipage}[t]{0.2\hsize}
        \centering
    {\label{fig:figA}
		\begin{tikzpicture}		
		\draw[thin ](-0.2,-0)--(-0.8,0);
		\draw[thin ](-1.2,-0)--(-1.8,0);
		\draw[thin ](-2,-0.2)--(-2,-0.8);
		\draw[thin ](-2,-1.2)--(-2,-1.8);
		\draw[thin ](-2,-2.2)--(-2,-2.8);
		\draw[thin ](-2,-3.2)--(-2,-3.8);
		\draw[thin ](-2,-4.2)--(-2,-4.8);
		\draw[thin ](-2,-5.2)--(-2,-5.8);
		\draw[thin ](-2.2,-4)--(-2.8,-4);
		\fill (0,0) circle (2pt);
		\draw[very thick ] (-1,0)circle(2pt);
		\draw[very thick] (-2,0)circle(2pt);
		\draw[very thick] (-2,-1)circle(2pt);
		\draw[very thick] (-2,-2)circle(2pt);
		\draw[very thick] (-2,-3)circle(2pt);
		\draw[very thick] (-2,-4)circle(2pt);
		\draw[very thick] (-2,-5)circle(2pt);
		\draw[very thick] (-2,-6)circle(2pt);
		\draw[] (-3,-4)circle(2pt);
		\node(a)at(0,0.3){$O$};
		\node(a)at(-1,0.3){$\Theta_{\infty, 0}$};
		\node(a)at(-2,0.3){$\Theta_{\infty, 1}$};
		\node(a)at(-1.4,-1){$\Theta_{\infty, 2}$};
		\node(a)at(-1.4,-2){$\Theta_{\infty, 3}$};
		\node(a)at(-1.4,-3){$\Theta_{\infty, 4}$};
		\node(a)at(-1.4,-4){$\Theta_{\infty, 5}$};
		\node(a)at(-1.4,-5){$\Theta_{\infty, 6}$};
		\node(a)at(-1.4,-6){$\Theta_{\infty, 7}$};
		\node(a)at(-3,-3.7){$\Theta_{\infty, 8}$};		
		\end{tikzpicture}
    }        \subcaption{Type (1)}
        \label{composite}
      \end{minipage}
         }
         &
      \begin{minipage}[t]{0.7\hsize}
        \centering
 {\label{fig:figB}
		\begin{tikzpicture}

		\draw[thin ](0.2,1)--(0.8,1);
		\draw[thin ](0.2,0)--(0.8,0);
		\draw[thin ](0.2,-1)--(0.8,-1);
		\draw[thin ](1.2,1)--(1.8,1);
		\draw[thin ](1.2,0)--(1.8,0);
		\draw[thin ](1.2,-1)--(1.8,-1);
		\draw[thin ](2.1,0.9)--(2.9,0.1);
		\draw[thin ](2.2,0)--(2.8,0);
		\draw[thin ](2.1,-0.9)--(2.9,-0.1);
		
		\draw[thin ](-1.8,1)--(-0.2,1);
		\draw[thin ](-0.8,0)--(-0.2,0);
		\draw[thin ](-1.8,-1)--(-0.2,-1);
		\draw[thin ](-2,0.8)--(-2,-0.8);
		\draw[thin ](-1.9,0.9)--(-1.1,0.1);
		\draw[thin ](-1.9,-0.9)--(-1.1,-0.1);

		\fill (0,1) circle (2pt);
		\fill (0,0) circle (2pt);
		\fill (0,-1) circle (2pt);
		\draw[very thick] (1,-1)circle(2pt);
		\draw[very thick] (1,0)circle(2pt);
		\draw[very thick] (1,1)circle(2pt);
		\draw[very thick] (2,-1)circle(2pt);
		\draw[very thick] (2,0)circle(2pt);
		\draw[very thick] (2,1)circle(2pt);
		\draw[] (3,0)circle(2pt);
		\draw[] (-1,0)circle(2pt);
		\draw[] (-2,1)circle(2pt);
		\draw[] (-2,-1)circle(2pt);

		\node(a)at(0,1.3){$O$};
		\node(a)at(0,0.3){$P$};
		\node(a)at(0,-0.7){$2P$};
		\node(a)at(1,1.3){$\Theta_{\infty, 2}$};
		\node(a)at(1,0.3){$\Theta_{\infty, 4}$};
		\node(a)at(1,-0.7){$\Theta_{\infty, 6}$};
		\node(a)at(2,1.3){$\Theta_{\infty, 1}$};
		\node(a)at(2,0.3){$\Theta_{\infty, 3}$};
		\node(a)at(2.6,-1){$\Theta_{\infty, 5}$};
		\node(a)at(-2.5,1){$\Theta_{0,0}$};
		\node(a)at(-1.5,0){$\Theta_{0,1}$};
		\node(a)at(-2.5,-1){$\Theta_{0,2}$};

		\node(a)at(3.5,0){$\Theta_{\infty, 0}$};

		\end{tikzpicture}
  }
        \subcaption{Type (2)}
        \label{Gradation}
      \end{minipage} \\
   
      \begin{minipage}[t]{0.2\hsize}
        \centering
%        \includegraphics[keepaspectratio, scale=0.8]{figures/canvas_fill.jpg}
%        \subcaption{Fill}
        \label{fill}
      \end{minipage} &
      \begin{minipage}[t]{0.7\hsize}
        \centering
{\label{fig:figC}
		\begin{tikzpicture}

		\draw (-0.1,2.06) parabola bend (-2.5,2.06) (-6.4,1.1);
		\draw (-0.1,2.02) parabola bend (-2,2.02) (-4.9,1.1);
		\draw (-0.1,1.98) parabola bend (-1.5,1.98) (-3.4,1.1);
		\draw (-0.1,1.94) parabola bend (-1,1.94) (-1.9,1.1);
		
		\draw[thin ](-0.1,1.56)--(-0.6,1.56);
		\draw[thin ](-0.1,1.52)--(-0.6,1.52);
		\draw[thin ](-0.1,1.48)--(-0.6,1.48);
		\draw[thin ](-0.1,1.44)--(-0.6,1.44);
		\node(a)at(-0.8,1.5){$\cdots$};
		\draw[thin ](-0.1,1.06)--(-0.6,1.06);
		\draw[thin ](-0.1,1.02)--(-0.6,1.02);
		\draw[thin ](-0.1,0.98)--(-0.6,0.98);
		\draw[thin ](-0.1,0.94)--(-0.6,0.94);
		\node(a)at(-0.8,1){$\cdots$};
		\draw[thin ](-0.1,0.56)--(-0.6,0.56);
		\draw[thin ](-0.1,0.52)--(-0.6,0.52);
		\draw[thin ](-0.1,0.48)--(-0.6,0.48);
		\draw[thin ](-0.1,0.44)--(-0.6,0.44);
		\node(a)at(-0.8,0.5){$\cdots$};
		\draw[thin ](-0.1,0.06)--(-0.6,0.06);
		\draw[thin ](-0.1,0.02)--(-0.6,0.02);
		\draw[thin ](-0.1,-0.02)--(-0.6,-0.02);
		\draw[thin ](-0.1,-0.06)--(-0.6,-0.06);
		\node(a)at(-0.8,0){$\cdots$};
		\draw[thin ](-0.1,-0.56)--(-0.6,-0.56);
		\draw[thin ](-0.1,-0.52)--(-0.6,-0.52);
		\draw[thin ](-0.1,-0.48)--(-0.6,-0.48);
		\draw[thin ](-0.1,-0.44)--(-0.6,-0.44);
		\node(a)at(-0.8,-0.5){$\cdots$};
		\draw[thin ](-0.1,-1.06)--(-0.6,-1.06);
		\draw[thin ](-0.1,-1.02)--(-0.6,-1.02);
		\draw[thin ](-0.1,-0.98)--(-0.6,-0.98);
		\draw[thin ](-0.1,-0.94)--(-0.6,-0.94);
		\node(a)at(-0.8,-1){$\cdots$};
		\draw[thin ](-0.1,-1.56)--(-0.6,-1.56);
		\draw[thin ](-0.1,-1.52)--(-0.6,-1.52);
		\draw[thin ](-0.1,-1.48)--(-0.6,-1.48);
		\draw[thin ](-0.1,-1.44)--(-0.6,-1.44);
		\node(a)at(-0.8,-1.5){$\cdots$};
		\draw[thin ](-0.1,-2.06)--(-0.6,-2.06);
		\draw[thin ](-0.1,-2.02)--(-0.6,-2.02);
		\draw[thin ](-0.1,-1.98)--(-0.6,-1.98);
		\draw[thin ](-0.1,-1.94)--(-0.6,-1.94);
		\node(a)at(-0.8,-2){$\cdots$};

		\draw[thin ](-1.7,-1)--(-2.3,-1);
		\draw[thin ](-1.5,-0.8)--(-1.9,0.9);
		\draw[thin ](-2.1,0.9)--(-2.5,-0.8);
		\draw[thin ](-3.2,-1)--(-3.8,-1);
		\draw[thin ](-3.0,-0.8)--(-3.4,0.9);
		\draw[thin ](-3.6,0.9)--(-4,-0.8);
		\draw[thin ](-4.7,-1)--(-5.3,-1);
		\draw[thin ](-4.5,-0.8)--(-4.9,0.9);
		\draw[thin ](-5.1,0.9)--(-5.5,-0.8);
		\draw[thin ](-6.2,-1)--(-6.8,-1);
		\draw[thin ](-6.0,-0.8)--(-6.4,0.9);
		\draw[thin ](-6.6,0.9)--(-7,-0.8);

		\fill (0,-2) circle (2pt);
		\fill (0,-1.5) circle (2pt);
		\fill (0,-1) circle (2pt);
		\fill (0,-0.5) circle (2pt);
		\fill (0,0) circle (2pt);
		\fill (0,0.5) circle (2pt);
		\fill (0,1) circle (2pt);
		\fill (0,1.5) circle (2pt);
		\fill (0,2) circle (2pt);
		\draw[] (-1.5,-1)circle(2pt);
		\draw[] (-2.5,-1)circle(2pt);
		\draw[] (-3,-1)circle(2pt);
		\draw[] (-4,-1)circle(2pt);
		\draw[] (-4.5,-1)circle(2pt);
		\draw[] (-5.5,-1)circle(2pt);
		\draw[] (-6,-1)circle(2pt);
		\draw[] (-7,-1)circle(2pt);
		\draw[] (-2,1)circle(2pt);
		\draw[] (-3.5,1)circle(2pt);
		\draw[] (-5,1)circle(2pt);
		\draw[] (-6.5,1)circle(2pt);

		\node(a)at(0.3,2){$O$};
		\node(a)at(0.3,1.5){$P$};
		\node(a)at(0.4,1){$2P$};
		\node(a)at(0.3,0.5){$Q$};
		\node(a)at(0.4,0){$2Q$};
		\node(a)at(0.6,-0.5){$P+Q$};
		\node(a)at(0.7,-1){$2P+Q$};
		\node(a)at(0.7,-1.5){$P+2Q$};
		\node(a)at(0.8,-2){$2P+2Q$};
		\node(a)at(-7,1){$\Theta_{0,0}$};
		\node(a)at(-5.6,1){$\Theta_{-1,0}$};
		\node(a)at(-4,1){$\Theta_{1,0}$};
		\node(a)at(-2.5,1){$\Theta_{\infty,0}$};
		
		\node(a)at(-7,-1.3){$\Theta_{0,1}$};
		\node(a)at(-6.2,-1.3){$\Theta_{0,2}$};
		\node(a)at(-5.4,-1.3){$\Theta_{-1,1}$};
		\node(a)at(-4.5,-1.3){$\Theta_{-1,2}$};
		\node(a)at(-3.7,-1.3){$\Theta_{1,1}$};
		\node(a)at(-3.05,-1.3){$\Theta_{1,2}$};
		\node(a)at(-2.2,-1.3){$\Theta_{\infty,1}$};
		\node(a)at(-1.4,-1.3){$\Theta_{\infty,2}$};
		\end{tikzpicture}
  }
  \vspace{-3em}
        \subcaption{Type (3)}
        \label{transform}
      \end{minipage} 
    \end{tabular}
     \caption{Dual graphs of negative rational curves in rational quasi-elliptic surfaces of types (1)--(3)}
    \label{fig:dual(1)-(3)}
  \end{figure}

	\begin{table}[htbp]
\caption{}
  \begin{tabular}{|l|l|l|l|l|} \hline
                              & $\beta=0$ & $\beta=-1$ &  $\beta=1$  &  $\beta=\infty$ \\ \hline 
       Sections adjacent      & $O, 2P+Q,$&$O, Q,$     & $O, P,$     & $O, P+Q,$       \\  
       to $\Theta_{\beta, 0}$ & $P+2Q$    &$2Q$        & $2P$        & $2P+2Q$         \\ \hline   
       Sections adjacent      & $P, Q,$   &$2P, 2P+Q,$ & $2Q, P+2Q,$ & $P, 2Q$         \\   
       to $\Theta_{\beta, 1}$ & $2P+2Q$   &$2P+2Q$     & $2P+2Q$     & $2P+Q$          \\ \hline 
       Sections adjacent      & $2P, 2Q,$ &$P, P+Q,$   & $Q, P+Q$    & $Q, 2P,$        \\   
       to $\Theta_{\beta, 2}$ & $P+Q$     &$P+2Q$      & $2P+Q$      & $P+2Q$          \\ \hline
  \end{tabular}
  \label{tab:type(3)}
\end{table}

\begin{thm}[{\cite[\S 5]{Ito2}}]\label{q-ell}
Suppose $p=2$. Then the following hold.
\begin{enumerate}
    \item[\textup{(1)}] The configurations of reducible fibers of rational quasi-elliptic surfaces and their Mordell-Weil groups are listed in Table \ref{q-ell2}, where we use Kodaira's notation. 
    \item[\textup{(2)}] Rational quasi-elliptic surfaces of each type (a)--(c) and (e) uniquely exist.
    \item[\textup{(3)}] For each rational quasi-elliptic surface of one of the types (a)--(e), sections are disjoint from each other.
    Moreover, the dual graphs of negative rational curves in rational quasi-elliptic surfaces of types (a)--(e) are as in Figure \ref{fig:dual(a)-(e)}, where black nodes (resp.\ white nodes) correspond to $(-1)$-curves (resp.\ $(-2)$-curves).
    \item[\textup{(4)}] For each rational quasi-elliptic surface of type (f), sections are disjoint from each other.
    There is an element $a \in k \setminus \{0\}$ such that the reducible fiber of type $\textup{I}^*_0$ lies over $t=1$ and reducible fibers of type \textup{III} lie over the points $t=0, \infty, \alpha_1, \alpha_2$ of the base curve $\PP^1_k$, where $\alpha_1$ and $\alpha_2$ are two solutions of the equation $t^2+at+1=0$. 
    Moreover, Figure \ref{fig:dual(f)} and Table \ref{typef} describe the dual graph of the configuration of negative rational curves.
    \item[\textup{(5)}] For each rational quasi-elliptic surface of type (g), there are eight pairs of two sections intersecting with each other transversally and not intersecting with any other sections. 
    There are no irreducible components of reducible fibers intersecting with two sections in a pair.
    Figure \ref{fig:dual(g)} describes the above situation.
\end{enumerate}
\end{thm}

\begin{table}[htbp]
\caption{}
  \begin{tabular}{|c|c|c||c|c|c|} \hline
       Type &Reducible fibers      &$\mathrm{MW}(Z)$ & Type & Reducible fibers      &$\mathrm{MW}(Z)$ \\ \hline \hline
\textup{(a)} &$\textup{II}^*$            &$\{1\}$  & \textup{(e)}  &$\textup{I}^*_2, \textup{III}, \textup{III}$       & $(\ZZ/2\ZZ)^2$  \\ \hline   
\textup{(b)} &$\textup{I}^*_4$          &$\ZZ/2\ZZ$  & \textup{(f)}  &$\textup{I}^*_0$ and four $\textup{III}$  & $(\ZZ/2\ZZ)^3$  \\ \hline  
\textup{(c)} &$\textup{III}^*, \textup{III}$     &$\ZZ/2\ZZ$  & \textup{(g)}  &eight $\textup{III}$             & $(\ZZ/2\ZZ)^4$ \\ \hline       
\textup{(d)} &$\textup{I}^*_0, \textup{I}^*_0$   &$(\ZZ/2\ZZ)^2$  &      &   &    \\ \hline 
  \end{tabular}
  \label{q-ell2}
\end{table}

\begin{rem}
\begin{enumerate}
    \item Table 2 of \cite{Ito2} contains misprints.
    By substituting $t=1$ to the equations of $P_2$, $P_3$, $Q_1$, and $R_1$ in the bottom of p.\ 246 of [\textit{ibid}], we see at once that $Q_1$ and $P_2$ in the bottom table should be interchanged with each other.
    We also have to replace $R_3$ by $R_2$.
    \item In Lemma \ref{lem:typeg}, we will clarify that Figure \ref{matrix(g)} is the intersection matrix of negative rational curves in a rational quasi-elliptic surface of type (g).
    \item In Corollary \ref{Itorem}, we will give the parametrizing spaces of the isomorphism classes of rational quasi-elliptic surfaces of type (d), (f), or (g).
\end{enumerate}
\end{rem}

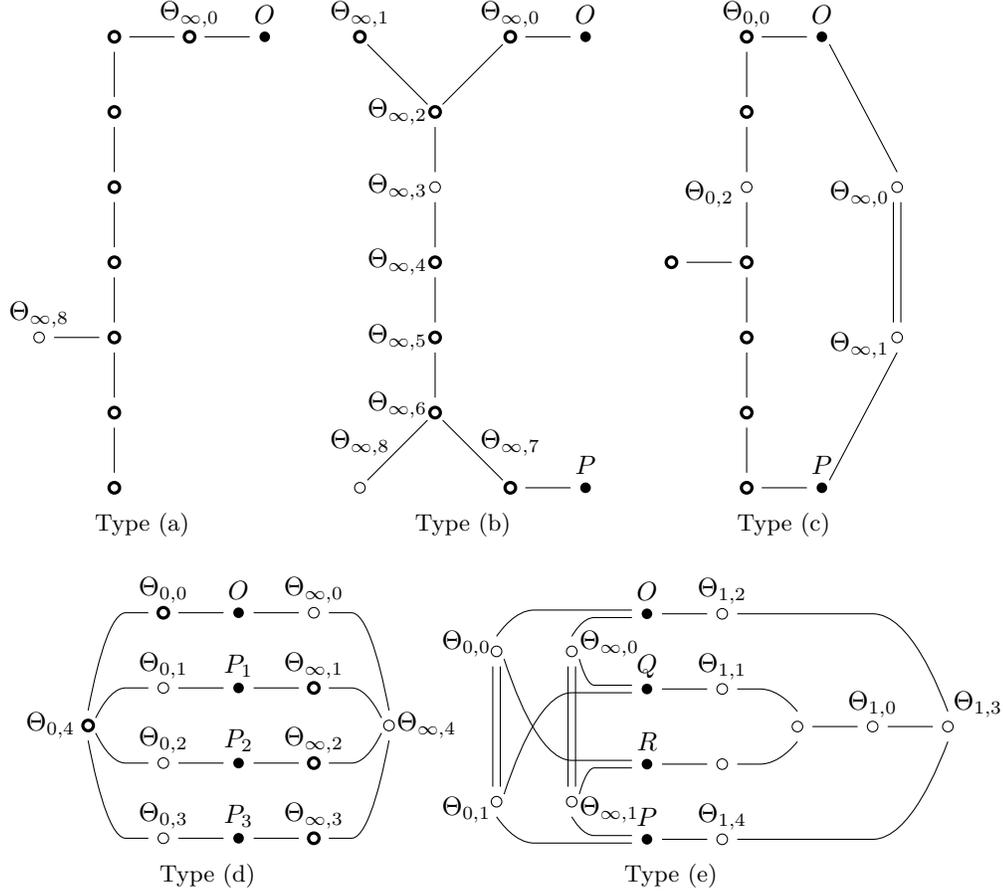
\begin{figure}[htbp]
\captionsetup[subfigure]{labelformat=empty}
\centering
    \begin{subfigure}[b]{0.3\textwidth}
        \centering
		\begin{tikzpicture}

		\draw[thin ](-0.2,-0)--(-0.8,0);
		\draw[thin ](-1.2,-0)--(-1.8,0);
		\draw[thin ](-2,-0.2)--(-2,-0.8);
		\draw[thin ](-2,-1.2)--(-2,-1.8);
		\draw[thin ](-2,-2.2)--(-2,-2.8);
		\draw[thin ](-2,-3.2)--(-2,-3.8);
		\draw[thin ](-2,-4.2)--(-2,-4.8);
		\draw[thin ](-2,-5.2)--(-2,-5.8);
		\draw[thin ](-2.2,-4)--(-2.8,-4);

		\fill (0,0) circle (2pt);
		\draw[very thick ] (-1,0)circle(2pt);
		\draw[very thick ] (-2,0)circle(2pt);
		\draw[very thick ] (-2,-1)circle(2pt);
		\draw[very thick ] (-2,-2)circle(2pt);
		\draw[very thick ] (-2,-3)circle(2pt);
		\draw[very thick ] (-2,-4)circle(2pt);
		\draw[very thick ] (-2,-5)circle(2pt);
		\draw[very thick ] (-2,-6)circle(2pt);
		\draw[] (-3,-4)circle(2pt);

		\node(a)at(0,0.3){$O$};
		\node(a)at(-1,0.3){$\Theta_{\infty, 0}$};
		\node(a)at(-3,-3.7){$\Theta_{\infty, 8}$};

		\end{tikzpicture}
        \caption{Type (a) }
        \label{fig:dual(a)}
    \end{subfigure}
            \quad    
    \begin{subfigure}[b]{0.3\textwidth}
        \centering
		\begin{tikzpicture}

		\draw[thin ](-0.8,3)--(-0.2,3);
		\draw[thin ](-0.8,-3)--(-0.2,-3);
		\draw[thin ](-1.1,2.9)--(-1.9,2.1);
		\draw[thin ](-1.1,-2.9)--(-1.9,-2.1);
		\draw[thin ](-2.1,2.1)--(-2.9,2.9);
		\draw[thin ](-2.1,-2.1)--(-2.9,-2.9);
		\draw[thin ](-2,1.2)--(-2,1.8);
		\draw[thin ](-2,0.2)--(-2,0.8);
		\draw[thin ](-2,-0.2)--(-2,-0.8);
		\draw[thin ](-2,-1.2)--(-2,-1.8);

		\fill (0,-3) circle (2pt);
		\fill (0,3) circle (2pt);
		\draw[very thick ] (-1,3)circle(2pt);
		\draw[very thick ] (-1,-3)circle(2pt);
		\draw[very thick ] (-2,2)circle(2pt);
		\draw[] (-2,1)circle(2pt);
		\draw[very thick ] (-2,0)circle(2pt);
		\draw[very thick ] (-2,-1)circle(2pt);
		\draw[very thick ] (-2,-2)circle(2pt);
		\draw[very thick ] (-3,3)circle(2pt);
		\draw[] (-3,-3)circle(2pt);

		\node(a)at(0,3.3){$O$};
		\node(a)at(0,-2.7){$P$};
		\node(a)at(-1,3.3){$\Theta_{\infty, 0}$};
		\node(a)at(-3,3.3){$\Theta_{\infty, 1}$};
		\node(a)at(-2.5,2){$\Theta_{\infty, 2}$};
		\node(a)at(-2.5,1){$\Theta_{\infty, 3}$};
		\node(a)at(-2.5,0){$\Theta_{\infty, 4}$};
		\node(a)at(-2.5,-1){$\Theta_{\infty, 5}$};
		\node(a)at(-2.5,-1.9){$\Theta_{\infty, 6}$};
		\node(a)at(-1,-2.4){$\Theta_{\infty, 7}$};
		\node(a)at(-3,-2.4){$\Theta_{\infty, 8}$};
		\end{tikzpicture}
        \caption{Type (b) }
        \label{fig:dual(b)}
    \end{subfigure} 
            \quad
    \begin{subfigure}[b]{0.3\textwidth}
    \centering
		\begin{tikzpicture}

		\draw[thin ](-0.8,3)--(-0.2,3);
		\draw[thin ](-0.8,-3)--(-0.2,-3);
		\draw[thin ](-1,2.8)--(-1,2.2);
		\draw[thin ](-1,-2.8)--(-1,-2.2);
		\draw[thin ](-1,1.2)--(-1,1.8);
		\draw[thin ](-1,0.2)--(-1,0.8);
		\draw[thin ](-1,-0.2)--(-1,-0.8);
		\draw[thin ](-1,-1.2)--(-1,-1.8);
		\draw[thin ](-1.8,0)--(-1.2,0);

		\draw[thin ](0.1,2.9)--(1,1.2);
		\draw[thin ](0.1,-2.9)--(1,-1.2);

		\draw[thin ](0.95,0.8)--(0.95,-0.8);
		\draw[thin ](1.05,0.8)--(1.05,-0.8);

		\fill (0,-3) circle (2pt);
		\fill (0,3) circle (2pt);
		\draw[very thick ] (-1,3)circle(2pt);
		\draw[very thick ] (-1,2)circle(2pt);
		\draw[] (-1,1)circle(2pt);
		\draw[very thick ] (-1,0)circle(2pt);
		\draw[very thick ] (-1,-1)circle(2pt);
		\draw[very thick ] (-1,-2)circle(2pt);
		\draw[very thick ] (-1,-3)circle(2pt);
		\draw[very thick ] (-2,0)circle(2pt);
		
		\draw[] (1,1)circle(2pt);
		\draw[] (1,-1)circle(2pt);

		\node(a)at(0,3.3){$O$};
		\node(a)at(0,-2.7){$P$};
		\node(a)at(0.5,0.9){$\Theta_{\infty, 0}$};
		\node(a)at(0.5,-1.1){$\Theta_{\infty, 1}$};
		\node(a)at(-1,3.3){$\Theta_{0, 0}$};
		\node(a)at(-1.5,0.9){$\Theta_{0, 2}$};
		
		\end{tikzpicture}
        \caption{Type (c) }
        \label{fig:dual(c)}
    \end{subfigure} \\
            \vspace{1em}
    \begin{subfigure}[b]{0.4\textwidth}
        \centering
		\begin{tikzpicture}

		\draw[thin ](0.2,-1.5)--(0.8,-1.5);
		\draw[thin ](0.2,-0.5)--(0.8,-0.5);
		\draw[thin ](0.2,0.5)--(0.8,0.5);
		\draw[thin ](0.2,1.5)--(0.8,1.5);
		\draw[thin ](-0.2,-1.5)--(-0.8,-1.5);
		\draw[thin ](-0.2,-0.5)--(-0.8,-0.5);
		\draw[thin ](-0.2,0.5)--(-0.8,0.5);
		\draw[thin ](-0.2,1.5)--(-0.8,1.5);
		\draw (1.2,1.5) parabola bend (1.5,1.5) (2,0.2);
		\draw (1.2,0.5) parabola bend (1.5,0.5) (1.9,0.1);
		\draw (1.2,-0.5) parabola bend (1.5,-0.5) (1.9,-0.1);
		\draw (1.2,-1.5) parabola bend (1.5,-1.5) (2,-0.2);
		\draw (-1.2,1.5) parabola bend (-1.5,1.5) (-2,0.2);
		\draw (-1.2,0.5) parabola bend (-1.5,0.5) (-1.9,0.1);
		\draw (-1.2,-0.5) parabola bend (-1.5,-0.5) (-1.9,-0.1);
		\draw (-1.2,-1.5) parabola bend (-1.5,-1.5) (-2,-0.2);

		\fill (0,-1.5) circle (2pt);
		\fill (0,-0.5) circle (2pt);
		\fill (0,0.5) circle (2pt);
		\fill (0,1.5) circle (2pt);
		\draw[very thick ] (1,-1.5)circle(2pt);
		\draw[very thick ] (1,-0.5)circle(2pt);
		\draw[very thick ] (1,0.5)circle(2pt);
		\draw[] (1,1.5)circle(2pt);
		\draw[] (-1,-1.5)circle(2pt);
		\draw[] (-1,-0.5)circle(2pt);
		\draw[] (-1,0.5)circle(2pt);
		\draw[very thick ] (-1,1.5)circle(2pt);
		\draw[] (2,0)circle(2pt);
		\draw[very thick ] (-2,0)circle(2pt);

		\node(a)at(0,1.8){$O$};
		\node(a)at(0,0.8){$P_1$};
		\node(a)at(0,-0.2){$P_2$};
		\node(a)at(0,-1.2){$P_3$};
		\node(a)at(-1,1.8){$\Theta_{0,0}$};
		\node(a)at(-1,0.8){$\Theta_{0,1}$};
		\node(a)at(-1,-0.2){$\Theta_{0,2}$};
		\node(a)at(-1,-1.2){$\Theta_{0,3}$};
		\node(a)at(-2.5,0){$\Theta_{0,4}$};
		\node(a)at(1,1.8){$\Theta_{\infty,0}$};
		\node(a)at(1,0.8){$\Theta_{\infty,1}$};
		\node(a)at(1,-0.2){$\Theta_{\infty,2}$};
		\node(a)at(1,-1.2){$\Theta_{\infty,3}$};
		\node(a)at(2.5,0){$\Theta_{\infty,4}$};

		\end{tikzpicture}
                \caption{Type (d) }
        \label{fig:dual(d)}
    \end{subfigure}
            \quad    
    \begin{subfigure}[b]{0.5\textwidth}
        \centering
		\begin{tikzpicture}

		\draw[thin ](0.2,-1.5)--(0.8,-1.5);
		\draw[thin ](0.2,-0.5)--(0.8,-0.5);
		\draw[thin ](0.2,0.5)--(0.8,0.5);
		\draw[thin ](0.2,1.5)--(0.8,1.5);
		\draw (1.2,1.5) parabola bend (3,1.5) (4,0.2);
		\draw (1.2,0.5) parabola bend (1.5,0.5) (2,0.2);
		\draw (1.2,-0.5) parabola bend (1.5,-0.5) (2,-0.2);
		\draw (1.2,-1.5) parabola bend (3,-1.5) (4,-0.2);
		\draw[thin ](2.2,0)--(2.8,0);
		\draw[thin ](3.2,0)--(3.8,0);
		\draw[thin ](-0.95,0.8)--(-0.95,-0.8);
		\draw[thin ](-1.05,0.8)--(-1.05,-0.8);
		\draw[thin ](-1.95,0.8)--(-1.95,-0.8);
		\draw[thin ](-2.05,0.8)--(-2.05,-0.8);

		\draw (-0.2,1.45) parabola bend (-0.7,1.45) (-1,1.2);
		\draw (-0.2,-1.45) parabola bend (-0.7,-1.45) (-1,-1.2);
		\draw (-0.2,0.55) parabola bend (-0.7,0.55) (-0.9,0.9);
		\draw (-0.2,-0.55) parabola bend (-0.7,-0.55) (-0.9,-0.9);

		\draw (-0.2,1.55) parabola bend (-1.5,1.55) (-2,1.2);
		\draw (-0.2,-1.55) parabola bend (-1.5,-1.55) (-2,-1.2);
		\draw (-0.2,0.45) parabola bend (-1,0.45) (-1.9,-0.9);
		\draw (-0.2,-0.45) parabola bend (-1,-0.45) (-1.9,0.9);

		\fill (0,-1.5) circle (2pt);
		\fill (0,-0.5) circle (2pt);
		\fill (0,0.5) circle (2pt);
		\fill (0,1.5) circle (2pt);
		\draw[] (1,-1.5)circle(2pt);
		\draw[] (1,-0.5)circle(2pt);
		\draw[] (1,0.5)circle(2pt);
		\draw[] (1,1.5)circle(2pt);
		\draw[] (2,0)circle(2pt);
		\draw[] (3,0)circle(2pt);
		\draw[] (4,0)circle(2pt);
		\draw[] (-1,-1)circle(2pt);
		\draw[] (-1,1)circle(2pt);
		\draw[] (-2,-1)circle(2pt);
		\draw[] (-2,1)circle(2pt);

		\node(a)at(0,1.8){$O$};
		\node(a)at(0,0.8){$Q$};
		\node(a)at(0,-0.2){$R$};
		\node(a)at(0,-1.2){$P$};
		\node(a)at(-2.4,1.1) {$\Theta_{0,0}$};
		\node(a)at(-2.4,-1.1) {$\Theta_{0,1}$};
		\node(a)at(-0.5,1.1) {$\Theta_{\infty,0}$};
		\node(a)at(-0.5,-1.1) {$\Theta_{\infty,1}$};
		\node(a)at(3,0.3) {$\Theta_{1,0}$};
		\node(a)at(1,0.8) {$\Theta_{1,1}$};
		\node(a)at(1,1.8) {$\Theta_{1,2}$};
		\node(a)at(4.4,0.3) {$\Theta_{1,3}$};
		\node(a)at(1,-1.2) {$\Theta_{1,4}$};

		\end{tikzpicture}
                \caption{Type (e) }
        \label{fig:dual(e)}
    \end{subfigure} 
\caption{Dual graphs of negative rational curves in rational quasi-elliptic surfaces of types (a)--(e)}
\label{fig:dual(a)-(e)}
\end{figure}

\begin{figure}[htbp] 
		\centering
		\begin{tikzpicture}
	
		\draw[thin ](1.2,0)--(1.8,0);
		\draw[thin ](2.2,0)--(2.8,0);
		\draw[thin ](2,0.2)--(2,0.8);
		\draw[thin ](2,-0.2)--(2,-0.8);
		\draw[thin ](-0.95,0.8)--(-0.95,-0.8);
		\draw[thin ](-1.05,0.8)--(-1.05,-0.8);
		\draw[thin ](-1.95,0.8)--(-1.95,-0.8);
		\draw[thin ](-2.05,0.8)--(-2.05,-0.8);
		\draw[thin ](-2.95,0.8)--(-2.95,-0.8);
		\draw[thin ](-3.05,0.8)--(-3.05,-0.8);
		\draw[thin ](-3.95,0.8)--(-3.95,-0.8);
		\draw[thin ](-4.05,0.8)--(-4.05,-0.8);
		
		\draw[thin ] (0.9,0.1)--(0.1,0.9);
		\draw[thin ] (0.9,-0.1)--(0.1,-0.9);

		\draw (0,1.1) parabola bend (-2,2.2) (-4,1.1);
		\draw (-0.03,1.07) parabola bend (-1.5,1.9) (-2.97,1.07);
		\draw (-0.07,1.03) parabola bend (-1,1.6) (-1.93,1.03);
		\draw (-0.1,1) parabola bend (-0.5,1.1) (-0.9,1);
		\draw (0,-1.1) parabola bend (-2,-2.2) (-4,-1.1);
		\draw (-0.03,-1.07) parabola bend (-1.5,-1.9) (-2.97,-1.07);
		\draw (-0.07,-1.03) parabola bend (-1,-1.6) (-1.93,-1.03);
		\draw (-0.1,-1) parabola bend (-0.5,-1.1) (-0.9,-1);
		
		\draw[thin ](1.9,1.1)--(1.6,1.4);
		\draw[thin ](2.1,1.1)--(2.4,1.4);
		\draw[thin ](1.9,-1.1)--(1.6,-1.4);
		\draw[thin ](2.1,-1.1)--(2.4,-1.4);
		\draw[thin ](3.1,0.1)--(3.4,0.4);
		\draw[thin ](3.1,-0.1)--(3.4,-0.4);
		
		\draw[thin ](1.56,1.6)--(1.56,1.9);
		\draw[thin ](1.52,1.6)--(1.52,1.9);
		\draw[thin ](1.48,1.6)--(1.48,1.9);
		\draw[thin ](1.44,1.6)--(1.44,1.9);
		\node(a)at(1.5,2.3){$\vdots$};
		\draw[thin ](2.56,1.6)--(2.56,1.9);
		\draw[thin ](2.52,1.6)--(2.52,1.9);
		\draw[thin ](2.48,1.6)--(2.48,1.9);
		\draw[thin ](2.44,1.6)--(2.44,1.9);
		\node(a)at(2.5,2.3){$\vdots$};
		\draw[thin ](1.56,-1.6)--(1.56,-1.9);
		\draw[thin ](1.52,-1.6)--(1.52,-1.9);
		\draw[thin ](1.48,-1.6)--(1.48,-1.9);
		\draw[thin ](1.44,-1.6)--(1.44,-1.9);
		\node(a)at(1.5,-2.1){$\vdots$};
		\draw[thin ](2.56,-1.6)--(2.56,-1.9);
		\draw[thin ](2.52,-1.6)--(2.52,-1.9);
		\draw[thin ](2.48,-1.6)--(2.48,-1.9);
		\draw[thin ](2.44,-1.6)--(2.44,-1.9);
		\node(a)at(2.5,-2.1){$\vdots$};
		\draw[thin ](3.6,0.56)--(3.9,0.56);
		\draw[thin ](3.6,0.52)--(3.9,0.52);
		\draw[thin ](3.6,0.48)--(3.9,0.48);
		\draw[thin ](3.6,0.44)--(3.9,0.44);
		\node(a)at(4.2,0.5){$\cdots$};
		\draw[thin ](3.6,-0.56)--(3.9,-0.56);
		\draw[thin ](3.6,-0.52)--(3.9,-0.52);
		\draw[thin ](3.6,-0.48)--(3.9,-0.48);
		\draw[thin ](3.6,-0.44)--(3.9,-0.44);
		\node(a)at(4.2,-0.5){$\cdots$};

		\fill (0,-1) circle (2pt);
		\fill (0,1) circle (2pt);
		\fill (3.5,-0.5) circle (2pt);
		\fill (3.5,0.5) circle (2pt);
		\fill (1.5,1.5) circle (2pt);
		\fill (2.5,1.5) circle (2pt);
		\fill (1.5,-1.5) circle (2pt);
		\fill (2.5,-1.5) circle (2pt);
		\draw[] (2,0)circle(2pt);
		\draw[] (1,0)circle(2pt);
		\draw[] (3,0)circle(2pt);
		\draw[] (2,1)circle(2pt);
		\draw[] (2,-1)circle(2pt);
		\draw[] (-1,-1)circle(2pt);
		\draw[] (-1,1)circle(2pt);
		\draw[] (-2,-1)circle(2pt);
		\draw[] (-2,1)circle(2pt);
		\draw[] (-3,-1)circle(2pt);
		\draw[] (-3,1)circle(2pt);
		\draw[] (-4,-1)circle(2pt);
		\draw[] (-4,1)circle(2pt);

        \begin{scriptsize}

		\node(a)at(-4.3,1.2){$\Theta_{0,0}$};
		\node(a)at(-4.3,-1.2){$\Theta_{0,1}$};
		\node(a)at(-3.3,1.2){$\Theta_{\infty,0}$};
		\node(a)at(-3.3,-1.2){$\Theta_{\infty,1}$};
		\node(a)at(-2.3,1.2){$\Theta_{\alpha_1,0}$};
		\node(a)at(-2.3,-1.2){$\Theta_{\alpha_1,1}$};
		\node(a)at(-1.3,1.2){$\Theta_{\alpha_2,0}$};
		\node(a)at(-1.3,-1.2){$\Theta_{\alpha_2,1}$};
		
		\node(a)at(0.3,1.2){$O$};
		\node(a)at(0.3,-1.2){$P_1$};
		\node(a)at(1.2,-1.5){$Q_2$};
		\node(a)at(2.8,-1.5){$R_2$};
		\node(a)at(3.5,-0.8){$Q_1$};
		\node(a)at(3.5,0.8){$R_1$};
		\node(a)at(1.2,1.5){$P_2$};
		\node(a)at(2.8,1.5){$P_3$};
		
		\node(a)at(1.3,0.2){$\Theta_{1, 0}$};
		\node(a)at(2.4,-0.9){$\Theta_{1, 1}$};
		\node(a)at(3.5,0.0){$\Theta_{1, 2}$};
		\node(a)at(2.4,0.9){$\Theta_{1, 3}$};
		\node(a)at(2.4,0.2){$\Theta_{1, 4}$};

		\end{scriptsize}	
		\end{tikzpicture}
		\caption{Dual graph of negative rational curves in a rational quasi-elliptic surface of type (f)}
		\label{fig:dual(f)}
	\end{figure}
	
\begin{table}[htbp]
\caption{}
  \begin{tabular}{|l|c|c|c|c|} \hline
                            & $\beta=0$  & $\beta=\infty$ &  $\beta=\alpha_1$ &  $\beta=\alpha_2$   \\ \hline 
       Section intersecting      & $O, R_2$   &$O, Q_2$    & $O, Q_2$   & $O, R_2$  \\  
       with $\Theta_{\beta, 0}$  & $R_1, P_3$ &$P_3, Q_1$  & $R_1, P_2$ & $P_2, Q_1$  \\ \hline   
       Section intersecting      & $P_1, Q_2$ &$P_1, R_2$  & $P_1, R_2$ & $P_1, Q_2$  \\   
       with $\Theta_{\beta, 1}$  & $P_2, Q_1$ &$R_1, P_2$  & $P_3, Q_1$ & $R_1, P_3$  \\ \hline                  
  \end{tabular}
  
  \vspace{1em}
  
    \begin{tabular}{|l|c|} \hline
                                                       & $\gamma=1$  \\ \hline 
       Section intersecting with $\Theta_{\gamma, 0}$  & $O, P_1$   \\ \hline
       Section intersecting with $\Theta_{\gamma, 1}$  & $Q_2, R_2$   \\ \hline   
       Section intersecting with $\Theta_{\gamma, 2}$  & $Q_1, R_1$   \\ \hline  
       Section intersecting with $\Theta_{\gamma, 3}$  & $P_2, P_3$   \\ \hline                  
  \end{tabular}
  \label{typef}
\end{table}

\begin{figure}[htbp] 
		\centering
		\begin{tikzpicture}[xscale=0.75]

		\draw[thin ](-0.95,0.8)--(-0.95,-0.8);
		\draw[thin ](-1.05,0.8)--(-1.05,-0.8);
		\draw[thin ](-1.95,0.8)--(-1.95,-0.8);
		\draw[thin ](-2.05,0.8)--(-2.05,-0.8);
		\draw[thin ](-2.95,0.8)--(-2.95,-0.8);
		\draw[thin ](-3.05,0.8)--(-3.05,-0.8);
		\draw[thin ](-3.95,0.8)--(-3.95,-0.8);
		\draw[thin ](-4.05,0.8)--(-4.05,-0.8);
		\draw[thin ](-4.95,0.8)--(-4.95,-0.8);
		\draw[thin ](-5.05,0.8)--(-5.05,-0.8);
		\draw[thin ](-5.95,0.8)--(-5.95,-0.8);
		\draw[thin ](-6.05,0.8)--(-6.05,-0.8);
		\draw[thin ](-6.95,0.8)--(-6.95,-0.8);
		\draw[thin ](-7.05,0.8)--(-7.05,-0.8);
		\draw[thin ](-7.95,0.8)--(-7.95,-0.8);
		\draw[thin ](-8.05,0.8)--(-8.05,-0.8);

		\draw[thin ] (0,0.8)--(0,-0.8);
		\draw[thin ] (1,0.8)--(1,-0.8);
		\draw[thin ] (2,0.8)--(2,-0.8);
		\draw[thin ] (3,0.8)--(3,-0.8);
		\draw[thin ] (4,0.8)--(4,-0.8);
		\draw[thin ] (5,0.8)--(5,-0.8);
		\draw[thin ] (6,0.8)--(6,-0.8);
		\draw[thin ] (7,0.8)--(7,-0.8);

		\draw (0,1.1) parabola bend (-4,3.2) (-8,1.1);
		\draw (-0.01,1.09) parabola bend (-3.5,2.9) (-6.99,1.09);
		\draw (-0.02,1.08) parabola bend (-3,2.6) (-5.98,1.08);
		\draw (-0.04,1.06) parabola bend (-2.5,2.3) (-4.96,1.06);
		\draw (-0.06,1.04) parabola bend (-2,2) (-3.94,1.04);
		\draw (-0.08,1.02) parabola bend (-1.5,1.7) (-2.92,1.02);
		\draw (-0.09,1.01) parabola bend (-1,1.4) (-1.91,1.01);
		\draw (-0.1,1) parabola bend (-0.5,1.1) (-0.9,1);

		\draw (0,-1.1) parabola bend (-4,-3.2) (-8,-1.1);
		\draw (-0.01,-1.09) parabola bend (-3.5,-2.9) (-6.99,-1.09);
		\draw (-0.02,-1.08) parabola bend (-3,-2.6) (-5.98,-1.08);
		\draw (-0.04,-1.06) parabola bend (-2.5,-2.3) (-4.96,-1.06);
		\draw (-0.06,-1.04) parabola bend (-2,-2) (-3.94,-1.04);
		\draw (-0.08,-1.02) parabola bend (-1.5,-1.7) (-2.92,-1.02);
		\draw (-0.09,-1.01) parabola bend (-1,-1.4) (-1.91,-1.01);
		\draw (-0.1,-1) parabola bend (-0.5,-1.1) (-0.9,-1);

		\draw[thin ](1.07,1.1)--(1.28,1.6);
		\draw[thin ](1.05,1.1)--(1.20,1.6);
		\draw[thin ](1.03,1.1)--(1.12,1.6);
		\draw[thin ](1.01,1.1)--(1.04,1.6);
		\draw[thin ](0.99,1.1)--(0.96,1.6);
		\draw[thin ](0.97,1.1)--(0.88,1.6);
		\draw[thin ](0.95,1.1)--(0.80,1.6);
		\draw[thin ](0.93,1.1)--(0.72,1.6);
		\node(a)at(1,2){$\vdots$};
		\draw[thin ](1.07,-1.1)--(1.28,-1.6);
		\draw[thin ](1.05,-1.1)--(1.20,-1.6);
		\draw[thin ](1.03,-1.1)--(1.12,-1.6);
		\draw[thin ](1.01,-1.1)--(1.04,-1.6);
		\draw[thin ](0.99,-1.1)--(0.96,-1.6);
		\draw[thin ](0.97,-1.1)--(0.88,-1.6);
		\draw[thin ](0.95,-1.1)--(0.80,-1.6);
		\draw[thin ](0.93,-1.1)--(0.72,-1.6);
		\node(a)at(1,-1.8){$\vdots$};
		\draw[thin ](2.07,1.1)--(2.28,1.6);
		\draw[thin ](2.05,1.1)--(2.20,1.6);
		\draw[thin ](2.03,1.1)--(2.12,1.6);
		\draw[thin ](2.01,1.1)--(2.04,1.6);
		\draw[thin ](1.99,1.1)--(1.96,1.6);
		\draw[thin ](1.97,1.1)--(1.88,1.6);
		\draw[thin ](1.95,1.1)--(1.80,1.6);
		\draw[thin ](1.93,1.1)--(1.72,1.6);
		\node(a)at(2,2){$\vdots$};
		\draw[thin ](2.07,-1.1)--(2.28,-1.6);
		\draw[thin ](2.05,-1.1)--(2.20,-1.6);
		\draw[thin ](2.03,-1.1)--(2.12,-1.6);
		\draw[thin ](2.01,-1.1)--(2.04,-1.6);
		\draw[thin ](1.99,-1.1)--(1.96,-1.6);
		\draw[thin ](1.97,-1.1)--(1.88,-1.6);
		\draw[thin ](1.95,-1.1)--(1.80,-1.6);
		\draw[thin ](1.93,-1.1)--(1.72,-1.6);
		\node(a)at(2,-1.8){$\vdots$};
		\draw[thin ](3.07,1.1)--(3.28,1.6);
		\draw[thin ](3.05,1.1)--(3.20,1.6);
		\draw[thin ](3.03,1.1)--(3.12,1.6);
		\draw[thin ](3.01,1.1)--(3.04,1.6);
		\draw[thin ](2.99,1.1)--(2.96,1.6);
		\draw[thin ](2.97,1.1)--(2.88,1.6);
		\draw[thin ](2.95,1.1)--(2.80,1.6);
		\draw[thin ](2.93,1.1)--(2.72,1.6);
		\node(a)at(3,2){$\vdots$};
		\draw[thin ](3.07,-1.1)--(3.28,-1.6);
		\draw[thin ](3.05,-1.1)--(3.20,-1.6);
		\draw[thin ](3.03,-1.1)--(3.12,-1.6);
		\draw[thin ](3.01,-1.1)--(3.04,-1.6);
		\draw[thin ](2.99,-1.1)--(2.96,-1.6);
		\draw[thin ](2.97,-1.1)--(2.88,-1.6);
		\draw[thin ](2.95,-1.1)--(2.80,-1.6);
		\draw[thin ](2.93,-1.1)--(2.72,-1.6);
		\node(a)at(3,-1.8){$\vdots$};
		\draw[thin ](4.07,1.1)--(4.28,1.6);
		\draw[thin ](4.05,1.1)--(4.20,1.6);
		\draw[thin ](4.03,1.1)--(4.12,1.6);
		\draw[thin ](4.01,1.1)--(4.04,1.6);
		\draw[thin ](3.99,1.1)--(3.96,1.6);
		\draw[thin ](3.97,1.1)--(3.88,1.6);
		\draw[thin ](3.95,1.1)--(3.80,1.6);
		\draw[thin ](3.93,1.1)--(3.72,1.6);
		\node(a)at(4,2){$\vdots$};
		\draw[thin ](4.07,-1.1)--(4.28,-1.6);
		\draw[thin ](4.05,-1.1)--(4.20,-1.6);
		\draw[thin ](4.03,-1.1)--(4.12,-1.6);
		\draw[thin ](4.01,-1.1)--(4.04,-1.6);
		\draw[thin ](3.99,-1.1)--(3.96,-1.6);
		\draw[thin ](3.97,-1.1)--(3.88,-1.6);
		\draw[thin ](3.95,-1.1)--(3.80,-1.6);
		\draw[thin ](3.93,-1.1)--(3.72,-1.6);
		\node(a)at(4,-1.8){$\vdots$};
		\draw[thin ](5.07,1.1)--(5.28,1.6);
		\draw[thin ](5.05,1.1)--(5.20,1.6);
		\draw[thin ](5.03,1.1)--(5.12,1.6);
		\draw[thin ](5.01,1.1)--(5.04,1.6);
		\draw[thin ](4.99,1.1)--(4.96,1.6);
		\draw[thin ](4.97,1.1)--(4.88,1.6);
		\draw[thin ](4.95,1.1)--(4.80,1.6);
		\draw[thin ](4.93,1.1)--(4.72,1.6);
		\node(a)at(5,2){$\vdots$};
		\draw[thin ](5.07,-1.1)--(5.28,-1.6);
		\draw[thin ](5.05,-1.1)--(5.20,-1.6);
		\draw[thin ](5.03,-1.1)--(5.12,-1.6);
		\draw[thin ](5.01,-1.1)--(5.04,-1.6);
		\draw[thin ](4.99,-1.1)--(4.96,-1.6);
		\draw[thin ](4.97,-1.1)--(4.88,-1.6);
		\draw[thin ](4.95,-1.1)--(4.80,-1.6);
		\draw[thin ](4.93,-1.1)--(4.72,-1.6);
		\node(a)at(5,-1.8){$\vdots$};
		\draw[thin ](6.07,1.1)--(6.28,1.6);
		\draw[thin ](6.05,1.1)--(6.20,1.6);
		\draw[thin ](6.03,1.1)--(6.12,1.6);
		\draw[thin ](6.01,1.1)--(6.04,1.6);
		\draw[thin ](5.99,1.1)--(5.96,1.6);
		\draw[thin ](5.97,1.1)--(5.88,1.6);
		\draw[thin ](5.95,1.1)--(5.80,1.6);
		\draw[thin ](5.93,1.1)--(5.72,1.6);
		\node(a)at(6,2){$\vdots$};
		\draw[thin ](6.07,-1.1)--(6.28,-1.6);
		\draw[thin ](6.05,-1.1)--(6.20,-1.6);
		\draw[thin ](6.03,-1.1)--(6.12,-1.6);
		\draw[thin ](6.01,-1.1)--(6.04,-1.6);
		\draw[thin ](5.99,-1.1)--(5.96,-1.6);
		\draw[thin ](5.97,-1.1)--(5.88,-1.6);
		\draw[thin ](5.95,-1.1)--(5.80,-1.6);
		\draw[thin ](5.93,-1.1)--(5.72,-1.6);
		\node(a)at(6,-1.8){$\vdots$};
		\draw[thin ](7.07,1.1)--(7.28,1.6);
		\draw[thin ](7.05,1.1)--(7.20,1.6);
		\draw[thin ](7.03,1.1)--(7.12,1.6);
		\draw[thin ](7.01,1.1)--(7.04,1.6);
		\draw[thin ](6.99,1.1)--(6.96,1.6);
		\draw[thin ](6.97,1.1)--(6.88,1.6);
		\draw[thin ](6.95,1.1)--(6.80,1.6);
		\draw[thin ](6.93,1.1)--(6.72,1.6);
		\node(a)at(7,2){$\vdots$};
		\draw[thin ](7.07,-1.1)--(7.28,-1.6);
		\draw[thin ](7.05,-1.1)--(7.20,-1.6);
		\draw[thin ](7.03,-1.1)--(7.12,-1.6);
		\draw[thin ](7.01,-1.1)--(7.04,-1.6);
		\draw[thin ](6.99,-1.1)--(6.96,-1.6);
		\draw[thin ](6.97,-1.1)--(6.88,-1.6);
		\draw[thin ](6.95,-1.1)--(6.80,-1.6);
		\draw[thin ](6.93,-1.1)--(6.72,-1.6);
		\node(a)at(7,-1.8){$\vdots$};

		\fill (7,-1) circle (2pt);
		\fill (7,1) circle (2pt);
		\fill (6,-1) circle (2pt);
		\fill (6,1) circle (2pt);
		\fill (5,-1) circle (2pt);
		\fill (5,1) circle (2pt);
		\fill (4,-1) circle (2pt);
		\fill (4,1) circle (2pt);
		\fill (3,-1) circle (2pt);
		\fill (3,1) circle (2pt);
		\fill (2,-1) circle (2pt);
		\fill (2,1) circle (2pt);
		\fill (1,-1) circle (2pt);
		\fill (1,1) circle (2pt);
		\fill (0,-1) circle (2pt);
		\fill (0,1) circle (2pt);
		\draw[] (-1,-1)circle(2pt);
		\draw[] (-1,1)circle(2pt);
		\draw[] (-2,-1)circle(2pt);
		\draw[] (-2,1)circle(2pt);
		\draw[] (-3,-1)circle(2pt);
		\draw[] (-3,1)circle(2pt);
		\draw[] (-4,-1)circle(2pt);
		\draw[] (-4,1)circle(2pt);
		\draw[] (-5,-1)circle(2pt);
		\draw[] (-5,1)circle(2pt);
		\draw[] (-6,-1)circle(2pt);
		\draw[] (-6,1)circle(2pt);
		\draw[] (-7,-1)circle(2pt);
		\draw[] (-7,1)circle(2pt);
		\draw[] (-8,-1)circle(2pt);
		\draw[] (-8,1)circle(2pt);

        \begin{tiny}

		\node(a)at(0.4,0.9){$A_{0,1}$};
		\node(a)at(0.4,-0.9){$A_{0,2}$};
		\node(a)at(0.4,-1.15){\rotatebox{90}{$\;\;=\;\;$}};
		\node(a)at(0.4,-1.4){$O$};
		\node(a)at(1.4,0.9){$A_{1,1}$};
		\node(a)at(1.4,-0.9){$A_{1,2}$};
		\node(a)at(2.4,0.9){$A_{2,1}$};
		\node(a)at(2.4,-0.9){$A_{2,2}$};
		\node(a)at(3.4,0.9){$A_{3,1}$};
		\node(a)at(3.4,-0.9){$A_{3,2}$};
		\node(a)at(4.4,0.9){$A_{4,1}$};
		\node(a)at(4.4,-0.9){$A_{4,2}$};
		\node(a)at(5.4,0.9){$A_{5,1}$};
		\node(a)at(5.4,-0.9){$A_{5,2}$};
		\node(a)at(6.4,0.9){$A_{6,1}$};
		\node(a)at(6.4,-0.9){$A_{6,2}$};
		\node(a)at(7.4,0.9){$A_{7,1}$};
		\node(a)at(7.4,-0.9){$A_{7,2}$};

		\node(a)at(-1.4,0.9){$\Theta_{0,1}$};
		\node(a)at(-1.4,-0.9){$\Theta_{0,2}$};
		\node(a)at(-2.4,0.9){$\Theta_{1,1}$};
		\node(a)at(-2.4,-0.9){$\Theta_{1,2}$};
		\node(a)at(-3.4,0.9){$\Theta_{2,1}$};
		\node(a)at(-3.4,-0.9){$\Theta_{2,2}$};
		\node(a)at(-4.4,0.9){$\Theta_{3,1}$};
		\node(a)at(-4.4,-0.9){$\Theta_{3,2}$};
		\node(a)at(-5.4,0.9){$\Theta_{4,1}$};
		\node(a)at(-5.4,-0.9){$\Theta_{4,2}$};
		\node(a)at(-6.4,0.9){$\Theta_{5,1}$};
		\node(a)at(-6.4,-0.9){$\Theta_{5,2}$};
		\node(a)at(-7.4,0.9){$\Theta_{6,1}$};
		\node(a)at(-7.4,-0.9){$\Theta_{6,2}$};
		\node(a)at(-8.4,0.9){$\Theta_{7,1}$};
		\node(a)at(-8.4,-0.9){$\Theta_{7,2}$};
		
        \end{tiny}

		\end{tikzpicture}
		\caption{Dual graph of negative rational curves in a rational quasi-elliptic surface of type (g)}
		\label{fig:dual(g)}
	\end{figure}
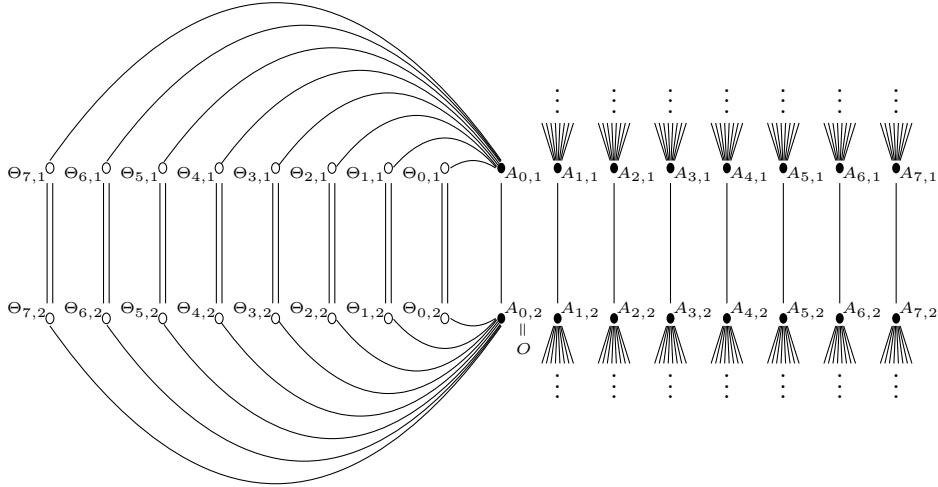

By \cite[Theorem 3.1]{LPS}, each cuspidal cubic curve in $\PP^2_{k, [x:y:z]}$ with an inflexion point is projectively equivalent to $C=\{x^3+y^2z=0\}$.
Moreover, since the automorphism $[x:y:z] \mapsto [ax:y:a^3z]$ of $\PP^2_k$ with $a \in k^*$ fixes $C$, the pair of $C$ and a point $p \in C$ is projectively equivalent to the pair of $C$ and $[1:1:-1]$ unless $p=[0:0:1]$ or $[0:1:0]$.
From these facts, we can interpret \cite[Example 3.8]{Ito1} and \cite[Remark 4]{Ito2} as follows.

\begin{lem}[{\cite[Example 3.8]{Ito1}, \cite[Remark 4]{Ito2}}]\label{lem:Ito2Rem4}
Let $Z$ be a quasi-elliptic surface of type one of (1)--(3) in characteristic three or one of (a)--(d) in characteristic two.
When $Z$ is of type (1), (a), or (b), we choose a general fiber $F$ in addition.
Then, contracting all curves corresponding to bold white node or black node in Figure \ref{fig:dual(1)-(3)} and Types (a)--(d) of Figure \ref{fig:dual(a)-(e)},
we obtain a morphism $h \colon Z \to \PP^2_k$.
Moreover, there are coordinates $[x:y:z]$ of $\PP^2_k$ such that the images of $F$ and negative rational curves by $h$ are written as follows.

If $Z$ is of type (1), then 
    \begin{align*}
    &h(F) = \{x^3+y^2z=0\}, 
    &&h(\Theta_{\infty, 8}) = \{z=0\}, 
    &&h(O)=[0:1:0].
    \end{align*}
    
If $Z$ is of type (2), then 
    \begin{align*}
    &h(\Theta_{\infty, 0}) = \{x=0\}, 
    &&h(\Theta_{0,0})=\{y=0\}, 
    &&h(\Theta_{0,1})=\{y=z\}, 
    &&h(\Theta_{0,2})=\{z=0\}, \\
    &h(O)=[0:0:1], 
    &&h(P)=[0:1:1], 
    &&h(2P)=[0:1:0].
    \end{align*}

If $Z$ is of type (3), then 
\begin{align*}
    &h(\Theta_{0,0})=\{x=0\}, 
    &&h(\Theta_{0,1})=\{x=z\}, 
    &&h(\Theta_{0,2})=\{x=-z\},  \\    
    &h(\Theta_{-1,0})=\{y=0\}, 
    &&h(\Theta_{-1,1})=\{y=z\}, 
    &&h(\Theta_{-1,2})=\{y=-z\}, \\ 
    &h(\Theta_{1,0})=\{x+y=0\}, 
    &&h(\Theta_{1,1})=\{x+y=-z\}, 
    &&h(\Theta_{1,2})=\{x+y=z\}, \\
    &h(\Theta_{\infty,0})=\{x-y=0\}, 
    &&h(\Theta_{\infty,1})=\{x-y=-z\}, 
    &&h(\Theta_{\infty,2})=\{x-y=z\},\\ 
    &h(O)=[0:0:1], 
    &&h(P)=[1:-1:1],
    &&h(2P)=[-1:1:1],\\
    &h(Q)=[1:0:1],
    &&h(2Q)=[-1:0:1],
    &&h(P+Q)=[-1:-1:1],\\
    &h(2P+Q)=[0:1:1],
    &&h(P+2Q)=[0:-1:1], 
    &&h(2P+2Q)=[1:1:1].
\end{align*}

If $Z$ is of type (a), then
\begin{align*}
&h(F) = \{x^3+y^2z=0\}, 
&&h(\Theta_{\infty, 8}) = \{z=0\}, 
&&h(O)=[0:1:0].
\end{align*}

If $Z$ is of type (b), then 
\begin{align*}
&h(F) = \{x^3+y^2z=0\}, 
&&h(\Theta_{\infty, 8})=\{z=0\}, 
&&h(\Theta_{\infty, 3})=\{x+z=0\}, \\
&h(P)=[0:1:0], 
&&h(O)=[1:1:1].
\end{align*}

If $Z$ is of type (c), then
\begin{align*}
&h(\Theta_{0,2})=\{z=0\}, 
&&h(\Theta_{\infty, 0})=\{x=0\}, 
&&h(\Theta_{\infty, 1})=\{xz+y^2=0\}, \\
&h(O)=[0:1:0], 
&&h(P)=[1:0:0].
\end{align*}

If $Z$ is of type (d), then
\begin{align*}
&h(\Theta_{\infty, 4}) = \{x=0\}, 
&&h(\Theta_{0, 1})=\{y=0\}, 
&&h(\Theta_{0, 2})=\{y+z=0\}, 
&&h(\Theta_{0, 3})=\{z=0\}, \\
&h(\Theta_{0,4})=[1:0:0], 
&&h(\Theta_{\infty, 1})=[0:0:1], 
&&h(\Theta_{\infty, 2})=[0:1:1], 
&&h(\Theta_{\infty, 3})=[0:1:0].
\end{align*}
\end{lem}

Rational quasi-elliptic surfaces are naturally endowed with the action of the Mordell-Weil groups.
The next lemma shows that these surfaces may have other automorphisms.

\begin{lem}\label{lem:type(d)auto}
A rational quasi-elliptic surface $Z$ of type (d) has an involution which sends $\Theta_{0, i}$ in Type (d) of Figure \ref{fig:dual(a)-(e)} to $\Theta_{\infty, i}$ for $0 \leq i \leq 4$.
\end{lem}

\begin{proof}
Let $\phi \colon Z \to \PP^1_k \times \PP^1_k$ be the contraction of $O$, $P_1, P_2, P_3$, $\Theta_{0, 0}$, $\Theta_{0, 1}$,$\Theta_{\infty, 2}$, and $\Theta_{\infty, 3}$.
Then we can choose coordinates $([x:y],[s:t])$ of $\PP^1_k \times \PP^1_k$ such that $\phi(\Theta_{0,4})=\{x=0\}$ and $\phi(\Theta_{\infty,4})=\{y=0\}$.
Hence the involution $([x:y],[s:t]) \mapsto ([y:x],[s:t])$ induces the desired involution on $Z$.
\end{proof}

The next lemma clarifies the whole configuration of negative rational curves in a quasi-elliptic surface of type (g). 

\begin{lem}\label{lem:typeg}
Figure \ref{matrix(g)} is the intersection matrix of negative rational curves on a rational quasi-elliptic surface of type (g).
\end{lem}

\begin{figure}[htbp]
    \centering
\begin{align*}
{\fontsize{7pt}{4}\selectfont
    \left(\!\!\!\!
    \vcenter{
        \xymatrix@=-1.8ex{
&&&&&&&&&\ar@{-}[ddddddddddddddddddddddddddddddddddddd]&&&&&&&&& \ar@{-}[ddddddddddddddddddddddddddddddddddddd] &&&&&&&&& \ar@{-}[ddddddddddddddddddddddddddddddddddddd] &&&&&&&&\\
&&&&&&&&&&&&&&&&&&&&&&&&&&&&&&&&&\\
&-1&0 &0 &0 &0 &0 &0 &0 &&1 &0 &0 &0 &0 &0 &0 &0 &&1 &1 &1 &1 &1 &1 &1 &1 &&0 &0 &0 &0 &0 &0 &0 &0 &\\
&  &-1&0 &0 &0 &0 &0 &0 &&0 &1 &0 &0 &0 &0 &0 &0 &&1 &1 &1 &1 &0 &0 &0 &0 &&0 &0 &0 &0 &1 &1 &1 &1 &\\
&  &  &-1&0 &0 &0 &0 &0 &&0 &0 &1 &0 &0 &0 &0 &0 &&1 &1 &0 &0 &1 &1 &0 &0 &&0 &0 &1 &1 &0 &0 &1 &1 &\\
&  &  &  &-1&0 &0 &0 &0 &&0 &0 &0 &1 &0 &0 &0 &0 &&1 &0 &1 &0 &1 &0 &1 &0 &&0 &1 &0 &1 &0 &1 &0 &1 &\\
&  &  &  &  &-1&0 &0 &0 &&0 &0 &0 &0 &1 &0 &0 &0 &&1 &1 &0 &0 &0 &0 &1 &1 &&0 &0 &1 &1 &1 &1 &0 &0 &\\
&  &  &  &  &  &-1&0 &0 &&0 &0 &0 &0 &0 &1 &0 &0 &&1 &0 &0 &1 &1 &0 &0 &1 &&0 &1 &1 &0 &0 &1 &1 &0 &\\
&  &  &  &  &  &  &-1&0 &&0 &0 &0 &0 &0 &0 &1 &0 &&1 &0 &1 &0 &0 &1 &0 &1 &&0 &1 &0 &1 &1 &0 &1 &0 &\\
&  &  &  &  &  &  &  &-1&&0 &0 &0 &0 &0 &0 &0 &1 &&1 &0 &0 &1 &0 &1 &1 &0 &&0 &1 &1 &0 &1 &0 &0 &1 &\\
&  \ar@{-}[rrrrrrrrrrrrrrrrrrrrrrrrrrrrrrrrrrrr]&&&&&&&&&          &&&&&&&&&            &&&&&&&&&&            &&&&&&&&\\
&  &  &  &  &  &  &  &  &&-1&0 &0 &0 &0 &0 &0 &0 &&0 &0 &0 &0 &0 &0 &0 &0 &&1 &1 &1 &1 &1 &1 &1 &1 &\\
&  &  &  &  &  &  &  &  &&  &-1&0 &0 &0 &0 &0 &0 &&0 &0 &0 &0 &1 &1 &1 &1 &&1 &1 &1 &1 &0 &0 &0 &0 &\\
&  &  &  &  &  &  &  &  &&  &  &-1&0 &0 &0 &0 &0 &&0 &0 &1 &1 &0 &0 &1 &1 &&1 &1 &0 &0 &1 &1 &0 &0 &\\
&  &  &  &  &  &  &  &  &&  &  &  &-1&0 &0 &0 &0 &&0 &1 &0 &1 &0 &1 &0 &1 &&1 &0 &1 &0 &1 &0 &1 &0 &\\
&  &  &  &  &  &  &  &  &&  &  &  &  &-1&0 &0 &0 &&0 &0 &1 &1 &1 &1 &0 &0 &&1 &1 &0 &0 &0 &0 &1 &1 &\\
&  &  &  &  &  &  &  &  &&  &  &  &  &  &-1&0 &0 &&0 &1 &1 &0 &0 &1 &1 &0 &&1 &0 &0 &1 &1 &0 &0 &1 &\\
&  &  &  &  &  &  &  &  &&  &  &  &  &  &  &-1&0 &&0 &1 &0 &1 &1 &0 &1 &0 &&1 &0 &1 &0 &0 &1 &0 &1 &\\
&  &  &  &  &  &  &  &  &&  &  &  &  &  &  &  &-1&&0 &1 &1 &0 &1 &0 &0 &1 &&1 &0 &0 &1 &0 &1 &1 &0 &\\
&  \ar@{-}[rrrrrrrrrrrrrrrrrrrrrrrrrrrrrrrrrrrr]&&&&&&&&&          &&&&&&&&&            &&&&&&&&&&            &&&&&&&&\\
&  &  &  &  &  &  &  &  &&  &  &  &  &  &  &  &  &&-2&0 &0 &0 &0 &0 &0 &0 &&2 &0 &0 &0 &0 &0 &0 &0 &\\
&  &  &  &  &  &  &  &  &&  &  &  &  &  &  &  &  &&  &-2&0 &0 &0 &0 &0 &0 &&0 &2 &0 &0 &0 &0 &0 &0 &\\
&  &  &  &  &  &  &  &  &&  &  &  &  &  &  &  &  &&  &  &-2&0 &0 &0 &0 &0 &&0 &0 &2 &0 &0 &0 &0 &0 &\\
&  &  &  &  &  &  &  &  &&  &  &  &  &  &  &  &  &&  &  &  &-2&0 &0 &0 &0 &&0 &0 &0 &2 &0 &0 &0 &0 &\\
&  &  &  &  &  &  &  &  &&  &  &  &  &  &  &  &  &&  &  &  &  &-2&0 &0 &0 &&0 &0 &0 &0 &2 &0 &0 &0 &\\
&  &  &  &  &  &  &  &  &&  &  &  &  &  &  &  &  &&  &  &  &  &  &-2&0 &0 &&0 &0 &0 &0 &0 &2 &0 &0 &\\
&  &  &  &  &  &  &  &  &&  &  &  &  &  &  &  &  &&  &  &  &  &  &  &-2&0 &&0 &0 &0 &0 &0 &0 &2 &0 &\\
&  &  &  &  &  &  &  &  &&  &  &  &  &  &  &  &  &&  &  &  &  &  &  &  &-2&&0 &0 &0 &0 &0 &0 &0 &2 &\\
&  \ar@{-}[rrrrrrrrrrrrrrrrrrrrrrrrrrrrrrrrrrrr]&&&&&&&&&          &&&&&&&&&            &&&&&&&&&&            &&&&&&&&\\
&  &  &  &  &  &  &  &  &&  &  &  &  &  &  &  &  &&  &  &  &  &  &  &  &  &&-2&0 &0 &0 &0 &0 &0 &0 &\\
&  &  &  &  &  &  &  &  &&  &  &  &  &  &  &  &  &&  &  &  &  &  &  &  &  &&  &-2&0 &0 &0 &0 &0 &0 &\\
&  &  &  &  &  &  &  &  &&  &  &  &  &  &  &  &  &&  &  &  &  &  &  &  &  &&  &  &-2&0 &0 &0 &0 &0 &\\
&  &  &  &  &  &  &  &  &&  &  &  &  &  &  &  &  &&  &  &  &  &  &  &  &  &&  &  &  &-2&0 &0 &0 &0 &\\
&  &  &  &  &  &  &  &  &&  &  &  &  &  &  &  &  &&  &  &  &  &  &  &  &  &&  &  &  &  &-2&0 &0 &0 &\\
&  &  &  &  &  &  &  &  &&  &  &  &  &  &  &  &  &&  &  &  &  &  &  &  &  &&  &  &  &  &  &-2&0 &0 &\\
&  &  &  &  &  &  &  &  &&  &  &  &  &  &  &  &  &&  &  &  &  &  &  &  &  &&  &  &  &  &  &  &-2&0 &\\
&  &  &  &  &  &  &  &  &&  &  &  &  &  &  &  &  &&  &  &  &  &  &  &  &  &&  &  &  &  &  &  &  &-2&\\
&  &  &  &  &  &  &  &  &&  &  &  &  &  &  &  &  &&  &  &  &  &  &  &  &  &&  &  &  &  &  &  &  &  &
} 
}\!\!\!\!\right) 
}
\end{align*}
    \caption{The intersection matrix of  
    $A_{0, 1}, \ldots, A_{7, 1}$, $A_{0, 2}, \ldots, A_{7,2}$, $\Theta_{0,1}, \ldots, \Theta_{7,1}$, $\Theta_{0,2}, \ldots, \Theta_{7,2}$ in this order
    in a rational quasi-elliptic surface of type (g).
    }
    \label{matrix(g)}
\end{figure}

\begin{proof}
Let $Z$ be a rational quasi-elliptic surface of type (g).
Then there are exactly sixteen $(-2)$-curves $\{\Theta_{i, j}\}_{0 \leq i \leq 7, 1 \leq j \leq 2}$ on $Z$, which satisfies that
\begin{align*}
    (\Theta_{i,j}, \Theta_{i',j'}) >0 \iff (\Theta_{i,j}, \Theta_{i',j'}) = 2 \iff i=i' \text{ and } j \neq j'.
\end{align*}
On the other hand, as we described in Theorem \ref{q-ell} (5), there is exactly sixteen sections $\{A_{k, l}\}_{0 \leq k \leq 7, 1 \leq l \leq 2}$ on $Z$, which satisfies that
\begin{align*}
    (A_{k,l}, A_{k',l'}) >0 \iff (A_{k,l}, A_{k',l'}) = 1 \iff k=k' \text{ and } l \neq l'.
\end{align*}
By Theorem \ref{q-ell} (5), we may assume that 
\begin{align*}
    (\Theta_{i,j}, A_{0,2})\neq 0 &\iff (\Theta_{i,j}, A_{0,2})=1 \iff j=2, \text{ and }\\
    (\Theta_{0,2}, A_{k,l})\neq 0 &\iff (\Theta_{0,2}, A_{k,l})=1 \iff l=2.
\end{align*}

By contracting $A_{0,2}, \Theta_{0,2}$ and $A_{i,1}$ for $1 \leq i \leq 7$, we get a birational morphism $h \colon Z \to \PP^2_k$.
Let $t=h(A_{0,2} \cup \Theta_{0,2})$, $t_i=h(A_{i,1})$, and $D_i=h_*\Theta_{i,1}$ for $1 \leq i \leq 7$.
To show the assertion, we prepare some claims.

\begin{cln}\label{cl:typeg-1}
$h_*\Theta_{i,j} \sim \sO_{\PP^2_k}(j)$ for each $1 \leq i \leq 7$ and $1 \leq j \leq 2$.
\end{cln}

\begin{clproof}
We need only consider the case where $i=1$ by symmetry and the case where $j=1$ since $h_*(\Theta_{1, 1}+\Theta_{1,2}) \sim h_*(-K_Y) \sim \sO_{\PP^2_k}(3)$.
Suppose by contradiction that $h_*\Theta_{1,1} \sim \sO_{\PP^2_k}(2)$.
Then exactly six of $A_{1,1}, A_{2, 1}, \ldots, A_{7,1}$ intersect with $\Theta_{1,1}$ since $(h_*\Theta_{1,1})^2-\Theta_{1,1}^2=6$.
We may assume that $(A_{1,1}, \Theta_{1,1})=0$.

Assume that $h_*\Theta_{i,1} \sim \sO_{\PP^2_k}(2)$ for some $2 \leq i \leq 7$.
Then $(h_*\Theta_{1,1}, h_*\Theta_{i,1}) = 4$.
However, at least five of $A_{1,1}$, $A_{2, 1}$, $\ldots$, $A_{7,1}$ intersect with both $\Theta_{1,1}$ and $\Theta_{i,1}$, which implies that $(h_*\Theta_{1,1}, h_*\Theta_{i,1}) \geq 5$, a contradiction.
Hence $h_*\Theta_{i,1} \sim \sO_{\PP^2_k}(1)$ and $(h_*\Theta_{1,1}, h_*\Theta_{i,1}) = 2$ for each $2 \leq i \leq 7$.
For such an $i$, exactly three of $A_{1,1}, A_{2, 1}, \ldots, A_{7,1}$ intersect with $\Theta_{i,1}$ since $(h_*\Theta_{i,1})^2-\Theta_{i,1}^2=3$.
Moreover, $A_{1,1}$ intersects with $\Theta_{i,1}$ since otherwise we would obtain $(h_*\Theta_{1,1}, h_*\Theta_{i,1}) \geq 3$.

On the other hand, assume that $A_{k, 1}$ intersects with both $\Theta_{i_1, 1}$ and $\Theta_{i_2, 1}$ for some $2 \leq k \leq 7$ and $2 \leq i_1 < i_2 \leq 7$.
Then $(h_*\Theta_{i_1,1}, h_*\Theta_{i_2,1}) = 1$ since they are lines. 
However, $A_{1, 1}$ also intersects with both $\Theta_{i_1, 1}$ and $\Theta_{i_2, 1}$, which implies that $(h_*\Theta_{i_1,1}, h_*\Theta_{i_2,1}) \geq 2$, a contradiction. 

Hence we may assume that $\Theta_{i,1}$ intersects with $A_{1,1}, A_{2i-2,1}, A_{2i-1,1}$ for $2 \leq i \leq 4$.
However, it implies that $(h_*\Theta_{5,1}, h_*\Theta_{i,1}) \geq 2$ for some $2 \leq i \leq 4$, a contradiction. 
Therefore $h_*(\Theta_{1,1}) \sim \sO_{\PP^2_k}(1)$.
\hfill $\blacksquare$
\end{clproof}

\begin{cln}\label{cl:typeg-2}
There are coordinates of $\PP^2_k$ such that $\{t_i\}_{1 \leq i \leq 7}$ is the set of $\FF_2$-rational points and $\{D_i\}_{1 \leq i \leq 7}$ is the set of lines defined over $\FF_2$.
\end{cln}

\begin{clproof}
By Claim \ref{cl:typeg-1}, $\{D_i\}_{1 \leq i \leq 7}$ is a set of lines passing through exactly three of $\{t_i\}_{1 \leq i \leq 7}$.
Hence the set $\Sigma \coloneqq \{(i,j) \mid D_i \text{ passes through } t_j\}$ consists of $21$ elements.
On the other hand, distinct two lines cannot share two points.
Combining this fact and $\sharp \Sigma=21$, we conclude that $\{t_i\}_{1 \leq i \leq 7}$ is a set of points contained in exactly three of $\{D_i\}_{1 \leq i \leq 7}$.

Next, let us show that $\{t_i\}_{1 \leq i \leq 7}$ contains four points in general position.
Changing the indices of $\{D_i\}_{1 \leq i \leq 7}$ and $\{t_i\}_{1 \leq i \leq 7}$, we may assume that $D_1$ (resp.\,$D_2$) passes through $t_1$ and $t_2$ (resp.\,$t_1$ and $t_3$).
Since three of $\{D_i\}_{1 \leq i \leq 7}$ passes through $t_2$, it contains the line spanned by $t_2$ and $t_3$, say $D_4$.
Then there is a unique point, say $t_7$, in $\{t_i\}_{1 \leq i \leq 7}$ disjoint from $D_1 \cup D_2 \cup D_4$.
Hence $t_1, t_2, t_3$, and $t_7$ are in general position.

Then there are coordinates of $\PP^2_k$ such that $t_1=[1:0:0]$, $t_2=[0:1:0]$, $t_3=[0:0:1]$ and $t_7=[1:1:1]$.
Then we may assume that $t_4=[1:1:0], t_5=[0:1:1]$ and $t_6=[1:0:1]$.
Since each of $D_i$ is a span of two $\FF_2$-rational point, it is also defined over $\mathbb{F}_2$.
\hfill $\blacksquare$
\end{clproof}

Fix coordinates of $\PP^2_k$ as above.
By construction, $t$ is not contained in $D_i$ for $1 \leq i \leq 7$.
Now define $\zeta \colon \{1, \ldots, 7\} \to I= \{(1,2,4)$, $(1,3,6)$, $(1,5,7)$, $(2,3,5)$, $(2,6,7)$, $(3,4,7)$, $(4,5,6)\}$ which maps $1$, $2$, $3$, $4$, $5$, $6$, and $7$ to $(1,2,4)$, $(1,3,6)$, $(1,5,7)$, $(2,3,5)$, $(2,6,7)$, $(3,4,7)$, and $(4,5,6)$ respectively.
By Claims \ref{cl:typeg-1} and \ref{cl:typeg-2}, we may assume that $D_i$ contains $t_l$ for all $l \in \zeta(i)$ and $1 \leq i \leq 7$.
Then the following hold for $1 \leq i \leq 7$.
\begin{itemize}
    \item $\Theta_{i,1}$ is the strict transform by $h$ of the line passing through $t_l$ for all $l \in \zeta(i)$.
    \item $\Theta_{i,2}$ is the strict transform by $h$ of the conic passing through $t$ and $t_l$ for all $l \in \{1, \ldots, 7\} \setminus \zeta(i)$.
    \item $A_{i, 1}$ is the exceptional divisor over $t_i$.
    \item $A_{i, 2}$ is the strict transform by $h$ of the line passing through $t$ and $t_i$.
    \item $\Theta_{0,1}$ is the strict transform by $h$ of the cuspidal cubic passing through $t, t_1, \ldots, t_7$ which has a cusp at $t$.
    \item The tangent line of $h(\Theta_{i,2})$ at $t$ is independent of the choice of $i$, and $A_{0,1}$ is the strict transform of this line by $h$.
    \item $Z$ is obtained by blowing up $\PP^2_k$ at $t_j$ once for $1 \leq j \leq 7$ and at $t$ twice along $h(A_{0,1})$, and the $h$-exceptional divisor over $t$ consists of $A_{0,2}$ and $\Theta_{0,2}$.
\end{itemize}

From these facts, it is easy to check that Figure \ref{matrix(g)} is the intersection matrix of $A_{0, 1}, \ldots, A_{7, 1}$, $A_{0, 2}, \ldots, A_{7,2}$, $\Theta_{0,1}, \ldots, \Theta_{7,1}$, $\Theta_{0,2}, \ldots, \Theta_{7,2}$ in this order.
\end{proof}

%%%%%%%%%%%%%%%%%%%%%%%%%%%%%%%%%%%%%%%%%%%%%%%%%%%%%%%%%%%%%%%%%%%%%%%%%%%%%%%%%%%%%%%%%%%%%%%%%%%%%%%%%%%%%%%%%%%%%

\section{Proof of Theorem \ref{smooth, Intro}}
\label{sec:smooth}
This section is devoted to proving Theorem \ref{smooth, Intro}. 
First, we show that (NL) $\Rightarrow$ (NB).

\begin{prop}\label{NBtoNL}
Let $X$ be a Du Val del Pezzo surface whose general anti-canonical member is smooth. Then $X$ is log liftable over every Noetherian complete local ring with the residue field $k$.
\end{prop}

\begin{proof}
Let $\pi\colon Y\to X$ be the minimal resolution.
By Proposition \ref{prop:W2lift} (2), it suffices to show that $H^2(X, T_X)=H^2(X, \sO_X)=0$.
Since $-K_X$ is ample, it follows that $H^2(X, \sO_X)\cong H^0(X, \sO_X(K_X))=0$.
Now we show that $H^2(X, T_X)=0$.
By the Serre duality, it follows that
\begin{align*}
H^2(X, T_X)\cong&\mathrm{Hom}_{\sO_X}(T_X, \sO_X(K_X))\\
                                 \cong&\mathrm{Hom}_{\sO_X}(\sO_X(-K_X), \Omega_X^{[1]}),
\end{align*}
where $\Omega_X^{[1]}$ denotes the double dual of $\Omega_X$.
Suppose by contradiction that there exists an injective $\sO_X$-module homomorphism $s\colon \sO_X(-K_X)\hookrightarrow \Omega_X^{[1]}$.
Let $C\in |-K_X|$ be a general member and $s|_C\colon\sO_C(-K_X)\rightarrow \Omega^{[1]}_{X}|_{C}$ be the restriction of $s$ on $C$.
By Lemma \ref{basic} (2), we may assume that $C$ is not contained in the zero locus of $s$. In particular, $s|_C$ is injective.
By assumption, we also may assume that $C$ is a smooth Cartier divisor.
In particular, $X$ is smooth along $C$ and hence $\Omega^{[1]}_{X}|_{C}=\Omega_{X}|_{C}$.
Let $t \colon \sO_C(-K_X) \to \Omega_C$ be the composition of $s|_C \colon \sO_C(-K_X)\hookrightarrow \Omega_{X}|_{C}$ and the canonical map
$\Omega_{X}|_{C} \to \Omega_C$. 
By the conormal exact sequence, we obtain the following diagram.
\begin{equation*}
\xymatrix{ & & \sO_C(-K_X) \ar@{.>}[ld] \ar[d]^{s|_C} \ar[rd]^{t} &\\
                 0\ar[r] &\sO_C(-C) \ar[r]   & \Omega_X|_{C} \ar[r]  & \Omega_C \ar[r] & 0.}
\end{equation*}
Then $t$ is the zero map since $\sO_C(-K_X)$ is ample and $\Omega_C=\sO_C$. 
Hence the above diagram induces an injective $\sO_C$-module homomorphism $\sO_C(-K_X) \hookrightarrow \sO_C(-C)$, but this is a contradiction because $\sO_C(-C)=\sO_C(K_X)$ is anti-ample.
Therefore we obtain the assertion.
\end{proof}

Next, we prove that (ND) $\Rightarrow$ (NL).

\begin{prop}\label{NDtoNL}
Let $X$ be a Du Val del Pezzo surface. Let $R$ be a Noetherian integral domain of characteristic zero with a surjective homomorphism $R \to k$.
If $X$ is log liftable over $R$ via the associated morphism $\alpha \colon \Spec k \to \Spec R$,
then there exists a Du Val del Pezzo surface over $\C$ which has the same Dynkin type, the same Picard rank, and the same degree as $X$.
\end{prop}
\begin{proof}
Let $m$ be the kernel of the homomorphism $R \to k$.
Replacing $R$ with the completion of $R_m$, we may assume that $R$ is a Noetherian complete local ring with residue field $k$.
Thus, by Lemma \ref{equiv} (1), the pair $(Y,  E_{\pi})$ lifts to $R$, where $\pi \colon Y \to X$ is the minimal resolution.
We denote by $E_{\pi}\coloneqq\sum_{i=1}^{r} E_i$ the irreducible decomposition.
Let $(\mathcal{Y}, \mathcal{E} \coloneqq \sum_{i=1}^{r} \mathcal{E}_i)$ be an $R$-lifting of $(Y, E_{\pi})$.
We take a subfield $K$ of the field of fractions of $R$ such that $K$ is of finite transcendence degree over $\Q$, and the generic fiber of $\mathcal{Y}$ and that of each $\mathcal{E}_i$ are defined over $K$.
Fix an inclusion $K \subset \C$ and take $\overline{K} \subset \C$ as the algebraic closure of $K$.
For a field extension $K \subset F$, we use the notation $Y_{F} \coloneqq \mathcal{Y}\otimes_{R} F$ and $E_{i,F} \coloneqq \mathcal{E}_i\otimes_{R} F$ for each $i$.
Since the geometrical connectedness are open property by \cite[Th\'eor\`eme 12.2.4 (viii)]{EGAIV3}, $Y_{\C}$ and $E_{i,\C}$ are smooth varieties. 
Since $E_{\C} \coloneqq \sum_{i=1}^{r} E_{i,\C}$ has the same intersection matrix as $E_{\pi}$, we have a contraction $\pi_{\C} \colon Y_{\C} \to X_{\C}$ of $E_{\C}$ and $X_{\C}$ has the same Dynkin type as $X$. 
By the crepantness of $\pi$ and $\pi_{\C}$, we obtain $K_X^2=K_Y^2=K_{Y_{\C}}^2=K_{X_{\C}}^2$.

Next, we prove that ${X_{\C}}$ is a Du Val del Pezzo surface.
For the sake of contradiction, we assume that $-K_{X_{\C}}$ is not ample. Since $K^2_{X_{\overline{K}}}=K_{X_{\C}}^2>0$, there exists an integral curve $C_0\subset Y_{\C}$ defined over $\overline{K}$ such that $C_0$ is not contained in $E_{\C}$ and $(-K_{Y_{\C}} \cdot C_0) \leq 0$.
We take a finite Galois extension field $L$ of $K$ such that $C_{0}$ is defined over $L$
and write $C \coloneqq \sum_{\sigma\in \mathrm{Gal}(L/K)} \sigma(C_{0})$, which is defined over $K$. 
By the choice of $C_0$, there are no components contained in both $C$ and $E_{K}$. 
We denote by $\overline{C}$ the closure of $C$ in $\mathcal{Y}$ and define an effective divisor $C_{k} \coloneqq \overline{C}\otimes_{R}k$. 

Now, assume that $\Supp{C_{k}}\subset E_{\pi}$. 
Then we can write $C_{k}=\sum_{i=1}^r a_iE_{i}$ for some $a_i \geq 0$. 
Since $C$ and $E_{K}$ have no common components, we have $C_{k}^2=(C \cdot \sum_{i=1}^r a_i E_{i,K}) \geq 0$.
By the negative definiteness of $E_{\pi}$, we obtain $a_i=0$ for $1 \leq i \leq r$, a contradiction.
Thus there exists an integral curve $C'_{k}\subset C_{k}$ such that $C'_{k}$ is not contained in $E_{\pi}$.
Since $-K_{Y}$ is nef, we have 
$0 \leq (-K_Y \cdot C'_k) \leq (-K_{Y}\cdot C_{k}) = (-K_{Y_{K}}\cdot C) = |\mathrm{Gal}(L/K)| (-K_{Y_{\C}}\cdot C_{0}) \leq 0$. 
Hence $(-K_Y \cdot C'_{k}) = (-K_X\cdot \pi_{*}(C'_{k})) = 0$, a contradiction with
the ampleness of $-K_{X}$. 
Therefore, $X_{\C}$ is a Du Val del Pezzo surface. 

Finally, we show that $\rho(X)=\rho(X_{\C})$. 
Since $Y$ and $Y_{\C}$ are smooth rational surfaces, we have $\rho(Y_{\C})=10-K_{Y_{\C}}^2=10-K_Y^2=\rho(Y)$. Then we obtain $\rho(X)=\rho(X_{\C})$ because $\pi_{\C}$ contracts the same number of $(-2)$-curves as $\pi$.
\end{proof}

Finally, we prove that (NK) $\Rightarrow$ (NL).

\begin{lem}\label{lem:KV-2}
Let $f\colon Z\to X$ be a birational morphism of normal projective klt surfaces and $A$ an ample $\Z$-divisor on $X$.
Suppose that $(Z, \lceil f^{*}A\rceil-f^{*}A)$ is klt.
Then $H^i(X, \sO_X(-A))=H^i(Z, \sO_Z(-\lceil f^{*}A\rceil))$ for $i\geq 0$.
\end{lem}
\begin{proof}
By \cite[Theorem 2.12]{Tan15}, it follows that $R^if_{*}\sO_Z(K_Z+\lceil f^{*}A\rceil))=0$ for $i \geq 1$.
Then the Leray spectral sequence
\[
E_2^{p,q}=H^{p}(X, R^{q}f_{*}\sO_{Z}(K_{Z}+\lceil f^{*}A \rceil))\Rightarrow E^{p+q}=H^{p+q}(Z,\sO_{Z}(K_{Z}+\lceil f^{*}A \rceil))
\]
gives $H^i(X, f_{*}\sO_Z(K_Z+\lceil f^{*}A \rceil))\cong H^i(Z, \sO_{Z}(K_{Z}+\lceil f^{*}A \rceil))$.
Since $X$ is klt, we obtain
\begin{align*}
K_{Z}+\lceil f^{*}A \rceil&=\lceil K_Z-f^{*}K_X+f^{*}(K_X+A) \rceil\\
                          &=\lfloor f^{*}(K_X+A) \rfloor +F
\end{align*}
for some effective $f$-exceptional $\Z$-divisor $F$.
Then $H^i(X, f_{*}\sO_Z(K_Z+\lceil f^{*}A \rceil))=H^i(X, \sO_X(K_X+A))$ by the projection formula.
Hence the assertion follows from the Serre duality for Cohen-Macaulay sheaves \cite[Theorem 5.71]{KM98}.
\end{proof}

\begin{prop}\label{lem:KV-3}
Let $X$ be a normal projective surface and $A$ an ample $\Q$-Cartier $\Z$-divisor.
Suppose that there exists a log resolution $f\colon Z\to X$ such that $(Z, E_{f})$ lifts to $W_2(k)$.
Then $H^1(X, \sO_X(-A))=0$.
\end{prop}
\begin{proof}
By Lemma \ref{lem:KV-2}, it follows that $H^1(X, \sO_X(-A))=H^1(Z, \sO_Z(-\lceil f^{*}A\rceil))$.
Take an $f$-exceptional effective $\Q$-divisor $F$ such that $\lceil f^{*}A-F\rceil=\lceil f^{*}A\rceil$ and $f^{*}A-F$ is ample. Since $\Supp(\lceil f^{*}A-F\rceil-(f^{*}A-F))$ is contained in $E_{f}$,
Theorem \ref{loglift vanishing} shows that $H^1(Z, \sO_Z(-\lceil f^{*}A-F\rceil))=0$. Hence we get the assertion.
\end{proof}

Now we can prove Theorem \ref{smooth, Intro}.

\begin{proof}[Proof of Theorem \ref{smooth, Intro}]
The assertions (1), (2), and (3) follow from Propositions \ref{NBtoNL}, \ref{NDtoNL}, and \ref{lem:KV-3} respectively.
\end{proof}

%%%%%%%%%%%%%%%%%%%%%%%%%%%%%%%%%%%%%%%%%%%%%%%%%%%%%%%%%%%%%%%%%%%%%%%%%%%%%%%%%%%%%%%%%%%%%%%%%%%%%%%%%%%%%%

\section{Dynkin types}\label{sec:sing}
In this section, we determine the Dynkin types of Du Val del Pezzo surfaces satisfying (NB).
By Lemma \ref{basic}, such a del Pezzo surface is of degree at most two, and $p=2$ or $3$.
First, we treat the case where the degree is one.

\begin{prop}\label{prop:deg=1}
Let $X$ be a Du Val del Pezzo surface with $K_X^2=1$ and $\pi \colon Y \to X$ the minimal resolution.
Take $g \colon Z \to Y$ as the blow-up at the base point of $|-K_Y|$ and $f \colon Z \to \PP^1_k$ the genus one fibration defined by $|-K_Z|$.
    \[
\xymatrix{
Z \ar[r]^{g}\ar[rrd]^f & Y \ar[r]^{\pi} &  X  \ar@{..>}[d]^{\phi_{|-K_X|}}  \\
                          &                        &                   \PP^1_k                    \\
}
\]
Then the following hold.
\begin{enumerate}
    \item[\textup{(1)}] $X$ satisfies (NB) if and only if $f$ is a quasi-elliptic fibration.
    \item[\textup{(2)}] For another Du Val del Pezzo surface $X'$ of degree one, take $\pi' \colon Y' \to X'$ and $g' \colon Z' \to Y'$ as above.
    Then $X \cong X'$ if and only if $Z \cong Z'$.
    \item[\textup{(3)}] Suppose that $p=3$ and $Z$ is a rational quasi-elliptic surface. 
    Then $Z$ is of type (1) (resp.\ (2), (3)) if and only if $\Dyn(X)=E_8$ (resp.\ $A_2+E_6, 4A_2$).
    \item[\textup{(4)}] Suppose that $p=2$ and $Z$ is a rational quasi-elliptic surface. 
    Then $Z$ is of type (a) (resp.\ (b), (c), (d), (e), (f), (g)) if and only if $\Dyn(X)=E_8$ (resp.\ $D_8, A_1+E_7, 2D_4, 2A_1+D_6, 4A_1+D_4, 8A_1$)
\end{enumerate}
\end{prop}

\begin{proof}
\noindent (1): A general member of $|-K_Y|$ is isomorphic to its image by $\pi$ since it is disjoint from the exceptional divisor $E_\pi$ of $\pi$.
On the other hand, the base locus of $|-K_{Y}|$ consists of one point, say $y$.
Since any two members of $|-K_{Y}|$ intersect transversely with each other at $y$, each $f$-fiber is isomorphic to its image on $Y$.
Hence a general $f$-fiber is isomorphic to its image on $X$, and the assertion holds.

\noindent (2):
Take $f' \colon Z' \to \PP^1_k$ as the morphism given by $|-K_{Z'}|$.
Suppose that there is an isomorphism $X \cong X'$.
Then it ascends to an isomorphism $Z \cong Z'$ since the construction of $\pi$, $\pi'$, $g$, and $g'$ are canonical.
On the other hand, suppose that there is an isomorphism $\sigma \colon Z \cong Z'$.
Then this isomorphism is compatible with the genus one fibration structures since $-K_Z \sim \sigma^*(-K_Z')$.
In particular, $\sigma$ maps each $f$-section to an $f'$-section.
Since $E_g$ (resp.\ $E_{g'}$) is an $f$-section (resp.\ an $f'$-section) and the Mordell-Weil group $\MW(Z)$ acts on the set of $f$-sections transitively, we may assume that $\sigma$ maps $E_g$ to $E_{g'}$.
Hence it descends to an isomorphism $\sigma_Y \colon Y \cong Y'$.
Since both $\pi$ (resp.\ $\pi'$) is the contraction of all the $(-2)$-curves on $Y$ (resp.\ $Y'$), $\sigma_Y$ also descends to the desired isomorphism $\sigma_X \colon X \cong X'$.

\noindent (3) and (4): Since $\MW(Z)$ acts on the set of $f$-sections transitively, we may assume that $E_g$ is the section $O$ in Figures \ref{fig:dual(1)-(3)}--\ref{fig:dual(g)}.
Hence the assertions follows from Figures \ref{fig:dual(1)-(3)}--\ref{matrix(g)} and Tables \ref{tab:type(3)} and \ref{typef}.
\end{proof}

We will use the following proposition in Section \ref{sec:singliftable}.

\begin{prop}\label{4A_2NB}
Let $X$ be a Du Val del Pezzo surface with $\Dyn(X)=4A_2$. Then $X$ satisfies (NB) if and only if $p=3$.  
\end{prop}

\begin{proof}
The only if part follows from Proposition \ref{prop:deg=1}.
To show the other direction, we suppose that $p=3$.
Let us take $Y$ and $Z$ as in Proposition \ref{prop:deg=1}.
Suppose by contradiction that a general anti-canonical member of $X$ is smooth.
Then $Z$ is an extremal rational elliptic surface with four singular fibers.
By \cite[Theorem 2.1]{Lang1}, its singular fibers are $(\textup{I}_8, \textup{I}_2, \textup{I}_1, \textup{I}_1)$,  $(\textup{I}_5, \textup{I}_5, \textup{I}_1, \textup{I}_1)$, or $(\textup{I}_4, \textup{I}_4, \textup{I}_2, \textup{I}_2)$.
However, this implies that $\Dyn(X)=A_1+A_7, 2A_4$, or $2A_1+2A_3$, a contradiction.
\end{proof}

Next, we treat the case where the degree is two.
The following proposition claims that the anti-canonical double covering must be purely inseparable.
\begin{prop}\label{sep}
Let $X$ be a Du Val del Pezzo surface with $K_X^2=2$.
Suppose that the anti-canonical double covering $\phi_{|-K_X|}\colon X \to \PP_k^2$ is separable.
Then a general anti-canonical member is smooth.
\end{prop}
\begin{proof}
Take the minimal resolution $\pi \colon Y \to X$.
Let $t \in \PP^2_k$ be a general point and $V \subset |-K_Y|$ the pullback of the pencil of lines in $\PP^2_k$ passing through $t$.
Then the base locus of $V$ consists of two points, say $y_1$ and $y_2$, such that there are no $(-1)$-curves passing through $y_1$ or $y_2$ because $t$ is general and there exist only finitely many $(-1)$-curves on $Y$.
Let $g \colon Z \to Y$ be the blow-up at $y_1$ and $y_2$, and $E_i$ the $g$-exceptional divisor over $y_i$ for $i \in \{1, 2\}$. 
Then $g$ gives a resolution $f\colon Z \to \PP^1_k$ of the indeterminacy of the pencil $\phi_V \colon Y \dashrightarrow \PP^1_k$. Since any two members of $V$ intersect transversely at $y_1$ and $y_2$,
a general $f$-fiber is isomorphic to its image on $Y$.

\[
\xymatrix{
Z \ar[r]^{g}\ar[rrd]^f & Y \ar[r]^{\pi} &  X  \ar[r]^{\phi_{|-K_X|}}\ar@{..>}[d]^{\phi_{V}} &      \PP^2_k \\
                          &                        &                   \PP^1_k &                   \\
}
\]

Now let us show that a general member of $|-K_X|$ is smooth.
Suppose by contradiction that members of $|-K_X|$ are all singular.
Then $Z$ is a quasi-elliptic surface, and $E_1$ and $E_2$ are $f$-sections by the same arguments as in Proposition \ref{prop:deg=1}.
Since there are no $(-1)$-curves on $Y$ which pass through $y_1$ or $y_2$,
each $(-2)$-curves in $Z$ either intersects with both $E_1$ and $E_2$ or is disjoint from both $E_1$ and $E_2$.
However, there is no such a choice of two sections by Figures \ref{fig:dual(1)-(3)}--\ref{matrix(g)} and Tables \ref{tab:type(3)} and \ref{typef}, a contradiction.
\end{proof}

\begin{prop}\label{insep}
Let $X$ be a Du Val del Pezzo surface with $K_X^2=2$ satisfying (NB).
Then $p=2$ and $\Dyn(X)=E_7$, $A_1+D_6$, $3A_1+D_4$, or $7A_1$. 
\end{prop}
\begin{proof}
By Proposition \ref{sep}, the anti-canonical double covering $\phi_{|-K_X|}\colon X \to \PP^2_k$ is purely inseparable.
In particular, we have $p=2$.
Take $\pi, t$ and $V$ as in Proposition \ref{sep}.
By the generality of $t$, the base locus of $V$ consists of one point, say $y$, and no $(-1)$-curves pass through $y$.
For general two members $C_1$ and $C_2$ of $V$, they intersect with each other at $y$ with multiplicity two since $\phi_{|-K_X|}$ is a homeomorphism.
Moreover, one of them is smooth at $y$ since otherwise $2=K_Y^2=(C_1 \cdot C_2) \geq 4$.
Thus general members of $V$ are smooth at $y$, and have the same tangent direction at $y$.
Hence there is a point $y'$ infinitely near $y$ such that the blow-up $g \colon Z \to Y$ at $y$ and $y'$ gives a resolution $f\colon Z \to \PP^1_k$ of indeterminacy of the pencil $\phi_V \colon Y \dashrightarrow \PP^1_k$.
Since a general member of $V$ is smooth at $y$, a general $f$-fiber is isomorphic to its image on $X$.
In particular, $Z$ is a quasi-elliptic surface.
By construction, $E_g$ consists of a $(-1)$-curve $E_1$ and a $(-2)$-curve $E_2$.
In particular, $E_1$ is an $f$-section and $E_2$ is contained in a reducible $f$-fiber.

Suppose that the $f$-fiber containing $E_2$ has simple normal crossing support.
Then there is another $(-2)$-curve $C$ intersecting with $E_2$.
Since $C$ and $E_2$ are contained in the same $f$-fiber, $E_1$ is disjoint from $C$.
This implies, however, $g_*C$ is a $(-1)$-curve passing through $y$, a contradiction with the choice of $y$.
Hence $E_2$ is contained in a reducible $f$-fiber whose support is not simple normal crossing. 
Theorem \ref{q-ell} now shows that $E_2$ is contained in a reducible $f$-fiber of type III, where we use Kodaira's notation, and $Y$ is one of the types (c), (e), (f), and (g) in Table \ref{q-ell2}.
By Figures \ref{fig:dual(a)-(e)}--\ref{matrix(g)} and Table \ref{typef}, we conclude that $\Dyn (X) = E_7$, $A_1+D_6$, $3A_1+D_4$, or $7A_1$.
\end{proof}

Finally, let us show that there are several constructions of Du Val del Pezzo surfaces of degree two satisfying (NB).

\begin{lem}\label{lem:deg=2pre}
Let $X$ be a del Pezzo surface satisfying (NB) such that $\Dyn(X)=E_7, A_1+D_6, 3A_1+D_4$, or $7A_1$.
Let $Y$ be the minimal resolution of $X$.
Then the following hold.
\begin{enumerate}
    \item[\textup{(0)}] $K_X^2=2$ and $p=2$.
    \item[\textup{(1)}] For each point $t \in Y$ not contained in any negative rational curves, there is a rational quasi-elliptic surface $Z$, an irreducible component $T$ of reducible fiber of type III, and a section $S$ of $Z$ intersecting with $T$ such that $Y$ is given from $Z$ by contracting $S \cup T$ to $t$.
    \item[\textup{(2)}] If $\Dyn(X)=E_7$ (resp.\ $A_1+D_6, 3A_1+D_4, 7A_1$), then $Z$ as in the assertion (1) is of type (c) (resp.\ (e), (f), (g)). 
    \item[\textup{(3)}] If $\Dyn(X)=E_7, A_1+D_6$, or $3A_1+D_4$, then the union of the negative rational curves on $Y$ is a simple normal crossing divisor.
    Moreover, Figure \ref{fig:dualE_7-3A_1D_4} is the dual graph of the configuration of the negative rational curve, where black nodes (resp. white nodes) corresponds to a $(-1)$-curve (resp.\ a $(-2)$-curve).
    \item[\textup{(4)}] If $\Dyn(X)=7A_1$, then there are exactly seven $(-1)$-curves and seven $(-2)$-curves whose intersection matrix is as in Figure \ref{matrix7A1}.
    \item[\textup{(5)}] If $\Dyn(X)=E_7$, then $Y$ is also obtained from the rational quasi-elliptic surface of type (a) by blowing down $O$ and $\Theta_{\infty, 0}$ in Type (a) of Figure \ref{fig:dual(a)-(e)}.
    \item[\textup{(6)}] If $\Dyn(X)=A_1+D_6$, then $Y$ is also obtained from the rational quasi-ellptic surface of type (b) (resp.\ (c)) by blowing down $O$ and $\Theta_{\infty, 0}$ (resp.\ $O$ and $\Theta_{0, 0}$) in Type (b) (resp.\ (c)) of Figure \ref{fig:dual(a)-(e)}.
    \item[\textup{(7)}] If $\Dyn(X)=3A_1+D_4$, then $Y$ is also obtained from a rational quasi-ellptic surface of type (d) (resp.\ (e)) by blowing down $O$ and $\Theta_{0, 0}$ (resp.\ $O$ and $\Theta_{1, 2}$) in Type (d) (resp.\ (e)) of Figure \ref{fig:dual(a)-(e)}.
    \item[\textup{(8)}] If $\Dyn(X)=7A_1$, then $Y$ is also obtained from a rational quasi-ellptic surface of type (f) by blowing down $O$ and $\Theta_{1, 0}$ in Figure \ref{fig:dual(f)}.
\end{enumerate}
\end{lem}

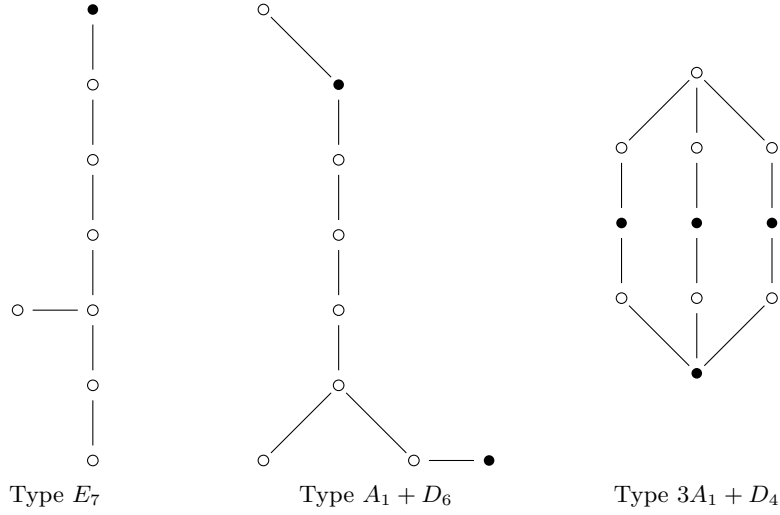
\begin{figure}[htbp]
\captionsetup[subfigure]{labelformat=empty}
\centering
    \begin{subfigure}[b]{0.3\textwidth}
        \centering
                \label{fig:dualE7}
		\begin{tikzpicture}

		\draw[thin ](0,-0.2)--(0,-0.8);
		\draw[thin ](0,-1.2)--(0,-1.8);
		\draw[thin ](0,-2.2)--(0,-2.8);
		\draw[thin ](0,-3.2)--(0,-3.8);
		\draw[thin ](0,-4.2)--(0,-4.8);
		\draw[thin ](0,-5.2)--(0,-5.8);
		\draw[thin ](-0.2,-4)--(-0.8,-4);

		\fill (0,0) circle (2pt);
		\draw[] (0,-1)circle(2pt);
		\draw[] (0,-2)circle(2pt);
		\draw[] (0,-3)circle(2pt);
		\draw[] (0,-4)circle(2pt);
		\draw[] (0,-5)circle(2pt);
		\draw[] (0,-6)circle(2pt);
		\draw[] (-1,-4)circle(2pt);

		\end{tikzpicture}
        \caption{Type $E_7$}
    \end{subfigure}
                \quad    
    \begin{subfigure}[b]{0.3\textwidth}
        \centering
                \label{fig:dualA1D6}
		\begin{tikzpicture}

		\draw[thin ](-0.8,-3)--(-0.2,-3);
		\draw[thin ](-1.1,-2.9)--(-1.9,-2.1);
		\draw[thin ](-2.1,2.1)--(-2.9,2.9);
		\draw[thin ](-2.1,-2.1)--(-2.9,-2.9);
		\draw[thin ](-2,1.2)--(-2,1.8);
		\draw[thin ](-2,0.2)--(-2,0.8);
		\draw[thin ](-2,-0.2)--(-2,-0.8);
		\draw[thin ](-2,-1.2)--(-2,-1.8);

		\fill (-2,2) circle (2pt);
		\fill (0,-3) circle (2pt);
		\draw[] (-1,-3)circle(2pt);
		\draw[] (-2,-2)circle(2pt);
		\draw[] (-2,1)circle(2pt);
		\draw[] (-2,0)circle(2pt);
		\draw[] (-2,-1)circle(2pt);
		\draw[] (-3,3)circle(2pt);
		\draw[] (-3,-3)circle(2pt);

		\end{tikzpicture}
        \caption{Type $A_1+D_6$}
    \end{subfigure} 
            \quad
    \begin{subfigure}[b]{0.3\textwidth}
    \centering
                \label{fig:dual3A1D4}
		\begin{tikzpicture}

		\draw[thin ](-0.1,1.9)--(-0.9,1.1);
		\draw[thin ](0,1.8)--(0,1.2);
		\draw[thin ](0.1,1.9)--(0.9,1.1);
		\draw[thin ](-1,0.8)--(-1,0.2);
		\draw[thin ](0,0.8)--(0,0.2);
		\draw[thin ](1,0.8)--(1,0.2);
		\draw[thin ](-1,-0.2)--(-1,-0.8);
		\draw[thin ](0,-0.2)--(0,-0.8);
		\draw[thin ](1,-0.2)--(1,-0.8);
		\draw[thin ](-0.1,-1.9)--(-0.9,-1.1);
		\draw[thin ](0,-1.8)--(0,-1.2);
		\draw[thin ](0.1,-1.9)--(0.9,-1.1);

		\fill (1,0) circle (2pt);
		\fill (0,0) circle (2pt);
		\fill (-1,0) circle (2pt);
		\fill (0,-2) circle (2pt);
		\draw[] (1,1)circle(2pt);
		\draw[] (0,1)circle(2pt);
		\draw[] (-1,1)circle(2pt);
		\draw[] (0,2)circle(2pt);
		\draw[] (1,-1)circle(2pt);
		\draw[] (0,-1)circle(2pt);
		\draw[] (-1,-1)circle(2pt);
		
		\node(a)at(0,-3){\phantom{$O$}};

		\end{tikzpicture}
        \caption{Type $3A_1+D_4$}
    \end{subfigure} 
\caption{Dual graphs of negative rational curves in a Du Val del Pezzo surface of type $E_7, A_1+D_6$, or $3A_1+D_4$ satisfying (NB)}
\label{fig:dualE_7-3A_1D_4}
\end{figure}

\begin{figure}[htbp]
    \centering
\begin{align*}
{\fontsize{7pt}{4}\selectfont
    \left(\!\!\!\!
    \vcenter{
        \xymatrix@=-1.8ex{
&&&&&&&&\ar@{-}[ddddddddddddddddd]&&&&&&&& \\
&  &  &  &  &  &  &  &&  &  &  &  &  &  &  &\\
&-1&0 &0 &0 &0 &0 &0 &&1 &1 &1 &0 &0 &0 &0 &\\
&  &-1&0 &0 &0 &0 &0 &&1 &0 &0 &1 &1 &0 &0 &\\
&  &  &-1&0 &0 &0 &0 &&0 &1 &0 &1 &0 &1 &0 &\\
&  &  &  &-1&0 &0 &0 &&1 &0 &0 &0 &0 &1 &1 &\\
&  &  &  &  &-1&0 &0 &&0 &0 &1 &1 &0 &0 &1 &\\
&  &  &  &  &  &-1&0 &&0 &1 &0 &0 &1 &0 &1 &\\
&  &  &  &  &  &  &-1&&0 &0 &1 &0 &1 &1 &0 &\\
\ar@{-}[rrrrrrrrrrrrrrrr]&  &&&&&&&&&       &&&&&&&&&\\
&  &  &  &  &  &  &  &&-2&0 &0 &0 &0 &0 &0 &\\
&  &  &  &  &  &  &  &&  &-2&0 &0 &0 &0 &0 &\\
&  &  &  &  &  &  &  &&  &  &-2&0 &0 &0 &0 &\\
&  &  &  &  &  &  &  &&  &  &  &-2&0 &0 &0 &\\
&  &  &  &  &  &  &  &&  &  &  &  &-2&0 &0 &\\
&  &  &  &  &  &  &  &&  &  &  &  &  &-2&0 &\\
&  &  &  &  &  &  &  &&  &  &  &  &  &  &-2&\\
&  &  &  &  &  &  &  &&  &  &  &  &  &  &  &
} 
}\!\!\!\!\right) 
}
\end{align*}
    \caption{The intersection matrix of negative rational curves in a Du Val del Pezzo surface of type $7A_1$
    }
    \label{matrix7A1}
\end{figure}

\begin{proof}
The assertion (0) follows from Lemma \ref{basic} and Propositions \ref{prop:deg=1} and \ref{insep}.
The essentially same proof as in that of Proposition \ref{insep} shows the assertions (1) and (2).
We see at once that the contraction of $O$ and $\Theta_{\infty, 0}$ in Types (c) and (e) of Figure \ref{fig:dual(a)-(e)} and Figure \ref{fig:dual(f)} gives the dual graph as in Figure \ref{fig:dualE_7-3A_1D_4}, and the assertion (3) holds.

Suppose that $\Dyn(X)=7A_1$ and we follow the notation of the proof of Lemma \ref{lem:typeg}.
By contracting $A_{0,2}$ and $\Theta_{0,2}$ in Figure \ref{fig:dual(g)}, $A_{i, 1}$, $A_{i, 2}$, $\Theta_{i,1}$, $\Theta_{i,2}$ $A_{0, 1}$, and $\Theta_{0,1}$ become a $(-1)$-curve, a $(0)$-curve, a $(-2)$-curve, a $(0)$-curve, a $(1)$-curve, and a cuspidal curve of self intersection number two respectively for $1 \leq i \leq 7$.
Hence the assertion (4) holds.

Finally, let us show the assertions (5)--(8).
Let $E$ be a $(-1)$-curve in $Y$ and $t \in E$ a point not contained in any $(-2)$-curve.
Then the blow-up $Y_t$ of $Y$ at $t$ is a weak del Pezzo surface whose anti-canonical members are all singular.
Hence $Y_t$ is the blow-down of a section in a rational quasi-elliptic surface $Z_t$.

Now suppose that $\Dyn(X)=E_7$ and let $E$ correspond the black node in Type $E_7$ of Figure \ref{fig:dualE_7-3A_1D_4}.
Then $Y_t$ contains eight $(-2)$-curves whose configuration is the Dynkin diagram $E_8$.
By Proposition \ref{prop:deg=1} (4), $Z_t$ is of type (a), and hence the assertion (5) holds.

Similarly, if $\Dyn(X)=A_1+D_6$ (resp.\ $3A_1+D_4$), then by Type $A_1+D_6$ (resp.\ $3A_1+D_4$) of Figure \ref{fig:dualE_7-3A_1D_4}, there are two possibility of the number of $(-2)$-curves intersecting with $E$, and $Y_t$ contains eight $(-2)$-curves whose configuration is the Dynkin diagram $D_8$ or $A_1+E_7$ (resp.\ $2D_4$ or $2A_1+D_6$).
On the other hand, if $\Dyn(X)=7A_1$, then Figure \ref{matrix7A1} shows that $E$ is unique up to symmetry, and $Y_t$ contains eight $(-2)$-curves whose configuration is the Dynkin diagram $D_4+4A_1$.
Therefore Proposition \ref{prop:deg=1} (4) shows assertions (6)--(8).
\end{proof}

%%%%%%%%%%%%%%%%%%%%%%%%%%%%%%%%%%%%%%%%%%%%%%%%%%%%%%%%%%%%%%%%%%%%%%%%%%%%%%%%%%%%%%%%%%%%%%%%%%%%%%%%%%%%%%%%%%

\section{Isomorphism classes and automorphism groups}
\label{sec:singisom}

In this section, we determine the isomorphism classes and the automorphism groups of Du Val del Pezzo surfaces satisfying (NB).

\subsection{Characteristic three}
In this subsection, we treat the case where $p=3$.

\begin{prop}\label{prop:p=3isom}
Let $X$ be a Du Val del Pezzo surface satisfying (NB) in $p=3$ and $\pi \colon Y \to X$ be the minimal resolution.
Then the following hold.
\begin{enumerate}
    \item[\textup{(1)}] $K_X^2=1$.
    \item[\textup{(2)}] $\Dyn(X) =E_8$, $A_2+E_6$, or $4A_2$.
    Moreover, $X$ is uniquely determined up to isomorphism by $\Dyn(X)$.
    \item[\textup{(3)}] If $\Dyn(X)=E_8$, then $Y$ is constructed from $\PP^2_{k, [x:y:z]}$ by blowing up at $[0:1:0]$ eight times along $\{x^3+y^2z=0\}$.
    Moreover, each negative rational curve is either exceptional over $\PP^2_{k}$ or the strict transform of $\{z=0\}$.
    \item[\textup{(4)}] If $\Dyn(X)=A_2+E_6$, then $Y$ is constructed from $\PP^2_{k, [x:y:z]}$ by blowing up at $[0:0:1]$ twice along $\{y=0\}$, at $[0:1:1]$ three times along $\{y=z\}$, and at $[0:1:0]$ three times along $\{z=0\}$.
    Moreover, each negative rational curve is either exceptional over $\PP^2_{k}$ or the strict transform of $\{y=0\}$, $\{y=z\}$, $\{z=0\}$, or $\{x=0\}$.
    \item[\textup{(5)}] If $\Dyn(X)=4A_2$, then $Y$ is constructed from $\PP^2_{k, [x:y:z]}$ by blowing up all the $\FF_3$-rational points on $\{z \neq 0\}$ except $[0:0:1]$. 
    Moreover, each negative rational curve on $Y$ is either exceptional over $\PP^2_{k}$ or the strict transform of lines passing through two of the eight points as above.
    \item[\textup{(6)}] $Y$ and each negative rational curve on $Y$ are defined over $\FF_3$.
\end{enumerate}
\end{prop}

\begin{proof}
\noindent (1): The assertion follows from Lemma \ref{basic} and Proposition \ref{insep}.

\noindent (2): The assertion follows from Proposition \ref{prop:deg=1} and Theorem \ref{thm:q-ell3}.

\noindent (3)--(5): Let $g \colon Z \to Y$ be the blow-up at the base point of $|-K_Y|$.
By Proposition \ref{prop:deg=1} (3), $Z$ is a rational quasi-elliptic surface of type (1), (2), and (3) when $\Dyn(X) = E_8, A_2+E_6$, and $4A_2$ respectively.
Since the $\MW(Z)$-action on the set of sections of $Z$ is transitive, we may assume that $g$ is the contraction of the section $O$ in Figure \ref{fig:dual(1)-(3)}.
Take $h \colon Z \to \PP^2_k$ as in Lemma \ref{lem:Ito2Rem4}.
Then the assertion follows from the description of the induced morphism $h' \colon Y \to \PP^2_k$ and the image of negative rational curves on $Z$ via $h$.

\noindent (6): The assertions directly follow from the assertions (3)--(5).
\end{proof}

\begin{cor}\label{cor:deg=1p=3auto}
Let $X$ be a Du Val del Pezzo surface satisfying (NB) in $p=3$.
When $\Dyn(X)=E_8$ (resp.\ $A_2+E_6$, $4A_2$), $\Aut X$ is isomorphic to 
\begin{align*}
\left\{
    \begin{psmallmatrix}
    a&0&c\\
    0&1&0\\
    0&0&a^3
    \end{psmallmatrix} 
    \in \PGL(3, k) \middle| a\in k^*, c\in k
    \right\} \text{ (resp.\ } k^* \times \ZZ/2\ZZ, \mathrm{GL}(2, \FF_3)\text{)}.
\end{align*}
\end{cor}

\begin{proof}
We follow the notation of the proof of Proposition \ref{prop:p=3isom}.

Since every morphism from $Y$ to $X$ factors through the minimal resolution $\pi$, we have a canonical homomorphism $\varphi \colon \Aut X \to \Aut Y$ such that $\sigma \circ \pi =\pi \circ \varphi(\sigma)$ for all $\sigma \in \Aut X$. 
On the other hand, $\pi$ is the contraction of all the $(-2)$-curves on $Y$.
Since each automorphism of $Y$ fixes the union of $(-2)$-curves, we also have a canonical homomorphism $\psi \colon \Aut Y \to \Aut X$, which is the inverse of $\varphi$.
Hence $\Aut X \cong \Aut Y$.

First, suppose that $\Dyn(X)=E_8$.
By Type (1) of Figure \ref{fig:dual(1)-(3)}, each negative rational curve on $Y$ is $g(\Theta_{\infty, i})$ for some $0 \leq i \leq 8$.
The $\Aut Y$-action on $Y$ fixes the unique $(-1)$-curve $g(\Theta_{\infty, 0})$.
It also fixes $g(\Theta_{\infty, 1})$, which is the unique $(-2)$-curve intersecting with $g(\Theta_{\infty, 0})$.
By a similar argument, it fixes each negative rational curve.
Hence the $\Aut Y$-action descends to $\PP^2_k$ via $h'$.
In particular, $\Aut Y$ is contained in the subgroup $G$ of $\PGL(3, k) \cong \Aut \PP^2_k$ fixing $h'_*|-K_Y|$.
On the other hand, since $h_*|-K_Z|=h'_*|-K_Y|$ and $|-K_Z|$ are base point free, $Z$ is the minimal resolution of indeterminacy of $h'_*|-K_Y|$.
In particular, the $G$-action on $\PP^2_k$ ascends to $Z$.
Since $Z$ has a unique section, it descends to $Y$.
Therefore $\Aut Y \cong G$.
Since $h(F)=\{x^3+y^2z=0\}$ and $h(\Theta_{\infty, 8})=\{z=0\}$, we have $h'_*|-K_Y| = \{sz^3+t(x^3+y^2z)=0 \mid [s:t] \in \PP^1_k\}$.
Let
\begin{align*}
A=
    \begin{psmallmatrix}
    a&b&c\\
    d&e&f\\
    g&h&i
    \end{psmallmatrix} 
\end{align*}
be an element of $G$.
Since $\Aut Y$ fixes $h'(g(\Theta_{\infty, 8}))= h(\Theta_{\infty, 8})=\{z=0\}$, we have $g=h=0$ and $i \neq 0$.
On the other hand, we have
\begin{align*}
A \cdot (x^3+y^2z) = (a^3x^3+b^3y^3+c^3z^3)+(dx+ey+fz)^2(iz) \in h'_*|-K_Y|.
\end{align*}
Since the coefficients of $y^3$, $x^2z$, and $yz^2$ must be zero, we have $b=0$, $d=0$, $e \neq 0$, and $f=0$.
Since the coefficient of $x^3$ must coincide with that of $y^2z$, we have $a^3=e^2i$.
Fixing $e=1$, we obtain the assertion.

Next, suppose that $\Dyn(X)=A_2+E_6$.
By Type (2) of Figure \ref{fig:dual(1)-(3)}, each negative rational curve on $Y$ is $g(P)$, $g(2P)$, $g(\Theta_{0, i})$ for some $0 \leq i \leq 2$, or $g(\Theta_{\infty, i})$ for some $0 \leq i \leq 6$.
The $\Aut Y$-action on $Y$ fixes $g(\Theta_{0, 0})$, which is the unique $(-1)$-curve intersecting with one $(-1)$-curve and two $(-2)$-curves.
Then it also fixes $g(\Theta_{\infty, 2})$, $g(\Theta_{\infty, 1})$, and $g(\Theta_{\infty, 0})$.
On the other hand, there are exactly two $(-1)$-curves on $Y$ intersecting with no other $(-1)$-curves, which are $g(P)$ and $g(2P)$.
Then the $\Aut Y$-action on $Y$ fixes $g(P) \cup g(2P)$.
Similarly, it fixes $g(\Theta_{\infty, 4}) \cup g(\Theta_{\infty, 6})$, $g(\Theta_{\infty, 3}) \cup g(\Theta_{\infty, 5})$, and $g(\Theta_{0, 1}) \cup g(\Theta_{0, 2})$. 
Hence the $\Aut Y$-action descends to $\PP^2_k$ via $h'$.
In particular, the $\Aut Y$-action on $\PP^2_k$ fixes $h(O)=[0:0:1]$, $h(\Theta_{\infty,0})=\{x=0\}$, and $h(\Theta_{0,1}) \cup h(\Theta_{0,2})=\{z(y+z)=0\}$.
In particular, it fixes $h(\Theta_{0,1}) \cap h(\Theta_{0,2})=[1:0:0]$ and $h(\Theta_{\infty,0}) \cap (h(\Theta_{0,1}) \cup h(\Theta_{0,2}))=\{[0:1:0], [0:1:1]\}$.
On the other hand, by construction, every automorphism on $\PP^2_k$ ascends to $Y$ via $h'$ if they fix $[0:0:1]$, $[1:0:0]$, and $\{[0:1:0], [0:1:1]\}$.
Hence $\Aut Y$ is isomorphic to
\begin{align*}
    \left\{
    \begin{psmallmatrix}
    a&0&0\\
    0&1&0\\
    0&h&i
    \end{psmallmatrix} 
    \in \PGL(3, k) \middle| a \in k^*, (h,i)=(0,1) \text{ or } (1, -1)
    \right\} \cong k^* \times \ZZ/2\ZZ.
\end{align*}

Finally, suppose that $\Dyn(X)=4A_2$.
By Type (3) of Figure \ref{fig:dual(1)-(3)}, each $(-1)$-curve on $Y$ is either $g(\Theta_{i, 0})$ for some $i=0, -1, 1, \infty$, or the image of a section.
The former intersects with another $(-1)$-curve at $g(O)$ and the latter intersects with no other $(-1)$-curves.
Hence the $\Aut Y$-action on $Y$ fixes $g(O)$ and $E_{h'}$.
In particular, it descends to $\PP^2_k$ via $h'$ and fixes $h(O)=[0:0:1]$ and $h(E_h)$, which are $\FF_3$-rational points not contained in $\{z=0\}$.
On the other hand, by construction, every automorphism on $\PP^2_k$ ascends to $Y$ via $h'$ if they fix $[0:0:1]$ and $h(E_h)$.
Hence $\Aut Y$ is isomorphic to the subgroup of $\PGL(3, \FF_3)$ which fixes $\{z=0\}$ and $[0:0:1]$, which is $\mathrm{GL}(2, \FF_3)$.

Combining these arguments, we complete the proof.
\end{proof}

\subsection{Characteristic two}
In this subsection, we always assume that $p=2$.

First let us show that, when the degrees are two, Dynkin types determine the isomorphism classes of Du Val del Pezzo surfaces satisfying (NB).

\begin{prop}\label{E_7}
The minimal resolution of each del Pezzo surface of type $E_7$ satisfying (NB) is constructed from $\PP^2_{k, [x:y:z]}$ by blowing up at $[0:1:0]$ seven times along $\{x^3+y^2z=0\}$.
In particular, there is a unique del Pezzo surface of type $E_7$ satisfying (NB).
\end{prop}

\begin{proof}
We follow the notation of Type (a) of Figure \ref{fig:dual(a)-(e)}.
Let $Z$ be the rational quasi-elliptic surface of type (a) and $F$ a general fiber.
Let $g \colon Z \to Y$ be the contraction of $O$ and $\Theta_{\infty, 0}$ and $\pi \colon Y \to X$ the contraction of all $(-2)$-curves.
Then the desired del Pezzo surface must be $X$ by Lemma \ref{lem:deg=2pre} (5).
Take $h \colon Z \to \PP^2_k$ and coordinates of $\PP^2_k$ as in Lemma \ref{lem:Ito2Rem4}.
Let $h' \colon Y \to \PP^2_k$ be the morphism induced by $h$.
Then $h'$ is the blow-up at $h(O)=[0:1:0]$ seven times along $h(F)=\{x^3+y^2z=0\}$.
Hence it suffices to show that $X$ satisfies (NB).

Since $\pi$ and $h'$ is an isomorphism around a general member of $|-K_X|$, we are reduced to proving that $h'_*|-K_Y|$ has only a singular member.
By construction, $h'_*|-K_Y|$ consists of cubic curves intersecting with $h(F)=\{x^3+y^2z=0\}$ at $h(O)=[0:1:0]$ with multiplicity seven.
Then it is generated by $\{x^3+y^2z=0\}, \{z^3=0\}$, and $\{xz^2=0\}$. 
The Jacobian criterion now shows that $h'_*|-K_{Y'}|$ has only a singular member, and the assertion holds.
\end{proof}

\begin{cor}\label{E_7aut}
Let $X$ be the del Pezzo surface of type $E_7$ satisfying (NB) and $\pi \colon Y \to X$ the minimal resolution.
Then the following hold.
\begin{enumerate}
    \item[\textup{(1)}] $Y$ and each negative rational curve on $Y$ are defined over $\FF_2$.
    \item[\textup{(2)}] $\Aut X$ is isomorphic to
\begin{align*}
    \left\{
    \begin{psmallmatrix}
    a&0&d^2a\\
    d&1&f\\
    0&0&a^3
    \end{psmallmatrix} 
    \in \PGL(3, k) \middle| a \in k^*, d \in k, f \in k
    \right\}.
\end{align*}
\end{enumerate}
\end{cor}

\begin{proof}
We follow the notation of the proof of Proposition \ref{E_7}.

\noindent (1): 
By the construction of $h'$, $Y$ and each irreducible component of the exceptional divisor $E_{h'}$ of $h'$ are defined over $\FF_2$.
Since $Z$ is of type (a), Lemma \ref{lem:Ito2Rem4} shows that a negative rational curve on $Y$ is either a component of $E_{h'}$ or the strict transform of $h(\Theta_{\infty, 8}) = \{z=0\}$.
Hence the assertion holds.

\noindent (2): As in the proof of Corollary \ref{cor:deg=1p=3auto}, we have $\Aut X \cong \Aut Y$.
By Type $E_7$ of Figure \ref{fig:dualE_7-3A_1D_4}, the $\Aut Y$-action on $Y$ fixes the $(-1)$-curve and each $(-2)$-curve.
In particular, the $\Aut Y$-action naturally descends to $\PP^2_k$ via $h'$.
Hence $\Aut Y$ is contained the subgroup $G$ of $\PGL(3, k)$ which fixes the net $h'_*|-K_Y|=\{sz^3+t(xz^2)+u(x^3+y^2z)=0 \mid [s:t:u] \in \PP^2_k\}$.

On the other hand, $|-K_Y|$ is basepoint free by Lemma \ref{basic} (5).
Since $(-1)$-curves on $Y$ are of $(-K_Y)$-degree one, every blow-down of $Y$ collapses the basepoint freeness of $|-K_Y|$.
Hence $Y$ is the minimal resolution of indeterminacy of $h'_*|-K_Y|$.
In particular, we obtain $G \subset \Aut Y$.

Therefore $\Aut Y \cong G$. 
Let
\begin{align*}
A=
    \begin{psmallmatrix}
    a&b&c\\
    d&e&f\\
    g&h&i
    \end{psmallmatrix} 
\end{align*}
be an element of $G$.
Then 
\begin{align*}
&A \cdot z^3 \\
= &i^3z^3+gi^2xz^2+(g^3x^3+h^2iy^2z)+h^3y^3+g^2hx^2y+g^2ix^2z+h^2gxy^2+hi^2yz^2.
\end{align*}
Since the coefficient of $y^3$ must be zero, we have $h=0$.
Now the coefficient of $x^3$ also must be zero. Hence $g=0$ and $i \neq 0$.
Similarly,
\begin{align*}
A \cdot xz^2 = ci^2z^3+ai^2xz^2+bi^2yz^2
\end{align*}
implies that $b=0$ and
\begin{align*}
A \cdot (x^3+y^2z) = (c^3+f^2i)z^3+ac^2xz^2+(a^3x^3+e^2iy^2z)+(a^2c+d^2i)x^2z
\end{align*}
implies that $a^3=e^2i$ and $a^2c=d^2i$.
Fixing $e=1$, we obtain the assertion.
\end{proof}

\begin{prop}\label{A_1D_6}
The minimal resolution of each del Pezzo surface of type $A_1+D_6$ satisfying (NB) is constructed from $\PP^2_{k, [x:y:z]}$ by blowing up at $[0:1:0]$ five times along $\{x^3+y^2z=0\}$ and at $[1:1:1]$ twice along $\{x^3+y^2z=0\}$.
In particular, there is a unique del Pezzo surface of type $A_1+D_6$ satisfying (NB).
\end{prop}

\begin{proof}
We follow the notation of Type (b) of Figure \ref{fig:dual(a)-(e)}.
Let $Z$ be the rational quasi-elliptic surface of type (b) and $F$ a general fiber.
Let $g \colon Z \to Y$ be the contraction of $O$ and $\Theta_{\infty, 0}$ and $\pi \colon Y \to X$ the contraction of all $(-2)$-curves.
Then the desired del Pezzo surface must be $X$ by Lemma \ref{lem:deg=2pre} (6).
Take $h \colon Z \to \PP^2_k$ and coordinates of $\PP^2_k$ as in Lemma \ref{lem:Ito2Rem4}.
Let $h' \colon Y \to \PP^2_k$ be the morphism induced by $h$.
Then $h'$ is the composition of the blow-ups at $h(P)=[0:1:0]$ five times along $h(F)=\{x^3+y^2z=0\}$ and at $h(O)=[1:1:1]$ twice along $\{x^3+y^2z=0\}$.
Hence it suffices to show that $X$ satisfies (NB).

Since $\pi$ and $h'$ is an isomorphism around a general member of $|-K_X|$, it suffices to show that $h'_*|-K_Y|$ has only a singular member.
By construction, $h'_*|-K_Y|$ consists of cubic curves intersecting with $h(F)=\{x^3+y^2z=0\}$ at $h(P)=[0:1:0]$ five times and at $h(O)=[1:1:1]$ twice.
Then it is generated by $\{x^3+y^2z=0\}, \{(x+z)z^2=0\}$, and $\{(x+z)^2z=0\}$. 
The Jacobian criterion now shows that $h'_*|-K_{Y}|$ has only a singular member, and the assertion holds.
\end{proof}

\begin{cor}\label{A_1D_6aut}
Let $X$ be the del Pezzo surface of type $A_1+D_6$ satisfying (NB) and $Y$ the minimal resolution of $X$.
Then the following hold.
\begin{enumerate}
    \item[\textup{(1)}] $Y$ and each negative rational curve on $Y$ are defined over $\FF_2$.
    \item[\textup{(2)}] $\Aut X$ is isomorphic to
\begin{align*}
    \left\{
    \begin{psmallmatrix}
    a&0&a^3+a\\
    d&1&a^3+d+1\\
    0&0&a^3
    \end{psmallmatrix} 
    \in \PGL(3, k) \middle| a \in k^*, d \in k
    \right\}.
\end{align*}
    \item[\textup{(3)}] There is a birational morphism $h'_1 \colon Y \to \PP^1_{k, [x:y]} \times \PP^1_{k, [s:t]}$ such that each negative rational curve on $Y$ is either $h'_1$-exceptional or the strict transform of $\{x=0\}$, $\{y=0\}$, or $\{s=0\}$.
    Moreover, $h'_1$ is decomposed into six blow-ups at $\FF_2$-rational points.
\end{enumerate}
\end{cor}

\begin{proof}
We follow the notation of the proof of Proposition \ref{A_1D_6}.

\noindent (1): 
By the construction of $h'$, $Y$ and each irreducible component of $E_{h'}$ are defined over $\FF_2$.
Since $Z$ is of type (b), Lemma \ref{lem:Ito2Rem4} shows that a negative rational curve on $Y$ is either a component of $E_{h'}$ or the strict transform of $h(\Theta_{\infty, 8})=\{z=0\}$ or $h(\Theta_{\infty, 3})=\{x+z=0\}$.
Hence the assertion holds.

\noindent (2): Analysis similar to that in the proof of Corollary \ref{E_7aut} shows that $\Aut X \cong \Aut Y$ is the subgroup of $\PGL(3, k)$ which fixes $h'_*|-K_Y|=\{s((x+z)z^2)+t((x+z)^2z)+u(x^3+y^2z)=0 \mid [s:t:u] \in \PP^2_k\}$, and the assertion holds.

\noindent (3): Take $h_1 \colon Z \to \PP^1_{k} \times \PP^1_{k}$ as the contraction of $O$, $P$, and $\Theta_{\infty, i}$ for $i=0, 2, 3, 5, 6$, and $7$.
The induced morphism $h'_1 \colon Y \to \PP^1_{k} \times \PP^1_{k}$ satisfies the first assertion.
The second assertion follows from (1).
\end{proof}

\begin{prop}\label{3A_1D_4}
The minimal resolution of each del Pezzo surface of type $3A_1+D_4$ satisfying (NB) is constructed from $\PP^2_{k, [x:y:z]}$ by blowing up at $[1:0:0]$ once,
at $[0:0:1]$ twice along $\{y=0\}$, 
at $[0:1:1]$ twice along $\{y+z=0\}$, 
and at $[0:1:0]$ twice along $\{z=0\}$.
In particular, there is a unique del Pezzo surface of type $3A_1+D_4$ satisfying (NB).
\end{prop}

\begin{proof}
We follow the notation of Type (d) of Figure \ref{fig:dual(a)-(e)}.
Let $Z$ be a rational quasi-elliptic surface of type (d).
Let $g \colon Z \to Y$ be the contraction of $O$ and $\Theta_{0, 0}$ and $\pi \colon Y \to X$ the contraction of all $(-2)$-curves.
Then the desired del Pezzo surface must be $X$ by suitable choice of $Z$ by Lemma \ref{lem:deg=2pre} (7).
Hence it suffices to show that $X$ is independent of the choice of $Z$ and satisfies (NB).
Take $h \colon Z \to \PP^2_k$ and coordinates of $\PP^2_k$ as in Lemma \ref{lem:Ito2Rem4}.
Let $h' \colon Y \to \PP^2_k$ be the morphism induced by $h$.
Then $h'$ is the composition of the blow-ups at $h(\Theta_{0,4})=[1:0:0]$ once,
at $h(\Theta_{\infty, 1})=[0:0:1]$ twice along $h(\Theta_{0, 1})=\{y=0\}$, 
at $h(\Theta_{\infty, 2})=[0:1:1]$ twice along $h(\Theta_{0, 2})=\{y+z=0\}$, 
and at $h(\Theta_{\infty, 3})=[0:1:0]$ twice along $h(\Theta_{0, 3})=\{z=0\}$.
Hence it suffices to show that $X$ satisfies (NB).

Since $\pi$ and $h'$ is an isomorphism around a general member of $|-K_X|$, we are reduced to proving that $h'_*|-K_Y|$ has only a singular member.
By construction, $h'_*|-K_{Y}|$ consists of cubic curves intersecting with $h(\Theta_{0, i})$ at $h(\Theta_{\infty, i})$ with multiplicity two for $1 \leq i \leq 3$ and passing through $h(\Theta_{0,4})$.
Then it is generated by $\{x^2y=0\}, \{x^2z=0\}$, and $\{yz(y+z)=0\}$. 
The Jacobian criterion now shows that $h'_*|-K_{Y}|$ has only a singular member, and the assertion holds.
\end{proof}

\begin{cor}\label{3A_1D_4aut}
Let $X$ be the del Pezzo surface of type $3A_1+D_4$ satisfying (NB) and $Y$ the minimal resolution of $X$.
Then the following hold.
\begin{enumerate}
    \item[\textup{(1)}] $Y$ and each negative rational curve on $Y$ are defined over $\FF_2$.
    \item[\textup{(2)}] $\Aut X \cong k^* \times \PGL(2, \FF_2)$.
\end{enumerate}
\end{cor}

\begin{proof}
We follow the notation of the proof of Proposition \ref{3A_1D_4}.

\noindent (1): 
By the construction of $h'$, $Y$ and every irreducible component of $E_{h'}$ are defined over $\FF_2$.
Since $Z$ is of type (d), Lemma \ref{lem:Ito2Rem4} shows that
a negative rational curve on $Y$ is either $h'$-exceptional or the strict transform of one of $\{y=0\}$, $\{y+z=0\}$, $\{z=0\}$, or $h(\Theta_{\infty, 4}) = \{x=0\}$.
Hence the assertion holds.

\noindent (2): By the symmetry of Type $3A_1+D_4$ of Figure \ref{fig:dualE_7-3A_1D_4}, the $\Aut X \cong \Aut Y$-action on $Y$ naturally descends to $\PP^2_k$ via $h'$.
Hence $\Aut Y$ is isomorphic to
the subgroup of $\Aut \PP^2_k$ generated by automorphisms fixing $\{[0:1:0], [0:1:1], [0:0:1]\}$ and $[1:0:0]$, which is
\begin{align*}
    \left\{
    \begin{psmallmatrix}
    a&0&0\\
    0&e&f\\
    0&h&i
    \end{psmallmatrix} 
    \in \PGL(3, k) \middle| a \in k^*,
    \begin{psmallmatrix}
    e&f\\
    h&i
    \end{psmallmatrix} 
    \in \PGL(2, \FF_2)
    \right\} \cong k^* \times \PGL(2, \FF_2).
\end{align*}
\end{proof}

\begin{prop}\label{7A_1}
The minimal resolution of each del Pezzo surface of type $7A_1$ is constructed from $\PP^2_{k, [x:y:z]}$ by blowing up all the $\FF_2$-rational points.
In particular, there is a unique del Pezzo surface of type $7A_1$.
\end{prop}

\begin{proof}
We follow the notation of the proof of Lemma \ref{lem:typeg}.
Since \cite{Ye} shows that the desired surface satisfies (ND), it also satisfies (NB) by Theorem \ref{smooth, Intro}.

Let $Z$ be a rational quasi-elliptic surface of type (g).
Let $g \colon Z \to Y$ be the contraction of $A_{0,2}$ and $\Theta_{0,2}$ and $\pi \colon Y \to X$ the contraction of all $(-2)$-curves.
Then the desired del Pezzo surface must be $X$ by a suitable choice of $Z$ by Lemma \ref{lem:deg=2pre} (1) and (2).
Claim \ref{cl:typeg-2} in the proof of Lemma \ref{lem:typeg} now shows that the morphism $h' \colon Y \to \PP^2_k$ induced by $h \colon Z \to \PP^2_k$ is the blow-up of all the points in $\PP^2_k$ defined over $\FF_2$.
\end{proof}

\begin{rem}
Cascini-Tanaka \cite[Proposition 6.4]{CT18} proved that some del Pezzo surfaces constructed by Keel-M\textsuperscript{c}Kernan \cite[end of section 9]{KM} are isomorphic to the del Pezzo surface constructed by Langer \cite[Example 8.2]{Lan}. 
Proposition \ref{7A_1} gives another proof of this fact.
Moreover, Proposition \ref{7A_1} says that this surface is also isomorphic to a counterexample to the Akizuki-Nakano vanishing theorem in \cite[Proposition 11.1 (1)]{Gra} with $p=n=2$.
\end{rem}

\begin{cor}\label{7A_1aut}
Let $X$ be the del Pezzo surface of type $7A_1$ and $Y$ the minimal resolution of $X$.
Let $h' \colon Y \to \PP^2_k$ be the blow-up of all the $\FF_2$-rational points.
Then the following hold.
\begin{enumerate}
    \item[\textup{(1)}] $(-1)$-curves (resp.\ $(-2)$-curves) on $Y$ are $h'$-exceptional (resp.\ the strict transform of lines in $\PP^2_k$ are defined over $\FF_2$).
    In particular, $Y$ and every negative rational curve on $Y$ are defined over $\FF_2$.
    \item[\textup{(2)}] The class divisor group of $Y$ is generated by the seven $(-1)$-curves and any one of $(-2)$-curves.
    \item[\textup{(3)}] $\Aut X \cong \Aut Y \cong \PGL(3, \FF_2)$.
    \item[\textup{(4)}] $\Aut Y$ acts on both the set of $(-1)$-curves on $Y$ and that of $(-2)$-curves transitively.
    \item[\textup{(5)}] For each $(-1)$-curve $E$ on $Y$, the stabilizer subgroup of $\Aut Y$ with respect to $E$ is isomorphic to $\FF_2^2 \rtimes \PGL(2, \FF_2)$.
    The first (resp.\ second) factor acts on $E$ trivially (resp.\ as $\Aut \PP^1_{\FF_2}$).
\end{enumerate}
\end{cor}

\begin{proof}
\noindent (1): There are seven $h'$-exceptional curves and the strict transform of lines in $\PP^2_k$ defined over $\FF_2$, which are $(-1)$-curves and $(-2)$-curves respectively.
On the other hand, Lemma \ref{lem:deg=2pre} (4) shows that $Y$ contains exactly seven $(-1)$-curves and seven $(-2)$-curves.
Hence the assertion holds.

\noindent (2): The assertion is obvious from the assertion (1).

\noindent (3): By the assertion (1), the $\Aut Y$-action on $Y$ fixes $E_{h'}$ and descends to $\PP^2_k$ via $h'$.
Hence $\Aut Y$ equals the stabilizer subgroup of $\PGL(3, k)$ with respect to the set of $\FF_2$-rational points on $\PP^2_k$, which is $\PGL(3, \FF_2)$.

\noindent (4): The assertion is obvious from the assertion (3).

\noindent(5): Fix coordinates $[x:y:z]$ of $\PP^2_k$.
By the assertion (4), we may assume that $E$ is the strict transform of $\{x=0\} \subset \PP^2_k$.
Then the stabilizer subgroup of $\Aut Y$ with respect to $E$ is
\begin{align*}
    \left\{
    \begin{psmallmatrix}
    1&0&0\\
    d&e&f\\
    g&h&i
    \end{psmallmatrix} 
    \in \PGL(3, \FF_2)
    \right\} 
    &\cong    \left\{
    \begin{psmallmatrix}
    1&0&0\\
    d&1&0\\
    g&0&1
    \end{psmallmatrix} 
    \in \PGL(3, \FF_2)
    \right\} 
    \rtimes
        \left\{
    \begin{psmallmatrix}
    1&0&0\\
    0&e&f\\
    0&h&i
    \end{psmallmatrix} 
    \in \PGL(3, \FF_2)
    \right\} \\
    &\cong \FF_2^2 \rtimes \PGL(2, \FF_2),
\end{align*}
and the assertion holds.
\end{proof}

Next, we treat the case where the degree is one.

\begin{prop}\label{prop:deg1isom}
Let $X$ be a Du Val del Pezzo surface satisfying (NB).
Suppose that $p=2$ and $\Dyn(X)=E_8, D_8, A_1+E_7$, or $2A_1+D_6$.
Then the isomorphism class of $X$ is uniquely determined by $\Dyn(X)$.
\end{prop}

\begin{proof}
By Proposition \ref{prop:deg=1} (4), the minimal resolution of $X$ is obtained from the rational quasi-elliptic surface $Z$ of type (a), (b), (c), or (e) by contracting a section.
Since $Z$ is unique up to isomorphism for each types by Theorem \ref{q-ell}, the assertion follows from Proposition \ref{prop:deg=1} (2).
\end{proof}

\begin{lem}\label{lem:deg=1p=2auto}
Let $X$ be a Du Val del Pezzo surface satisfying (NB) and $\pi \colon Y \to X$ be the minimal resolution.
Suppose that $p=2$ and $\Dyn(X)=E_8$, $D_8$, or $A_1+E_7$ in addition.
Then the following hold.
\begin{enumerate}
    \item[\textup{(1)}] $Y$ and every negative rational curve on $Y$ are defined over $\FF_2$.
    \item[\textup{(2)}] $\Aut X$ is isomorphic to 
\begin{align*}
\left\{
    \begin{psmallmatrix}
    a&0&0\\
    0&1&f\\
    0&0&a^3
    \end{psmallmatrix} 
    \in \PGL(3, k) \middle| a\in k^*, f\in k
    \right\}
\end{align*}
when $\Dyn(X)=E_8$, 
\begin{align*}
\left\{
    \begin{psmallmatrix}
    1&0&0\\
    d&1&d\\
    0&0&1
    \end{psmallmatrix} 
    \in \PGL(3, k) \middle| d\in k
    \right\} \cong k
    \end{align*}
    when $\Dyn(X)=D_8$, and
\begin{align*}
\left\{
    \begin{psmallmatrix}
    1&0&0\\
    0&e&0\\
    0&0&e^2
    \end{psmallmatrix} 
    \in \PGL(3, k) \middle| e\in k^*
    \right\}
    \cong k^*
    \end{align*}
    when $\Dyn(X)=A_1+E_7$.
\end{enumerate}
\end{lem}

\begin{proof}
Let $g \colon Z \to Y$ be the blow-up at the base point of $|-K_Y|$.
When $\Dyn(X)=E_8$ (resp.~$D_8$, $A_1+E_7$), $Z$ is the rational quasi-elliptic surface of type (a) (resp.\ (b), (c)) by Proposition \ref{prop:deg=1} (4).
We may assume that $g$ is the contraction of $O$ in Types (a)--(c) of Figure \ref{fig:dual(a)-(e)} by virtue of the $\MW(Z)$-action on $Y$.
From now on, we follow the notation of Lemma \ref{lem:Ito2Rem4}.
Then $h \colon Z \to \PP^2_k$ induces a morphism $h' \colon Y \to \PP^2_k$.

\noindent (1): 
First, suppose that $\Dyn(X)=E_8$.
Then $h'$ is the blow-up of $\PP^2_k$ at $h(O)=[0:1:0]$ eight times along $h(F) = \{x^3+y^2z=0\}$.
Hence $Y$ and each irreducible component of $E_{h'}$ are defined over $\FF_2$.
Since each negative rational curve on $Y$ is either a component of $E_h$ or the strict transform of $h(\Theta_{\infty, 8}) = \{z=0\}$, the assertion holds.

Next, suppose that $\Dyn(X)=D_8$.
Then $h'$ is the composition of the blow-up of $\PP^2_k$ at $h(P)=[0:1:0]$ five times along $h(F) = \{x^3+y^2z=0\}$ and the blow-up at $h(O)=[1:1:1]$ three times along $\{x^3+y^2z=0\}$.
Hence $Y$ and each irreducible component of $E_{h'}$ are defined over $\FF_2$.
Since each negative rational curve on $Y$ is either a component of $E_{h'}$ or the strict transform of $h(\Theta_{\infty, 8})=\{z=0\}$ or $h(\Theta_{\infty, 3})=\{x+z=0\}$, the assertion holds.

Finally, suppose that  $\Dyn(X)=A_1+E_7$.
Then $h'$ is the composition of the blow-up of $\PP^2_k$ at $h(P)=[1:0:0]$ six times along $h(\Theta_{\infty, 1})=\{xz+y^2=0\}$ and the blow-up at $h(O)=[0:1:0]$ twice along $h(\Theta_{\infty, 0})=\{x=0\}$.
Hence $Y$ and each irreducible component of $E_{h'}$ are defined over $\FF_2$.
Since each negative rational curve on $Y$ is either a component of $E_{h'}$ or the strict transform of $\{xz+y^2=0\}$, $\{x=0\}$, or $h(\Theta_{0,2})=\{z=0\}$, the assertion holds.

\noindent (2): From Types (a)--(c) of Figure \ref{fig:dual(a)-(e)}, it is easily seen that the $\Aut Y$-action on $Y$ fixes each negative rational curve.
In particular, the $\Aut Y$-action naturally descends to $\PP^2_k$ via $h'$.
Hence $\Aut Y$ is contained in the subgroup $G$ of $\PGL(3, k)$ which fixes the net $h'_*|-K_Y|$.

On the other hand, we have $h'_*|-K_Y| = h_*|-K_Z|$.
Since $|-K_Z|$ is base point free, $Z$ is the minimal resolution of indeterminacy of $h'_*|-K_Y|$.
Hence the $G$-action on $\PP^2_k$ ascends to $Z$.
When $\Dyn(X)=E_8$, it descends to $Y$ since there is a unique section on $Z$.
On the other hand, when $\Dyn(X)=D_8$ or $A_1+E_7$, it also descends to $Y$ by the asymmetry of $E_{h}$.
Therefore $\Aut Y \cong G$.
By the choice of coordinates $[x:y:z]$ of $\PP^2_k$, $h'_*|-K_Y|$ is generated by $\{x^3+y^2z=0\}$ and $\{z^3=0\}$ (resp.\ $\{x^3+y^2z=0\}$ and $\{(x+z)^2z=0\}$, $\{(xz+y^2)x=0\}$ and $\{z^3=0\}$) when $\Dyn(X)=E_8$ (resp.\ $D_8$, $A_1+E_7$).
Hence an easy computation as in the proof of Corollary \ref{E_7aut} gives the assertion.
\end{proof}

\begin{lem}\label{2A_1D_6aut}
Let $X$ be the Du Val del Pezzo surface of type $2A_1+D_6$ satisfying (NB) and $\pi \colon Y \to X$ be the minimal resolution.
Then the following hold.
\begin{enumerate}
    \item[\textup{(1)}] $Y$ and every negative rational curve on $Y$ are defined over $\FF_2$.
    \item[\textup{(2)}] $\Aut X$ is isomorphic to 
\begin{align*}
\left\{
    \begin{psmallmatrix}
    1&0&0\\
    0&1&0\\
    0&0&1
    \end{psmallmatrix}
    , 
    \begin{psmallmatrix}
    1&0&0\\
    0&1&1\\
    0&0&1
    \end{psmallmatrix} 
    \in \PGL(3, k)
    \right\}
    \cong \ZZ/2\ZZ.
    \end{align*}
\end{enumerate}
\end{lem}

\begin{proof}
Since $\pi$ is the minimal resolution, we have $\Aut Y \cong \Aut X$.
Let $g \colon Z \to Y$ be the blow-up at the base point of $|-K_Y|$.
Then $Z$ be the rational quasi-elliptic surface of type (e) by Proposition \ref{prop:deg=1} (4).
In what follows, we use the notation of Type (e) of Figure \ref{fig:dual(a)-(e)}.
Then we may assume that $g$ is the contraction of $O$.
By the shape of Type (e) of Figure \ref{fig:dual(a)-(e)}, the $\Aut Y$-action on $Y$ fixes $g(\Theta_{1,2})$.
By Lemma \ref{lem:deg=2pre} (7), the contraction of $g(\Theta_{1,2})$ gives a morphism $h \colon Y \to W$ to the minimal resolution of the Du Val del Pezzo surface of type $3A_1+D_4$ satisfying (NB).
Hence $\Aut Y$ is isomorphic to the stabilizer subgroup of $\Aut W \cong k^* \times \PGL(2, \FF_2)$ with respect to $t=h \circ g(\Theta_{1,2})$.

\noindent (2): By Type (e) of Figure \ref{fig:dual(a)-(e)}, $E=h \circ g(\Theta_{1, 3})$ is the unique negative rational curve containing $t$. 
Hence the $\Aut Y$-action on $W$ fixes $E$.
Moreover $E$ is a $(-1)$-curve intersecting with exactly two $(-2)$-curves, which are $E_1 = h \circ g(\Theta_{1, 0})$ and $E_2 = h \circ g(\Theta_{1, 4})$.
We have seen in the proof of Corollary \ref{3A_1D_4aut} that the first factor
(resp.\ the second factor) of $\Aut W \cong k^* \times \PGL(2, \FF_2)$ acts on $E \setminus (E \cap (E_1 \cup E_2)) \cong k^*$ freely and transitively
(resp.\ acts as a permutation of the third nodes from the top in Type $3A_1+D_4$ of Figure \ref{fig:dualE_7-3A_1D_4}). 
Hence the assertion holds.

\noindent (1): By Corollary \ref{3A_1D_4aut}, $W$ and each negative rational curve on $W$ are defined over $\FF_2$.
By virtue of the $k^*$-action on $W$, we may assume that $t$ is an $\FF_2$-rational point.
Hence $Y$ and each negative rational curve on $Y$ except $g(\Theta_{0,0})$ and $g(\Theta_{\infty,0})$ are defined over $\FF_2$.
On the other hand, $g(\Theta_{0,0})$ and $g(\Theta_{\infty,0})$ are defined over $\FF_{2^m}$ for some $m>0$ since $Y$ is defined over $\overline{\FF_2}$, and are the unique $(-1)$-curves on $Y$ intersecting with $g(\Theta_{0,1})$ and $g(\Theta_{\infty,1})$ twice respectively.
Since the field extension $\FF_{2^m}/\FF_2$ is Galois, they are also defined over $\FF_2$.
Hence the assertion holds.
\end{proof}

To determine the isomorphism classes of Du Val del Pezzo surfaces of one of the types $2D_4$, $4A_1+D_4$, and $8A_1$ satisfying (NB), we need the following notation and auxiliary lemmas.

\begin{defn}
For coordinates of $\PP^n_k$, let $\mathcal{D}_n \subset \PP^n_k$ denote the complement of all the hyperplane sections defined over $\FF_2$.
Note that $\PGL (n+1, \FF_2)$ naturally acts on $\mathcal{D}_n$.
\end{defn}

\begin{lem}\label{lem:D1aut}
Let $\Sigma_t$ be the stabilizer subgroup of $\PGL(2, \FF_2)$ with respect to $t \in \mathcal{D}_1$.
Then the following hold.
\begin{enumerate}
    \item[\textup{(1)}] $\Sigma_t$ is trivial unless $t$ is an $\FF_4$-rational point.
    \item[\textup{(2)}] The $\PGL(2, \FF_2)$-action on the set $\mathcal{D}_1(\FF_4)$ of $\FF_4$-rational points on $\mathcal{D}_1$ is transitive.
    \item[\textup{(3)}] $\Sigma_t = \ZZ/3\ZZ$ if $t$ is an $\FF_4$-rational point.
    \item[\textup{(4)}] $\mathcal{D}_1 / \PGL(2, \FF_2) \cong \mathbb{A}^1_k$ with a distinct point which corresponds to $\mathcal{D}_1(\FF_4)$.
\end{enumerate}
\end{lem}

\begin{proof}
\noindent (1): 
Suppose that there is a non-trivial element $A \in \Sigma_t$.
Since $\PGL(2, \FF_2)$ is isomorphic to the symmetric group of three letters, it has exactly three conjugacy classes.
Since $A_1=
    \begin{psmallmatrix}
    0&1\\
    1&0\\
    \end{psmallmatrix}
$ and $A_2=
    \begin{psmallmatrix}
    0&1\\
    1&1\\
    \end{psmallmatrix}
$ are non-trivial and have different minimal polynomials, $A$ is conjugate to $A_i$ for some $i$.
Then $A_i$ also fixes some point in $\mathcal{D}_1$.

In $\PP^1_{k, [x:y]}$, the fixed point locus of $A_1$ (resp.\ $A_2$) equals $\{[1:1]\}$ (resp.\ $\{[1:s] \mid s^2+s+1=0\}$).
Hence $i=2$ and $t \in \mathcal{D}_1(\FF_4)$.

\noindent (2):
Since $A_1$ interchanges two points in $\mathcal{D}_1(\FF_4)$ with each other, the assertion holds.

\noindent (3):
Since the order of $\Sigma_t$ equals $|\PGL(2, \FF_2)|/|\mathcal{D}_1(\FF_4)|=3$, we obtain $\Sigma_t = \ZZ/3\ZZ$.

\noindent (4):
$\mathcal{D}_1 / \PGL(2, \FF_2)$ is naturally embedded into $\PP^1_k / \PGL(2, \FF_2) \cong \PP^1_k$.
The complement is a point since $\PGL(2, \FF_2)$ acts on $\PP^1_k(\FF_2)$ transitively.
\end{proof}

\begin{lem}\label{lem:D2aut}
Let $\Sigma_t$ be the stabilizer subgroup of $\PGL(3, \FF_2)$ with respect to $t \in \mathcal{D}_2$.
Then the following hold.
\begin{enumerate}
    \item[\textup{(1)}] $\Sigma_t$ is trivial unless $t$ is an $\FF_8$-rational point.
    \item[\textup{(2)}] The $\PGL(3, \FF_2)$-action on the set $\mathcal{D}_2(\FF_8)$ of $\FF_8$-rational points on $\mathcal{D}_2$ is transitive.
    \item[\textup{(3)}] $\Sigma_t = \ZZ/7\ZZ$ if $t$ is an $\FF_8$-rational point.
    \item[\textup{(4)}] $\mathcal{D}_2 / \PGL(3, \FF_2)$ is a surface with a unique singular point, which corresponds to $\mathcal{D}_2(\FF_8)$.
\end{enumerate}
\end{lem}

\begin{proof}
\noindent (1): Suppose that there is a non-trivial element $A \in \Sigma_t$.
By \cite[27.1 Lemma]{J-L}, $\PGL(3, \FF_2) \cong \mathrm{PSL}(2, \FF_7)$ has exactly six conjugacy classes.
Since
\begin{align*}
    A_1=
    \begin{psmallmatrix}
    1&1&0\\
    0&1&0\\
    0&0&1
    \end{psmallmatrix},
    A_2=
    \begin{psmallmatrix}
    1&1&0\\
    0&1&1\\
    0&0&1
    \end{psmallmatrix},
    A_3=
    \begin{psmallmatrix}
    1&0&0\\
    0&0&1\\
    0&1&1
    \end{psmallmatrix},
    A_4=
    \begin{psmallmatrix}
    0&1&0\\
    0&0&1\\
    1&1&0
    \end{psmallmatrix},
    \text{ and } 
    A_5=
    \begin{psmallmatrix}
    0&1&0\\
    0&0&1\\
    1&0&1
    \end{psmallmatrix}
\end{align*}
are non-trivial and have different minimal polynomials to each other, $A$ is conjugate to $A_i$ for some $1 \leq i \leq 5$.
Then $A_i$ also fixes some point in $\mathcal{D}_2$.

In $\PP^2_{k, [x:y:z]}$, the fixed point locus of $A_i$ equals
\begin{align*}
    \begin{cases}
\{y=0\} & (i=1) \\
\{[1:0:0]\} & (i=2)\\
\{[1:0:0]\} \cup \{[0:1:s] \mid s^2+s+1=0\} & (i=3) \\
\{[1:s:s^2] \mid s^3+s+1=0\} & (i=4)\\
\{[1:s:s^2] \mid s^3+s^2+1=0\} & (i=5)\\
    \end{cases}
\end{align*}
Hence $i=4$ or $5$, and $t \in \mathcal{D}_2(\FF_8)$.
We have proved more, namely that $A$ is of order seven and fixes exactly three points in $\mathcal{D}_2(\FF_8)$.
By \cite[27.1 Lemma]{J-L}, the size of its conjugacy class is 24.

\noindent (2): $\PP^1_k(\FF_8)$ (resp.\ $\PP^2_k(\FF_8)$) consists of nine (resp.\ 73) points.
Since $\PP^2_k \setminus \mathcal{D}_2$ is the union of seven $\PP^1_k$'s passing through three $\FF_2$-rational points, $\mathcal{D}_2(\FF_8)$ consists of $73-7 \cdot 9 + (3-1) \cdot 7 = 24$ points.
The Burnside lemma now shows that the number of the $\PGL(3, \FF_2)$-orbits is
\begin{align*}
    |\mathcal{D}_2(\FF_8)/\PGL(3, \FF_2)|=\frac{1}{168}(24 \cdot 3+24 \cdot 3+1 \cdot 24 + (168-24-24-1) \cdot 0)=1.
\end{align*}
Hence $\PGL(3, \FF_2)$ acts on $\mathcal{D}_2(\FF_8)$ transitively.

\noindent (3): Since the order of $\Sigma_t$ equals $|\PGL(3, \FF_2)|/|\mathcal{D}_2(\FF_8)|=7$, we obtain $\Sigma_t = \ZZ/7\ZZ$.

\noindent (4): The quotient morphism $\mathcal{D}_2 \to \mathcal{D}_2 / \PGL(3, \FF_2)$ is \'{e}tale outside the image of $\mathcal{D}_2(\FF_8)$.
Hence the assertion holds.
\end{proof}

Finally, let us investigate Du Val del Pezzo surfaces of one of types $2D_4$, $4A_1+D_4$, and $8A_1$ satisfying (NB).

\begin{prop}\label{2D_4isom}
Let $W$ be the minimal resolution of the Du Val del Pezzo surface of type $3A_1+D_4$ satisfying (NB) and $E$ the $(-1)$-curve intersecting with three $(-2)$-curves.
Note that $E$ is unique by Lemma \ref{lem:deg=2pre} (3) and $W$ and $E$ are defined over $\FF_2$ by Corollary \ref{3A_1D_4aut}.
Then the following holds.
\begin{enumerate}
    \item[\textup{(1)}] The minimal resolution of each Du Val del Pezzo surface of type $2D_4$ satisfying (NB) is obtained from $W$ by blowing up a point in $E \setminus E(\FF_2) \cong \mathcal{D}_1$.
    \item[\textup{(2)}] Let $h_t \colon Y_{t} \to W$ be the blow-up at $t \in E \setminus E(\FF_2)$.
    Then $Y_t$ is the minimal resolution of a Du Val del Pezzo surface of type $2D_4$ satisfying (NB).
    Moreover, for $t' \in E \setminus E(\FF_2)$,  $Y_{t} \cong Y_{t'}$ if and only if $t'$ is contained in the $\PGL(2, \FF_2)$-orbit of $t$.
\end{enumerate}
As a result, there is a one-to-one correspondence between the isomorphism classes of del Pezzo surfaces of type $2D_4$ satisfying (NB) and the closed points of $\mathcal{D}_1 /\PGL (2, \FF_2)$.
\end{prop}

\begin{proof}
We follow the notation of Type (d) of Figure \ref{fig:dual(a)-(e)}. 

\noindent(1): Let $Z$ be a rational quasi-elliptic surface of type (d).
Let $g \colon Z \to Y$ be the contraction of $O$. 
Then the minimal resolution of each del Pezzo surface of type $2D_4$ satisfying (NB) is isomorphic to $Y$ by suitable choice of $Z$ by Proposition \ref{prop:deg=1} (4).
On the other hand, by Lemma \ref{lem:deg=2pre} (7), the contraction of $g(\Theta_{0,0})$ gives a morphism $h \colon Y \to W$.
We check at once that $E=h \circ g(\Theta_{0, 4})$ and it contains the point $t=h \circ g(\Theta_{0,0})$, which is not contained in any $(-2)$-curves.
By Corollary \ref{3A_1D_4aut}, the set of $\FF_2$-rational points on $E$ is the intersection of $E$ and all $(-2)$-curves.
Therefore $Y$ is the blow-up of $W$ at $t \in E \setminus E(\FF_2)$.

\noindent (2): 
By Lemma \ref{lem:deg=2pre} (7), $Y_t$ is obtained from a rational quasi-elliptic surface of type (d) by contracting a section.
Hence the former assertion follows from Proposition \ref{prop:deg=1} (4).

We have seen in the proof of Corollary \ref{3A_1D_4aut} that first (resp.\ second) factor of $\Aut W=k^* \times \PGL(2, \FF_2)$ acts on $h(\Theta_{\infty, 4})$ trivially (resp.\ as $\Aut \PP^1_{\FF_2}$).
The same conclusion can be drawn for $E$ by the choice of its coordinates. 
In particular, $t' \in E \setminus E(\FF_2)$ is contained in the $\PGL(2, \FF_2)$-orbit of $t$ if and only if it is contained in the $\Aut W$-orbit of $t$ in $W$.
On the other hand, $Y_t \cong Y_{t'}$ if $t'$ is contained in the $\Aut W$-orbit of $t$.
Hence it remains to prove that $t'$ is contained in the $\Aut W$-orbit of $t$ if $Y_t \cong Y_{t'}$. 

Suppose that there is an isomorphism $\sigma \colon Y_t \cong Y_{t'}$.
Since the involution as in Lemma \ref{lem:type(d)auto} fixes the section $O$, it descends to an involution $\tau \in \Aut Y_t$.
Then, replacing $\sigma$ with $\sigma \circ \tau$ if necessary, we may assume that $\sigma(E_{h_t}) = E_{h_{t'}}$.
Hence $\sigma$ descends to an isomorphism $\overline \sigma \in \Aut W$ such that $\overline \sigma(t)=t'$, and the latter assertion holds.
\end{proof}

\begin{cor}\label{2D_4aut}
Let $X_t$ be the contraction of all $(-2)$-curves in $Y_t$ as in Proposition \ref{2D_4isom}. 
Then $\Aut X_t \cong \Aut Y_t \cong (k^* \times \ZZ/3\ZZ) \rtimes \ZZ/2\ZZ$ when $t$ is an $\FF_4$-rational point of $\mathcal{D}_1$ and $k^* \rtimes \ZZ/2\ZZ$ otherwise.
In particular, there is a unique Du Val del Pezzo surface $X(2D_4)$ satisfying (NB) such that $\Aut X \cong (k^* \times \ZZ/3\ZZ) \rtimes \ZZ/2\ZZ$. 
\end{cor}

\begin{proof}
We follow the notation of Proposition \ref{2D_4isom}.
Let $\Sigma$ be the stabilizer subgroup of $\Aut W$ with respect to $t$.
Then $\Sigma = k^* \times \Sigma'$ for some $\Sigma' \subset \PGL(2, \FF_2)$ since $k^*$ acts on $E$ trivially.
By Lemma \ref{lem:D1aut}, $\Sigma = k^* \times \ZZ/3\ZZ$ if $t \in \mathcal{D}_1(\FF_4)$ and $\Sigma = k^*$ otherwise.
On the other hand, we can identify $\Sigma$ with the stabilizer subgroup of $\Aut Y_t$ with respect to $E_{h_t}$.
For $\eta \in \Aut Y_t$, either $\eta$ or $\eta \circ \tau$ belongs to $\Sigma$.
Hence $\Aut Y_t \cong \Sigma \rtimes \ZZ/2\ZZ$, where the last factor is generated by $\tau$, and the first assertion holds.
Since $\PGL(2, \FF_2)$ acts on $\mathcal{D}_1(\FF_4)$ transitively, the second assertion follows from Proposition \ref{2D_4isom} (2).
\end{proof}

\begin{prop}\label{4A_1D_4isom}
Let $W$ be the minimal resolution of the Du Val del Pezzo surface of type $7A_1$ and $E$ a $(-1)$-curve.
Note that $W$ and $E$ are defined over $\FF_2$ and $E$ is unique up to the $\Aut W$-action on $W$ by Corollary \ref{7A_1aut} (1) and (4).
Then the following hold.
\begin{enumerate}
    \item[\textup{(1)}] The minimal resolution of each Du Val del Pezzo surface of type $4A_1+D_4$ is obtained from $W$ by blowing up a point in $E \setminus E(\FF_2) \cong \mathcal{D}_1$.
    \item[\textup{(2)}] Let $h_t \colon Y_{t} \to W$ be the blow-up at $t \in E \setminus E(\FF_2)$.
    Then $Y_t$ is the minimal resolution of a Du Val del Pezzo surface of type $4A_1+D_4$.
    Moreover, for $t' \in E \setminus E(\FF_2)$,  $Y_{t} \cong Y_{t'}$ if and only if $t'$ is contained in the $\PGL(2, \FF_2)$-orbit of $t$.
\end{enumerate}
As a result, there is a one-to-one correspondence between the isomorphism classes of del Pezzo surfaces of type $4A_1+D_4$ and the closed points of $\mathcal{D}_1 /\PGL (2, \FF_2)$.
\end{prop}

\begin{proof}
We follow the notation of Figure \ref{fig:dual(f)}. 
Note that, since \cite{Ye} shows that $4A_1+D_4$ is not feasible over $\C$, every Du Val del Pezzo surface of type $4A_1+D_4$ satisfies (NB) by Theorem \ref{smooth, Intro}.

\noindent(1): Let $Z$ be a rational quasi-elliptic surface of type (f).
Let $g \colon Z \to Y$ 
be the contraction of $O$.
Then the minimal resolution of each del Pezzo surface of type $4A_1+D_4$ is isomorphic to $Y$ by suitable choice of $Z$ by Proposition \ref{prop:deg=1} (4).
On the other hand, by Lemma \ref{lem:deg=2pre} (8), the contraction of $g(\Theta_{1,0})$ gives a morphism $h \colon Y \to W$.
We may assume that $E=h \circ g(\Theta_{1, 4})$.
Then $E$ contains the point $t=h \circ g(\Theta_{1,0})$, which is not contained in any $(-2)$-curves.
By Corollary \ref{7A_1aut} (1), the set of $\FF_2$-rational points on $E$ is the intersection of $E$ and all $(-2)$-curves.
Therefore $Y$ is the blow-up of $W$ at $t \in E \setminus E(\FF_2)$.

\noindent (2): 
By Lemma \ref{lem:deg=2pre} (8), $Y_t$ is obtained from a rational quasi-elliptic surface of type (f) by contracting a section.
Hence the former assertion follows from Proposition \ref{prop:deg=1} (4).

Let $C_t$ be the strict transform of $E$ in $Y_t$, which is a $(-2)$-curve.
By Figure \ref{matrix7A1}, $C_t$ intersects with three $(-2)$-curves in $Y_t$.
Hence $C_t$ is the central curve of the Dynkin diagram $D_4$.
In particular, every automorphism of $Y_t$ fixes $C_t$. 

By Corollary \ref{7A_1aut} (5), $t' \in E \setminus E(\FF_2)$ is contained in the $\PGL(2, \FF_2)$-orbit of $t$ if and only if it is contained in the $\Aut W=\FF_2^2 \rtimes \PGL(2, \FF_2)$-orbit of $t$ in $W$.
On the other hand, $Y_t \cong Y_{t'}$ if $t'$ is contained in the $\Aut W$-orbit of $t$.
Hence it remains to prove that $t'$ is contained in the $\Aut W$-orbit of $t$ if $Y_t \cong Y_{t'}$. 

Suppose that there is an isomorphism $\sigma \colon Y_t \cong Y_{t'}$.
Then $\sigma(C_t) = C_{t'}$.
By Figure \ref{fig:dual(f)}, $E_{h_t}$ is the unique $(-1)$-curve intersecting with $C_t$.
Hence $\sigma(E_{h_t})=E_{h_{t'}}$ and $\sigma$ descends to an isomorphism $\overline{\sigma} \in \Aut W$ such that $\overline{\sigma}(t)=t'$, and the latter assertion holds.
\end{proof}

\begin{cor}\label{4A_1D_4aut}
Let $X_t$ be the contraction of all $(-2)$-curves in $Y_t$ as in Proposition \ref{4A_1D_4isom}. 
Then $\Aut X_t \cong \Aut Y_t \cong (\ZZ/2\ZZ)^2\rtimes \ZZ/3\ZZ$ when $t$ is an $\FF_4$-rational point and $(\ZZ/2\ZZ)^2$ otherwise.
In particular, there is a unique Du Val del Pezzo surface $X(4A_1+D_4)$ such that $\Aut X \cong (\ZZ/2\ZZ)^2\rtimes \ZZ/3\ZZ$. 
\end{cor}

\begin{proof}
We follow the notation of Proposition \ref{4A_1D_4isom}.
Since each automorphism of $Y_t$ fixes $E_{h_t}$, the group $\Aut Y_t$ equals the stabilizer subgroup $\Sigma$ of $\Aut W = \FF_2^2 \rtimes \PGL(2, \FF_2)$ with respect to $t$.
Then $(\ZZ/2\ZZ)^2 \cong \FF_2^2 \subset \Aut Y_t$ since $\FF_2^2$ acts on $E$ trivially.
The rest of the proof runs as in Corollary \ref{2D_4aut}.
\end{proof}

\begin{prop}\label{8A_1isom}
Let $W$ be the minimal resolution of the Du Val del Pezzo surface of type $7A_1$ and $B$ the union of all negative rational curves on $W$.
Note that $W$ and $B$ are defined over $\FF_2$ and $W \setminus B \cong \mathcal{D}_2$ by Corollary \ref{7A_1aut} (1).
Then the following hold.
\begin{enumerate}
    \item[\textup{(1)}] The minimal resolution of each Du Val del Pezzo surface of type $8A_1$ is obtained from $W$ by blowing up a point in $W \setminus B$.
    \item[\textup{(2)}] Let $h_t \colon Y_{t} \to W$ be the blow-up at $t \in W \setminus B$.
    Then $Y_t$ is the minimal resolution of a Du Val del Pezzo surface of type $8A_1$.
    \item[\textup{(3)}] For $t \in W \setminus B$, Figure \ref{matrix} is the intersection matrix of negative rational curves on $Y_t$. 
    Moreover, there is a $(-2)$-curve $C_t$ such that $E_{h_t}$ is a unique $(-1)$-curve intersecting with $C_t$ twice.
    \item[\textup{(4)}] For $t \in W \setminus B$, $\Aut Y_t$ is contained in the affine linear group $\FF_2^3 \rtimes \mathrm{GL}(3, \FF_2)$ and contains its normal subgroup $\FF_2^3$, which acts on the set of $(-2)$-curves transitively.
    \item[\textup{(5)}] For $t$ and $t' \in W \setminus B$,  $Y_{t} \cong Y_{t'}$ if and only if $t'$ is contained in the $\Aut W \cong \PGL(3, \FF_2)$-orbit of $t$.
\end{enumerate}
As a result, there is a one-to-one correspondence between the isomorphism classes of del Pezzo surfaces of type $8A_1$ and the closed points of $\mathcal{D}_2 /\PGL (3, \FF_2)$.
\end{prop}

\begin{figure}[htbp]
    \centering
\begin{align*}
{\fontsize{7pt}{4}\selectfont
    \left(
    \vcenter{
        \xymatrix@=-1.8ex{
&&&&&&&&\ar@{-}[ddddddddddddddddddddddddddddddddddd]&&&&&&&& \ar@{-}[ddddddddddddddddddddddddddddddddddd] &&&&&&&&& \ar@{-}[ddddddddddddddddddddddddddddddddddd] &&&&&&&&\\
&&&&&&&&&&&&&&&&&&&&&&&&&&&&&&&&&\\
&-1&0 &0 &0 &0 &0 &0 & &1 &0 &0 &0 &0 &0 &0 &&1 &1 &1 &1 &0 &0 &0 &0 &&0 &0 &0 &0 &1 &1 &1 &1 &\\
&  &-1&0 &0 &0 &0 &0 & &0 &1 &0 &0 &0 &0 &0 &&1 &1 &0 &0 &1 &1 &0 &0 &&0 &0 &1 &1 &0 &0 &1 &1 &\\
&  &  &-1&0 &0 &0 &0 & &0 &0 &1 &0 &0 &0 &0 &&1 &0 &1 &0 &1 &0 &1 &0 &&0 &1 &0 &1 &0 &1 &0 &1 &\\
&  &  &  &-1&0 &0 &0 & &0 &0 &0 &1 &0 &0 &0 &&1 &1 &0 &0 &0 &0 &1 &1 &&0 &0 &1 &1 &1 &1 &0 &0 &\\
&  &  &  &  &-1&0 &0 & &0 &0 &0 &0 &1 &0 &0 &&1 &0 &0 &1 &1 &0 &0 &1 &&0 &1 &1 &0 &0 &1 &1 &0 &\\
&  &  &  &  &  &-1&0 & &0 &0 &0 &0 &0 &1 &0 &&1 &0 &1 &0 &0 &1 &0 &1 &&0 &1 &0 &1 &1 &0 &1 &0 &\\
&  &  &  &  &  &  &-1& &0 &0 &0 &0 &0 &0 &1 &&1 &0 &0 &1 &0 &1 &1 &0 &&0 &1 &1 &0 &1 &0 &0 &1 &\\
\ar@{-}[rrrrrrrrrrrrrrrrrrrrrrrrrrrrrrrrrr]&&&&&&&&          &&&&&&&&            &&&&&&&&&&            &&&&&&&&\\
&  &  &  &  &  &  &  & &-1&0 &0 &0 &0 &0 &0 &&0 &0 &0 &0 &1 &1 &1 &1 &&1 &1 &1 &1 &0 &0 &0 &0 &\\
&  &  &  &  &  &  &  & &  &-1&0 &0 &0 &0 &0 &&0 &0 &1 &1 &0 &0 &1 &1 &&1 &1 &0 &0 &1 &1 &0 &0 &\\
&  &  &  &  &  &  &  & &  &  &-1&0 &0 &0 &0 &&0 &1 &0 &1 &0 &1 &0 &1 &&1 &0 &1 &0 &1 &0 &1 &0 &\\
&  &  &  &  &  &  &  & &  &  &  &-1&0 &0 &0 &&0 &0 &1 &1 &1 &1 &0 &0 &&1 &1 &0 &0 &0 &0 &1 &1 &\\
&  &  &  &  &  &  &  & &  &  &  &  &-1&0 &0 &&0 &1 &1 &0 &0 &1 &1 &0 &&1 &0 &0 &1 &1 &0 &0 &1 &\\
&  &  &  &  &  &  &  & &  &  &  &  &  &-1&0 &&0 &1 &0 &1 &1 &0 &1 &0 &&1 &0 &1 &0 &0 &1 &0 &1 &\\
&  &  &  &  &  &  &  & &  &  &  &  &  &  &-1&&0 &1 &1 &0 &1 &0 &0 &1 &&1 &0 &0 &1 &0 &1 &1 &0 &\\
\ar@{-}[rrrrrrrrrrrrrrrrrrrrrrrrrrrrrrrrrr]&&&&&&&&          &&&&&&&&            &&&&&&&&&&            &&&&&&&&\\
&  &  &  &  &  &  &  & &  &  &  &  &  &  &  &&-2&0 &0 &0 &0 &0 &0 &0 &&2 &0 &0 &0 &0 &0 &0 &0 &\\
&  &  &  &  &  &  &  & &  &  &  &  &  &  &  &&  &-2&0 &0 &0 &0 &0 &0 &&0 &2 &0 &0 &0 &0 &0 &0 &\\
&  &  &  &  &  &  &  & &  &  &  &  &  &  &  &&  &  &-2&0 &0 &0 &0 &0 &&0 &0 &2 &0 &0 &0 &0 &0 &\\
&  &  &  &  &  &  &  & &  &  &  &  &  &  &  &&  &  &  &-2&0 &0 &0 &0 &&0 &0 &0 &2 &0 &0 &0 &0 &\\
&  &  &  &  &  &  &  & &  &  &  &  &  &  &  &&  &  &  &  &-2&0 &0 &0 &&0 &0 &0 &0 &2 &0 &0 &0 &\\
&  &  &  &  &  &  &  & &  &  &  &  &  &  &  &&  &  &  &  &  &-2&0 &0 &&0 &0 &0 &0 &0 &2 &0 &0 &\\
&  &  &  &  &  &  &  & &  &  &  &  &  &  &  &&  &  &  &  &  &  &-2&0 &&0 &0 &0 &0 &0 &0 &2 &0 &\\
&  &  &  &  &  &  &  & &  &  &  &  &  &  &  &&  &  &  &  &  &  &  &-2&&0 &0 &0 &0 &0 &0 &0 &2 &\\
\ar@{-}[rrrrrrrrrrrrrrrrrrrrrrrrrrrrrrrrrr]&&&&&&&&          &&&&&&&&            &&&&&&&&&&            &&&&&&&&\\
&  &  &  &  &  &  &  & &  &  &  &  &  &  &  &&  &  &  &  &  &  &  &  &&-1&1 &1 &1 &1 &1 &1 &1 &\\
&  &  &  &  &  &  &  & &  &  &  &  &  &  &  &&  &  &  &  &  &  &  &  &&  &-1&1 &1 &1 &1 &1 &1 &\\
&  &  &  &  &  &  &  & &  &  &  &  &  &  &  &&  &  &  &  &  &  &  &  &&  &  &-1&1 &1 &1 &1 &1 &\\
&  &  &  &  &  &  &  & &  &  &  &  &  &  &  &&  &  &  &  &  &  &  &  &&  &  &  &-1&1 &1 &1 &1 &\\
&  &  &  &  &  &  &  & &  &  &  &  &  &  &  &&  &  &  &  &  &  &  &  &&  &  &  &  &-1&1 &1 &1 &\\
&  &  &  &  &  &  &  & &  &  &  &  &  &  &  &&  &  &  &  &  &  &  &  &&  &  &  &  &  &-1&1 &1 &\\
&  &  &  &  &  &  &  & &  &  &  &  &  &  &  &&  &  &  &  &  &  &  &  &&  &  &  &  &  &  &-1&1 &\\
&  &  &  &  &  &  &  & &  &  &  &  &  &  &  &&  &  &  &  &  &  &  &  &&  &  &  &  &  &  &  &-1&\\
&&&&&&&&          &&&&&&&&            &&&&&&&&&            &&&&&&&&&
} 
}\right) 
}.
\end{align*}
    \caption{The intersection matrix of the $(-1)$-curves and $(-2)$-curves in a del Pezzo surface of type $8A_1$}
    \label{matrix}
\end{figure}

\begin{proof}
We follow the notation of the proof of Lemma \ref{lem:typeg}.
Note that, since \cite{Ye} shows that $8A_1$ is not feasible over $\C$, every Du Val del Pezzo surface of type $8A_1$ satisfies (NB) by Theorem \ref{smooth, Intro}.

\noindent(1): Let $Z$ be a rational quasi-elliptic surface of type (g).
Let $g \colon Z \to Y$ be the contraction of $A_{0,2}$.
Then the minimal resolution of each del Pezzo surface of type $8A_1$ is isomorphic to $Y$ by suitable choice of $Z$ by Proposition \ref{prop:deg=1} (4).
On the other hand, by Lemma \ref{lem:deg=2pre} (1) and (2), the contraction of $g(\Theta_{0,2})$ gives a morphism $h \colon Y \to W$ such that $t = h \circ g(\Theta_{0,2}) \in W \setminus B$.
Therefore $Y$ is the blow-up of $W$ at $t \in W \setminus B$.

\noindent (2): 
By Lemma \ref{lem:deg=2pre} (1) and (2), $Y_t$ is obtained from a rational quasi-elliptic surface $Z_t$ of type (g) by contracting a section.
Hence the assertion follows from Proposition \ref{prop:deg=1} (4).

\noindent (3): 
By Lemma \ref{lem:typeg}, $Z_t$ has exactly sixteen $(-1)$-curves $A_{0, 1}, \ldots, A_{7,1}$, $A_{0,2}, \ldots, A_{7, 2}$ and exactly sixteen $(-2)$-curves $\Theta_{0,1}, \ldots, \Theta_{7,1}$, $\Theta_{0,2}, \ldots, \Theta_{7,2}$, whose intersection matrix is Figure \ref{matrix(g)}.
We may assume that the contraction of $A_{0,2}$ gives a morphism $g_t \colon Z_t \to Y_t$ and $E_{h_t}=g_t(\Theta_{0,2})$.
Then $g_t(A_{0,1})$ is a $(0)$-curve and $A'_{i,j} \coloneqq g_t(A_{i,j})$ is a $(-1)$-curve for $1 \leq i \leq 7$ and $j=1,2$.
Moreover, $\Theta'_{i,1} \coloneqq g_t(\Theta_{i,1})$ and $\Theta'_{i,2} \coloneqq g_t(\Theta_{i,2})$ is a $(-2)$-curve and a $(-1)$-curve respectively for $0 \leq i \leq 7$.
Hence Figure \ref{matrix} is the intersection matrix of $A'_{1, 1}, \ldots, A'_{7,1}$, $A'_{1,2}, \ldots, A'_{7, 2}$, $\Theta'_{0,1}, \ldots, \Theta'_{7,1}$, $\Theta'_{0,2}, \ldots, \Theta'_{7,2}$ in this order.
Moreover, $E_{h_t}=\Theta'_{0,2}$ is a unique $(-1)$-curve intersecting with $C_t = \Theta'_{0,1}$ twice.

\noindent (4): 
Suppose that an automorphism of $Y_t$ fixes each $(-1)$-curve and each $(-2)$-curve.
Then it fixes $A'_{1, 1}, \ldots, A'_{7,1}$, and $\Theta'_{0,2}$.
By Claim \ref{cl:typeg-2} in the proof of Lemma \ref{lem:typeg}, it descends to an automorphism of $\PP^2_k$ fixing all the $\FF_2$-rational points, which is the identity.
Thus an automorphism of $Y_t$ is determined by the image of all $(-1)$-curves and $(-2)$-curves.

Let $S_8$ be the permutation group of $\{0, 1, \ldots, 7\}$.
By Figure \ref{matrix}, the images of all $(-1)$-curves are determined by those of $(-2)$-curves $\Theta'_{0,1}, \ldots, \Theta'_{7,1}$.
Hence there is an injection $\iota \colon \Aut Y_t \to S_8$ which sends $\eta \in \Aut Y_t$ to $\sigma \in S_8$ such that $\eta(\Theta'_{i,1})=\Theta'_{\sigma(i),1}$ for $0 \leq i \leq 7$.
Moreover, $(0^8)$, $(1^8)$, and fourteen rows in the $(1, 3)$ block or $(2,3)$ block of Figure \ref{matrix} form the $[8,4,4]$ extended Hamming code, which is also the Reed-Muller code $R(1,3)$.
Hence $\iota$ factors through the automorphism group of $R(1,3)$, which is the affine linear group $\FF_2^3 \rtimes \mathrm{GL}(3, \FF_2) \subset S_8$ by \cite[Chapter 13, \S 9, Theorem 24]{Mac-Sloane}.
Since the normal group $\FF_2^3 \subset S_8$ is generated by $(01)(23)(45)(67)$, $(02)(13)(46)(57)$ and $(04)(15)(26)(37)$, it acts on $\{0, 1, \ldots, 7\}$ transitively.
Hence it suffices to show that $\FF_2^3 \subset \Aut Y_t$.
We show only the existence of $\eta \in \Aut Y_t$ such that $\iota(\eta) = (01)(23)(45)(67)$; the same proof works for $(02)(13)(46)(57)$ and $(04)(15)(26)(37)$.

Let $\varphi \colon Y_t \to \overline{Y}$ be the contraction of $A'_{1,1}$, $A'_{2,1}$, and $A'_{4,1}$.
Set $s_i=\varphi(A'_{i,1})$ for $i=1,2,$ and $4$.
Then $\overline{Y}$ is a smooth del Pezzo surface of degree four since each $(-2)$-curve in $Y_t$ intersects with $A'_{1,1}$, $A'_{2,1}$, or $A'_{4,1}$.
Generally speaking, a smooth del Pezzo surface of degree four contains sixteen $(-1)$-curves, and each $(-1)$-curve intersects with five $(-1)$-curves.
In the present case, $\overline{A}_{i,1}=\varphi(A'_{i,1})$, $\overline{A}_{i,2}=\varphi(A'_{i,2})$, $\overline{\Theta}_{j,1}=\varphi(\Theta'_{j,1})$, and $\overline{\Theta}_{k,1}=\varphi(\Theta'_{k,1})$ are $(-1)$-curves on $Y$ for $i=3,5,6$, or $7$, $2 \leq j \leq 7$, and $k=0,1$.
Figure \ref{matrixdeg5} is the intersection matrix of $\overline{A}_{3,1}, \ldots, \overline{A}_{7,1}$, $\overline{A}_{3,2}, \ldots, \overline{A}_{7,2}$, $\overline{\Theta}_{2,1}, \ldots, \overline{\Theta}_{7,1}$, $\overline{\Theta}_{0,2}$, and $\overline{\Theta}_{1,2}$ in this order. 
In particular, $\mathcal{M}=(\overline{\Theta}_{1,2}, \overline{A}_{3,2}, \overline{A}_{5,2}, \overline{A}_{6,2}, \overline{A}_{7,2})$ and $\mathcal{M'}=(\overline{\Theta}_{0,2}, \overline{A}_{3,1}, \overline{A}_{5,1}, \overline{A}_{6,1}, \overline{A}_{7,1})$ are the 5-tuples of $(-1)$-curves intersecting with $\overline{\Theta}_{0,2}$ and $\overline{\Theta}_{1,2}$ respectively.

Since $\mathcal{M}$ and $\mathcal{M}'$ satisfy the condition (1) of \cite[Theorem 2.1]{Hos}, there is an automorphism $\overline{\eta}$ of $\overline{Y}$ which interchanges $\mathcal{M}$ with $\mathcal{M'}$.
By Figure \ref{matrixdeg5}, $\overline{\eta}$ also interchanges $\overline{\Theta}_{2,1}$ with $\overline{\Theta}_{3,1}$, $\overline{\Theta}_{4,1}$ with $\overline{\Theta}_{5,1}$, and $\overline{\Theta}_{6,1}$ with $\overline{\Theta}_{7,1}$.
Then $\overline{\eta}$ fixes $s_1 = \overline{\Theta}_{2,1} \cap \overline{\Theta}_{3,1}$, $s_2 = \overline{\Theta}_{4,1} \cap \overline{\Theta}_{5,1}$, and $s_4 = \overline{\Theta}_{6,1} \cap \overline{\Theta}_{7,1}$.
Hence $\overline{\eta}$ induces an automorphism $\eta \in \Aut Y_t$, which interchanges $\Theta'_{0,2}$ with $\Theta'_{1,2}$, $\Theta'_{2,1}$ with $\Theta'_{3,1}$, $\Theta'_{4,1}$ with $\Theta'_{5,1}$, and $\Theta'_{6,1}$ with $\Theta'_{7,1}$.
Since $\Theta'_{0,1}$ (resp.\ $\Theta'_{1,1}$) is the unique $(-2)$-curve which intersects with $\Theta'_{0,2}$ (resp.\ $\Theta'_{1,2}$) twice, $\eta$ also interchanges $\Theta'_{0,1}$ with $\Theta'_{1,1}$.
Hence $\iota(\eta) = (01)(23)(45)(67)$, and the assertion holds.

\begin{figure}[t]
    \centering
\begin{align*}
{\fontsize{7pt}{4}\selectfont
    \left(
    \vcenter{
        \xymatrix@=-1.8ex{
&&&&&\ar@{-}[ddddddddddddddddddddd]&&&&& \ar@{-}[ddddddddddddddddddddd] &&&&&&& \ar@{-}[ddddddddddddddddddddd]&&\\
&&&&&          &&&&&            &&&&&&&&            &&\\
&-1&0 &0 &0 &&1 &0 &0 &0 &&1 &0 &1 &0 &1 &0 &&0 &1 &\\
&  &-1&0 &0 &&0 &1 &0 &0 &&0 &1 &1 &0 &0 &1 &&0 &1 &\\
&  &  &-1&0 &&0 &0 &1 &0 &&1 &0 &0 &1 &0 &1 &&0 &1 &\\
&  &  &  &-1&&0 &0 &0 &1 &&0 &1 &0 &1 &1 &0 &&0 &1 &\\
\ar@{-}[rrrrrrrrrrrrrrrrrrrr]&&&&&          &&&&&            &&&&&&&&            &&\\
&  &  &  &  &&-1&0 &0 &0 &&0 &1 &0 &1 &0 &1 &&1 &0 &\\
&  &  &  &  &&  &-1&0 &0 &&1 &0 &0 &1 &1 &0 &&1 &0 &\\
&  &  &  &  &&  &  &-1&0 &&0 &1 &1 &0 &1 &0 &&1 &0 &\\
&  &  &  &  &&  &  &  &-1&&1 &0 &1 &0 &0 &1 &&1 &0 &\\
\ar@{-}[rrrrrrrrrrrrrrrrrrrr]&&&&&          &&&&&            &&&&&&&&            &&\\
&  &  &  &  &&  &  &  &  &&-1&1 &0 &0 &0 &0 &&0 &0 &\\
&  &  &  &  &&  &  &  &  &&  &-1&0 &0 &0 &0 &&0 &0 &\\
&  &  &  &  &&  &  &  &  &&  &  &-1&1 &0 &0 &&0 &0 &\\
&  &  &  &  &&  &  &  &  &&  &  &  &-1&0 &0 &&0 &0 &\\
&  &  &  &  &&  &  &  &  &&  &  &  &  &-1&1 &&0 &0 &\\
&  &  &  &  &&  &  &  &  &&  &  &  &  &  &-1&&0 &0 &\\
\ar@{-}[rrrrrrrrrrrrrrrrrrrr]&&&&&          &&&&&            &&&&&&&&            &&\\
&  &  &  &  &&  &  &  &  &&  &  &  &  &  &  &&-1&1 &\\
&  &  &  &  &&  &  &  &  &&  &  &  &  &  &  &&  &-1&\\
&&&&&          &&&&&            &&&&&&&&            &&\\
} 
}\right) 
}.
\end{align*}
    \caption{The intersection matrix of the $(-1)$-curves in $\overline{Y}$}
    \label{matrixdeg5}
\end{figure}

\noindent (5):
If some automorphism of $W$ sends $t$ to $t'$, then it ascends to an isomorphism $Y_t \cong Y_{t'}$.
On the other hand, suppose that there is an isomorphism $Y_{t} \cong Y_{t'}$.
By the assertion (4), we may assume that this isomorphism sends $C_t$ to $C_{t'}$.
Then by the assertion (3), it also sends $E_{h_t}$ to $E_{h_{t'}}$.
Hence it descends to an isomorphism of $W$, which sends $t$ to $t'$.
\end{proof}

\begin{cor}\label{8A_1aut}
Let $X_t$ be the contraction of all $(-2)$-curves in $Y_t$ as in Proposition \ref{8A_1isom}. 
Then $\Aut X_t \cong \Aut Y_t \cong (\ZZ/2\ZZ)^3 \rtimes \ZZ/7\ZZ$ when $t$ is an $\FF_8$-rational point and $(\ZZ/2\ZZ)^3$ otherwise.
In particular, there is a unique Du Val del Pezzo surface $X(8A_1)$ such that $\Aut X \cong (\ZZ/2\ZZ)^3 \rtimes \ZZ/7\ZZ$.
\end{cor}

\begin{proof}
We follow the notation of Proposition \ref{8A_1isom}.
Let $\Sigma \subset \Aut Y_t$ be the stabilizer subgroup with respect to $C_t$. 
Since $\FF_2^3 \cong (\ZZ/2\ZZ)^3$ is a normal subgroup of $\Aut Y_t$ which acts on the set of $(-2)$-curves in $Y_t$ transitively, we obtain $\Aut Y_t \cong (\ZZ/2\ZZ)^3 \rtimes \Sigma$.
By Proposition \ref{8A_1isom} (3), $\Sigma$ is the same as the stabilizer subgroup of $\PGL(3, \FF_2)$ with respect to $t \in \mathcal{D}_2$.
Now the first assertion follows from Lemma \ref{lem:D2aut}.
Since $\PGL(3, \FF_2)$ acts on $\mathcal{D}_2(\FF_8)$ transitively, the second assertion follows from Proposition \ref{8A_1isom} (5).
\end{proof}

\begin{cor}\label{Itorem}
There are one-to-one correspondences between the isomorphism classes of rational quasi-elliptic surfaces of type (d), (f), and (g), and the closed points of $\mathcal{D}_1 /\PGL (2, \FF_2)$, $\mathcal{D}_1 /\PGL (2, \FF_2)$, and $\mathcal{D}_2 /\PGL (3, \FF_2)$ respectively.
\end{cor}

\begin{proof}
By Proposition \ref{prop:deg=1}, there is one-to-one correspondence between isomorphism classes of del Pezzo surfaces of type $2D_4$ satisfying (NB) (resp.\ type $4A_1+D_4$, type $8A_1$) and those of rational quasi-elliptic surfaces of type (d) (resp.\ (f), (g)).
Hence the assertion follows from Propositions \ref{2D_4isom}, \ref{4A_1D_4isom} and \ref{8A_1isom}.
\end{proof}

Now we can prove Theorem \ref{sing}.

\begin{proof}[Proof of Theorem \ref{sing}]
The assertions (0), (1), and (2) follow from Lemma \ref{basic}, Proposition \ref{prop:deg=1}, and Propositions \ref{sep} and \ref{insep} respectively.
The assertion (3) follows from Propositions \ref{prop:p=3isom}, \ref{E_7}, \ref{A_1D_6}, \ref{3A_1D_4}, \ref{7A_1}, \ref{prop:deg1isom}, \ref{2D_4isom}, \ref{4A_1D_4isom}, and \ref{8A_1isom}.
\end{proof}

\subsection{List of automorphism groups}
As a consequence, we obtain the list of automorphisms of Du Val del Pezzo surfaces satisfying (NB) and rational quasi-elliptic surfaces as follows.

\begin{thm}\label{auto}
Let $X$ be a Du Val del Pezzo surface satisfying (NB).
Then $\Aut X$ is described in Table \ref{table:auto}.
Furthermore, suppose that $p=2$.
Then for each of types $2D_4$, $4A_1+D_4$, and $8A_1$, there is a unique del Pezzo surface of the given type such that the group $G$ in Table \ref{table:auto} is non-trivial.
    \begin{table}[htbp]
\caption{}
  \begin{tabular}{|c|c|c|} \hline
       $\Dyn(X)$ & Characteristic & Automorphism groups \\ \hline \hline
       \multirow{2}{*}{$E_8$} &$p=2$        &     
     $\left\{
    \begin{psmallmatrix}
    a&0&0\\
    0&1&f\\
    0&0&a^3
    \end{psmallmatrix} 
    \in \PGL(3, k) \middle| a\in k^*, f\in k
    \right\}$
    \\   \cline{2-3}
                         &\multirow{3}{*}{$p=3$}      &
         $\left\{
    \begin{psmallmatrix}
    a&0&c\\
    0&1&0\\
    0&0&a^3
    \end{psmallmatrix} 
    \in \PGL(3, k) \middle| a\in k^*, c\in k
    \right\}$
    \\ \cline{1-1}  \cline{3-3}   
       $A_2+E_6$         &               &$k^* \times \ZZ/2\ZZ$       \\ \cline{1-1}  \cline{3-3}
       $4A_2$            &               &$\mathrm{GL}(2, \FF_3)$       \\ \hline
       $D_8$             &\multirow{10}{*}{$p=2$}     &$k$    \\ \cline{1-1}  \cline{3-3}                  
       $A_1+E_7$         &               &$k^*$\\ \cline{1-1}  \cline{3-3}                    
       $2D_4$            &               &$(k^* \times G) \rtimes \ZZ/2\ZZ$ with $G=\{1\}$ or $\ZZ/3\ZZ$       \\ \cline{1-1}  \cline{3-3}
       $2A_1+D_6$        &               &$\ZZ/2\ZZ$       \\ \cline{1-1}  \cline{3-3}
       $4A_1+D_4$        &               &$(\ZZ/2\ZZ)^2 \rtimes G$ with $G=\{1\}$ or $\ZZ/3\ZZ$      \\ \cline{1-1}  \cline{3-3}
       $8A_1$            &               &$(\ZZ/2\ZZ)^3 \rtimes G$ with $G=\{1\}$ or $\ZZ/7\ZZ$       \\ \cline{1-1}  \cline{3-3}
       $E_7$             &               &$\left\{
    \begin{psmallmatrix}
    a&0&d^2a\\
    d&1&f\\
    0&0&a^3
    \end{psmallmatrix} 
    \in \PGL(3, k) \middle| a \in k^*, d \in k, f \in k
    \right\}$       \\ \cline{1-1}  \cline{3-3}
       $A_1+D_6$         &               &
       $    \left\{
    \begin{psmallmatrix}
    a&0&a^3+a\\
    d&1&a^3+d+1\\
    0&0&a^3
    \end{psmallmatrix} 
    \in \PGL(3, k) \middle| a \in k^*, d \in k
    \right\}$\\ \cline{1-1}  \cline{3-3}
       $3A_1+D_4$        &               &$k^* \times \PGL(2, \FF_2)$       \\ \cline{1-1}  \cline{3-3}
       $7A_1$            &               &$\PGL(3, \FF_2)$       \\ \hline
  \end{tabular}
  \label{table:auto}
\end{table}
\end{thm}

\begin{proof}
The assertion follows from Corollaries \ref{cor:deg=1p=3auto}, \ref{E_7aut}, \ref{A_1D_6aut}, \ref{3A_1D_4aut}, \ref{7A_1aut}, Lemmas \ref{lem:deg=1p=2auto}, \ref{2A_1D_6aut}, Corollaries \ref{2D_4aut}, \ref{4A_1D_4aut}, and \ref{8A_1aut}.
\end{proof}

\begin{cor}\label{cor:q-ellaut}
Let $Z$ be a rational quasi-elliptic surface and $O \subset Z$ a section.
Take $g \colon Z \to Y$ as the contraction of $O$ and $\pi \colon Y \to X$ the contraction of all the $(-2)$-curves.
Then $\Aut Z \cong \MW (Z) \cdot \Aut X$.
In particular, $\Aut Z \cong (\ZZ/p\ZZ)^n \cdot H$ for some $0 \leq n \leq 4$ and for some group $H$ listed in Table \ref{table:auto}.
\end{cor}

\begin{proof}
Note that $X$ is a Du Val del Pezzo surface satisfying (NB) by Proposition \ref{prop:deg=1} (1) and $\Aut Y \cong \Aut X$ since $\pi$ is the minimal resolution.
Since $h(O)$ is the base point of $|-K_Y|$, $h$ induces an isomorphism between $\Aut Y$ and the stabilizer subgroup of  $\Aut Z$ with respect to $O$.
Hence the first assertion follows from the transitivity of the $\MW (Z)$-action on the set of sections on $Z$.
The second assertion follows from Theorems \ref{thm:q-ell3}, \ref{q-ell}, and \ref{auto}.
\end{proof}

\begin{rem}\label{autorem}
We follow the notation in Corollary \ref{cor:q-ellaut}.
We have described the reduced scheme structure of $\Aut Y$ and $\Aut Z$.
We can also describe the scheme structure of them by virtue of \cite[Main Theorem]{M-S}, which calculates the identity component of $\Aut Y$ as a scheme.

On the other hand, what is still lacking is the determination of the scheme structure of $\Aut X$ since the contraction of $(-2)$-curves may thicken the scheme structures of the automorphism groups.
For example, smooth K3 surfaces in characteristic $p>0$ admit no non-trivial $\mu_p$-actions but RDP K3 surfaces may admit such actions (see \cite[Remark 2.3]{Mat17}).
\end{rem}

\section{Log liftability}
\label{sec:singliftable}

In this section, we determine all the Du Val del Pezzo surfaces which are not log liftable over $W(k)$. 
Note that by Theorem \ref{smooth, Intro} (1), it suffices to consider Du Val del Pezzo surfaces satisfying (NB).

\begin{prop}\label{p=3liftable}
Let $X$ be a Du Val del Pezzo surface satisfying (NB) and $\pi \colon Y \to X$ the minimal resolution. 
Suppose that $p=3$ and $\Dyn(X)=E_8$ or $A_2+E_6$.
Then the pair of $(Y, E_\pi)$ lifts to $\Spec \Z$ via $\Spec \FF_3 \to \Spec \Z$.
As a result, $X$ is log liftable both over $\Z$ via $\Spec \FF_3 \to \Spec \Z$ and over $W(k)$.
\end{prop}

\begin{proof}
Note that $Y$ and each $(-2)$-curve on $Y$ are defined over $\FF_3$ by Proposition \ref{prop:p=3isom} (6).
Suppose that $\Dyn(X)=E_8$.
Take a birational morphism $h'_{\ZZ} \colon Y_{\ZZ} \to \PP^2_{\ZZ}$ as the blow-up at $[0:1:0]$ eight times along $\{x^3+y^2z=0\}$.
By Proposition \ref{prop:p=3isom} (3), we have $Y \cong Y_{\ZZ} \otimes_\ZZ \FF_3$ and each negative rational curve on $Y$ is the specialization of either an $h'_{\ZZ}$-exceptional curve or the strict transform of $\{z=0\} \subset \PP^2_{\ZZ}$ via $h'_{\ZZ}$.
Hence we obtain the desired lift.
The proof for the case where $\Dyn(X)=A_2+E_6$ is similar by virtue of Proposition \ref{prop:p=3isom} (4).
\end{proof}

\begin{prop}\label{p=2liftable}
Let $X$ be a Du Val del Pezzo surface satisfying (NB) and $\pi \colon Y \to X$ the minimal resolution. 
Suppose that $p=2$.
Then the following hold.
\begin{enumerate}
    \item[\textup{(1)}] Suppose that $\Dyn(X)=E_7$, $A_1+D_6$, or $3A_1+D_4$. 
    Then the log smooth pair of $Y$ and the union $B$ of negative rational curves lifts to $\Spec \Z$ via $\Spec \FF_2 \to \Spec \Z$.
    \item[\textup{(2)}] Suppose that $\Dyn(X)=E_8$, $D_8$, $A_1+E_7$, or $2A_1+D_6$.
    Then the pair $(Y, E_\pi)$ lifts to $\Spec \Z$ via $\Spec \FF_2 \to \Spec \Z$.
\end{enumerate}
As a result, $X$ is log liftable both over $\Z$ via $\Spec \FF_2 \to \Spec \Z$ and over $W(k)$.
\end{prop}

\begin{proof}
By Theorem \ref{sing}, $X$ is uniquely determined up to isomorphism by $\Dyn(X)$.
Moreover, we have shown in \S \ref{sec:singisom} that $Y$ and each negative rational curve on $Y$ are defined over $\FF_2$.

\noindent (1): 
By Lemma \ref{lem:deg=2pre} (3), the pair $(Y, B)$ is log smooth.
Now suppose that $\Dyn(X)=E_7$.
Take a birational morphism $h'_{\ZZ} \colon Y_{\ZZ} \to \PP^2_{\ZZ}$ as the blow-up at $[0:1:0]$ seven times along $\{x^3+y^2z=0\}$.
By Proposition \ref{E_7}, we have $Y \cong Y_{\ZZ} \otimes_\ZZ \FF_2$ and each negative rational curve on $Y$ is the specialization of either an $h'_{\ZZ}$-exceptional curve or the strict transform of $\{z=0\} \subset \PP^2_{\ZZ}$ via $h'_{\ZZ}$.
Hence we obtain the desired lift.
The proof for the cases where $\Dyn(X)=A_1+D_6$ and $3A_1+D_4$ is similar by virtue of Corollary \ref{A_1D_6aut} (3) and Proposition \ref{3A_1D_4} respectively.

\noindent (2): 
By Proposition \ref{prop:deg=1} (4) and Lemma \ref{lem:deg=2pre}
(5)--(7), for some Du Val del Pezzo surface of type $E_7$, $A_1+D_6$, or $3A_1+D_4$ satisfying (NB) and its minimal resolution $W$, there exist a $(-1)$-curve $E \subset W$ and an $\FF_2$-rational point $t \in E$ not contained in any $(-2)$-curves such that $Y$ is the blow-up of $W$ at $t$.
Hence the assertion follows from the assertion (1) and \cite[Proposition 2.9]{ABL}.
\end{proof}

By Proposition \ref{2D_4isom}, there are infinitely many Du Val del Pezzo surfaces of type $2D_4$ satisfying (NB).
In particular, they are not defined over $\FF_2$ in general.
On the other hand, we can show their log liftability over $W(k)$ as follows.

\begin{prop}
Let $X$ be a Du Val del Pezzo surface of type $2D_4$ satisfying (NB) in $p=2$.
Take $R$ as a Noetherian irreducible ring with surjective ring homomorphism $f \colon R \rightarrow k$.
Then $X$ is log liftable over $R$ via the induced morphism $\Spec k \to \Spec R$.
\end{prop}

\begin{proof}
Let $\pi \colon Y \to X$ be the minimal resolution.
By Proposition \ref{2D_4isom} (1), on the minimal resolution $W$ of the Du Val del Pezzo surface of type $3A_1+D_4$ satisfying (NB), there are the $(-1)$-curve $E \subset W$ intersecting with exactly three $(-2)$-curves and a closed point $t \in E$ not contained in any $(-2)$-curves such that $Y$ is the blow-up of $W$ at $t$.

Let $D$ be the union of the $(-2)$-curves in $W$.
By Proposition \ref{p=2liftable} (1), the log smooth pair $(W, D \cup E)$ lifts to $\Spec \ZZ$ via $\Spec \FF_2 \to \Spec \ZZ$.
Take $(\mathcal{W}, \mathcal{D} \cup \mathcal{E})$ as the base change of such a lifting by the natural homomorphism $\ZZ \to R$.

Fix coordinates $[x:y]$ of $\mathcal{E} \cong \PP^1_R$ and choose $a, b \in k$ so that 
$t=[a:b] \in \PP^1_{k, [x:y]}$. 
Since $f\colon R\to k$ is surjective, we can take a lifting $\tilde{a}$ (resp.~$\tilde{b}$) of $a$ (resp.~$b$).
Then $\tilde{t}= [\tilde{a}:\tilde{b}]\in \mathcal{E} \cong \PP^1_R$ is a lifting of $t$.
Let $\Phi\colon \mathcal{Y} \to \mathcal{W}$ be the blow-up along $\tilde{t}$.
Then $(\mathcal{Y}, \Phi^{-1}_{*}(\mathcal{D} \cup \mathcal{E}))$ is the desired lift.
\end{proof}

\begin{prop}\label{ND}
Let $X$ be a Du Val del Pezzo surface with $\Dyn(X)=4A_1+D_4$, $8A_1$, or $7A_1$.
Then $X$ is not log liftable over any Noetherian integral domain $R$ of characteristic zero via any morphism $\Spec k \to \Spec R$ induced by a surjective homomorphism $R \to k$.
\end{prop}

\begin{proof}
By \cite[Theorem 1.2]{Ye}, the surface $X$ satisfies (ND).
Hence the assertion follows from Proposition \ref{NDtoNL}.
\end{proof}

\begin{prop}\label{4A_2lift}
Let $X$ be a Du Val del Pezzo surface of type $4A_2$ in $p=3$.
Then $X$ is not log liftable over $W(k)$.
\end{prop}

\begin{proof}
We note that $X$ satisfies (NB) by Proposition \ref{4A_2NB}. 
Suppose by contradiction that $X$ is log liftable over $W(k)$.
Take $\pi \colon Y \to X$ as the minimal resolution and $(\mathcal{Y}, \mathcal{E})$ as a $W(k)$-lifting of $(Y, E_{\pi})$.
We follow the notation used in the proof of Proposition \ref{NDtoNL}.
Then the blow-up $Z_K \to Y_K$ at the base point of $|-K_{Y_K}|$ gives the anti-canonical morphism $f_K \colon Z_K \to \PP^1_K$.
Let $G$ be the strict transform of $E_K=\sum_{i=1}^8 E_{i, K}$ in $Z_K$.
Then $f_K(G)$ consists of four $K$-rational points.
We fix coordinates of $\PP^1_K$ such that $f(G)=\{0, 1, \infty, \alpha\}$ for some $\alpha \in \PP^1_K \setminus \{0,1,\infty\}$.

On the other hand, by Proposition \ref{NDtoNL}, $X_{\C}$ is the del Pezzo surface of degree one of type $4A_2$.
By \cite[Table 4.1]{Ye}, the blow-up $Z_{\C} \to Y_{\C}$ at the base point of $|-K_{Y_{\C}}|$ gives an elliptic fibration $f_{\C} \colon Z_{\C} \to \PP^1_{\C}$ with four singular fibers of type $\textup{I}_3$. 
Since $f_K(G) \subset \PP^1_{\C}$ is the singular fiber locus of $f_{\C}$, \cite[Th\'{e}or\`{e}me]{Bea} now yields the existence of $\sigma \in \Aut \PP^1_{\C}$ which sends $f_K(G)$ to $\{1, \omega, \omega^2, \infty\}$, where $\omega$ is a primitive cube root of unity.
An easy computation shows that $\alpha=-\omega$ and hence $\omega \in K$.
However, by the Eisenstein criterion and the Gauss lemma, the cyclotomic polynomial $t^2+t+1$ is irreducible in $K[t]$, a contradiction.
Therefore $(Y, E_{\pi})$ does not lift to $W(k)$.
\end{proof}

\begin{rem}
One question still unanswered is whether $X(4A_2)$ in $p=3$ is log liftable over any Noetherian integral domain of characteristic zero.
\end{rem}

\begin{rem}\label{remliftable}
As we saw in the proof of Propositions \ref{7A_1}, \ref{4A_1D_4isom}, and \ref{8A_1isom} (resp.\ Proposition \ref{prop:p=3isom} (4)), the surfaces as in Proposition \ref{ND} (resp.\ Proposition \ref{4A_2lift}) are obtained from the configuration of all the lines in $\PP^2_k$ defined over $\FF_p$, which is not realizable in $\PP_{\C}^2$ by the Hirzebruch inequality for line arrangements (see \cite{Hir} and \cite[Example 3.2.2]{Miy}).
This is the reason why we cannot apply the proof of Proposition \ref{p=3liftable} for such surfaces.
\end{rem}

\section{Kodaira type vanishing theorem}\label{sec:KV}
In this section, we determine all the Du Val del Pezzo surfaces which violate the Kodaira vanishing theorem for ample $\Z$-divisors. Note that by Theorem \ref{smooth, Intro} (3), it suffices to consider Du Val del Pezzo surfaces satisfying (NL).

\begin{lem}[\textup{cf.~\cite[Theorem 4.8]{Kaw2}}]\label{lem:KV-1}
Let $X$ be a Du Val del Pezzo surface and $A$ an ample $\Z$-divisor on $X$. If $H^1(X, \sO_X(-A))\neq 0$, then $p=2$ and $(-K_X \cdot A)=1$. 
\end{lem}
\begin{proof}
We refer to the proof of \cite[Theorem 4.8]{Kaw2} for the details.
\end{proof}

\begin{prop}\label{KV:8A_1}
Let $X$ be a del Pezzo surface of type $8A_1$.
Then there is an ample $\Z$-divisor $A$ such that $H^1(X, \sO_X(-A)) \neq 0$.
\end{prop}

\begin{proof}
We follow the notation of the proof of Proposition \ref{8A_1isom}.
Let $\pi \colon Y \to X$ be the minimal resolution and
$A \coloneqq \pi_*(A'_{1,1}+A'_{2,1}-A'_{4,1})$. 
Then $A$ is ample since $\rho(X)=1$ and $(-K_X \cdot A)=1$.
By Figure \ref{matrix}, we have
\begin{align*}
&\lceil \pi^* A \rceil\\
=&\lceil A'_{1,1}+\frac12(\Theta'_{0,1}+\Theta'_{1,1}+\Theta'_{2,1}+\Theta'_{3,1})+A'_{2,1}+\frac12(\Theta'_{0,1}+\Theta'_{1,1}+\Theta'_{4,1}+\Theta'_{5,1})\\
&-A'_{4,1}-\frac12(\Theta'_{0,1}+\Theta'_{1,1}+\Theta'_{6,1}+\Theta'_{7,1}) \rceil \\
=&A'_{1,1} + A'_{2,1} - A'_{4,1} + \Theta'_{0,1} + \Theta'_{1,1} + \Theta'_{2,1} + \Theta'_{3,1} + \Theta'_{4,1} + \Theta'_{5,1}
\end{align*}
In particular, $\lceil \pi^* A \rceil^2=-3$ and $(-K_Y \cdot \lceil \pi^* A \rceil)=1$. Lemma \ref{lem:KV-2} now yields $H^{i}(Y, \sO_Y(-\lceil \pi^* A \rceil))=H^{i}(X, \sO_X(-A))$ for $i \geq 0$.
Since $\lceil \pi^* A \rceil$ is big, we have $H^0(Y, \sO_Y(-\lceil \pi^* A \rceil))=0$.

Next assume that $H^2(Y, \sO_Y(-\lceil \pi^* A \rceil)) \neq 0$.
Then there is an effective divisor $C \sim K_Y+\lceil \pi^* A \rceil$ by the Serre duality.
Since $(-K_Y \cdot C)=0$, the curve $C$ is a sum of $(-2)$-curves.
Since $(-2)$-curves in $Y$ are disjoint from each other, we have $(C \cdot \Theta'_{0,1}) \in 2 \ZZ$.
However, $(C \cdot \Theta'_{0,1})=(\lceil \pi^* A \rceil \cdot \Theta'_{0,1})=-1$ by Figure \ref{matrix}, a contradiction.

Combining these results and the Riemann-Roch theorem, we conclude that
\begin{align*}
\dim_k H^1(X, \sO_X(-A))
&=\dim_k H^1(Y, \sO_Y(-\lceil \pi^* A \rceil))\\
&=-\chi(Y, \sO_Y(-\lceil \pi^* A \rceil))\\
&=-(\chi(Y, \sO_Y)+\frac12((-\lceil \pi^* A \rceil)^2+(-K_Y \cdot -\lceil \pi^* A \rceil)))=1.
\end{align*}
Therefore $H^1(X, \sO_X(-A)) \neq 0$.
\end{proof}

\begin{prop}\label{KV:4A_1+D_4}
Let $X$ be a del Pezzo surface of type $4A_1+D_4$.
Then $H^1(X, \sO_X(-A))=0$ for any ample $\Z$-divisor $A$.
\end{prop}

\begin{proof}
Let $\pi \colon Y \to X$ be the minimal resolution.
By Proposition \ref{prop:deg=1} (4), there exist a rational quasi-elliptic surface $Z$ of type (f) and a section $O$ such that the contraction of $O$ gives a birational morphism $g \colon Z \to Y$.
In what follows, we use the notation of Figure \ref{fig:dual(f)}.
For a birational morphism $Z \to S$ and a curve $C \subset Z$, we denote $(C)_S$ the strict transform of $C$ in $S$. 

By Lemma \ref{lem:deg=2pre} (8), the contraction of $(\Theta_{1,0})_Y$ gives a morphism $h \colon Y \to W$ to the minimal resolution of the Du Val del Pezzo surface $V$ of type $7A_1$.
Let $\xi \colon W \to V$ be the contraction of all the $(-2)$-curves and $\nu = \xi \circ h$. 
By Corollary \ref{7A_1aut} (2), the class divisor group of $W$ is generated by $(\Theta_{1, 4})_W$, $(R_2)_W$, $(Q_2)_W$, $(R_1)_W$, $(Q_1)_W$, $(P_3)_W$, $(P_2)_W$, and any one of $(-2)$-curves.
Since the point $h((\Theta_{1,0})_Y)$ lies on $(\Theta_{1, 4})_W$ and $\pi$ contracts all the $(-2)$-curves, the class divisor group of $X$ is generated by
$(R_2)_X$, $(Q_2)_X$, $(R_1)_X$, $(Q_1)_X$, $(P_3)_X$, $(P_2)_X$, and $(\Theta_{1,0})_X$, whose anti-canonical degrees are one.
Then an easy computation shows the following.
\begin{align*}
    \pi^* (\Theta_{1,0})_X 
    &=&&(\Theta_{1, 0})_Y + (\Theta_{1, 1})_Y + (\Theta_{1, 2})_Y + (\Theta_{1, 3})_Y + 2(\Theta_{1, 4})_Y\\
    &=&&2((\Theta_{1, 0})_Y + \frac12(\Theta_{1, 1})_Y + \frac12(\Theta_{1, 2})_Y + \frac12(\Theta_{1, 3})_Y + (\Theta_{1, 4})_Y)-(\Theta_{1, 0})_Y\\
    &=&&2\nu^*(\Theta_{1,4})_V - (\Theta_{1, 0})_Y,\\
    \pi^* (Q_1)_X 
    &=&&(Q_1)_Y + \frac12 (\Theta_{0, 1})_Y + \frac12 (\Theta_{\alpha_1, 1})_Y\\
    &&&+ \frac12 (\Theta_{1, 1})_Y + (\Theta_{1, 2})_Y + \frac12 (\Theta_{1, 3})_Y + (\Theta_{1, 4})_Y\\
    &=&&((Q_1)_Y + \frac12 (\Theta_{0, 1})_Y + \frac12 (\Theta_{\alpha_1, 1})_Y + \frac12 (\Theta_{1, 2})_Y)\\
    &&&+ ((\Theta_{1, 0})_Y + \frac12 (\Theta_{1, 1})_Y + \frac12 (\Theta_{1, 2})_Y + \frac12 (\Theta_{1, 3})_Y + (\Theta_{1, 4})_Y) - (\Theta_{1, 0})_Y\\
    &=&&\nu^*(Q_1)_V + \nu^*(\Theta_{1,4})_V - (\Theta_{1, 0})_Y,\\
    \pi^* (R_1)_X&=&&\nu^*(R_1)_V + \nu^*(\Theta_{1,4})_V - (\Theta_{1, 0})_Y,\\
    \pi^* (Q_2)_X&=&&\nu^*(Q_2)_V + \nu^*(\Theta_{1,4})_V - (\Theta_{1, 0})_Y,\\
    \pi^* (R_2)_X&=&&\nu^*(R_2)_V + \nu^*(\Theta_{1,4})_V - (\Theta_{1, 0})_Y,\\
    \pi^* (P_2)_X&=&&\nu^*(P_2)_V + \nu^*(\Theta_{1,4})_V - (\Theta_{1, 0})_Y,\\
    \pi^* (P_3)_X&=&&\nu^*(P_3)_V + \nu^*(\Theta_{1,4})_V - (\Theta_{1, 0})_Y.
\end{align*}

Now let us show the assertion.
Let $A$ be an ample $\Z$-divisor on $X$.
By Lemma \ref{lem:KV-2}, we only have to show that $H^1(Y, \sO_Y(-\lceil \pi^*A \rceil))=0$.
By Lemma \ref{lem:KV-1}, we may assume that $(-K_X \cdot A)=1$.
Then $A \sim n_1 (R_2)_X + n_2 (Q_2)_X + n_3 (R_1)_X + n_4 (Q_1)_X + n_5 (P_3)_X + n_6 (P_2)_X + n_7 (\Theta_{1,0})_X$ with $n_1+ \cdots +n_7=1$.
Set $B= n_1 (R_2)_V + n_2 (Q_2)_V + n_3 (R_1)_V + n_4 (Q_1)_V + n_5 (P_3)_V + n_6 (P_2)_V + n_7 (\Theta_{1,4})_V$.
Then we obtain $\pi^* A =\nu^*(B + (\Theta_{1,4})_V) - (\Theta_{1, 0})_Y$.
Since $\nu$ sends $E_h = (\Theta_{1, 0})_Y$ to a smooth point of $V$, the support of $\lceil \nu^*(B + (\Theta_{1,4})_V) \rceil - \nu^*(B + (\Theta_{1,4})_V)$ is contained in $E_\xi$.
Since $E_h$ is disjoint from $E_\xi$ in $Y$, we obtain 
$(E_h \cdot \lceil \pi^*A \rceil) 
= (E_h \cdot \nu^*(B + (\Theta_{1,4})_V) -E_h) =1$.
Hence we have an exact sequence
\begin{align*}
    \xymatrix@C=10pt{
0 \ar[r] &\sO_Y(-\lceil \nu^*(B + (\Theta_{1,4})_V) \rceil ) \ar[r] & \sO_Y(-\lceil \pi^*A \rceil) \ar[r] & \sO_{E_h}(-1) \ar[r] & 0.
}
\end{align*}
Thus $H^1(Y, \sO_Y(-\lceil \pi^*A \rceil)) \cong H^1(Y, \sO_Y(-\lceil \nu^*(B + (\Theta_{1,4})_V) \rceil ))$.
Since $(Y, E_\nu)$ is a log smooth pair, Lemma \ref{lem:KV-2} yields $H^1(Y, \sO_Y(-\lceil \nu^*(B + (\Theta_{1,4})_V) \rceil )) \cong H^1(V, \sO_V(-(B + (\Theta_{1,4})_V))).$
Since $(-K_V \cdot B+(\Theta_{1,4})_V)=2$, Lemma \ref{lem:KV-1} yields $H^1(V, \sO_V(-(B + (\Theta_{1,4})_V))) \cong 0$.
Hence the assertion holds.
\end{proof}

Now we can prove Theorem \ref{pathologies}.

\begin{proof}[Proof of Theorem \ref{pathologies}]
By Theorem \ref{smooth, Intro}, it suffices to show the assertions when $X$ satisfies (NB), i.e., $X$ is listed in Table \ref{table:sing}.
Then the assertions (1) and (2) follow from Propositions \ref{p=3liftable}--\ref{4A_2lift} and \cite[Theorem 2, Table (II)]{Fur} respectively. 
Finally, we show the assertion (3). 
Suppose that $X$ satisfies (NK). 
Then $p=2$ and $X$ satisfies (NL) by Lemma \ref{lem:KV-1} and Theorem \ref{smooth, Intro} (3) respectively. 
The assertion (1) now shows that $\Dyn(X)=7A_1$, $8A_1$, or $4A_1+D_4$.
If $\Dyn(X)=7A_1$, then $X$ satisfies (NK) by \cite[Theorem 4.2 (6)]{CT19} with $(d, q_1, q_2)=(3, 1, 2)$. 
If $\Dyn(X)=8A_1$, then $X$ satisfies (NK) by Proposition \ref{KV:8A_1}.
If $\Dyn(X)=4A_1+D_4$, then $X$ does not satisfy (NK) by Proposition \ref{KV:4A_1+D_4}. 
Hence we get the assertion (3).
\end{proof}

%%%%%%%%%%%%%%%%%%%%%%%%%%%%%%%%%%%%%%%%%%%%%%%%%%%%%%%%%%%%%%%%%%%%%%%%%%%%%%%%%%%%%%%%%%%%%%%%%%%%%%%%%%%%%%%%

\section*{Acknowledgements}
The authors would like to thank Professor Keiji Oguiso and Professor Shunsuke Takagi for their helpful advice and comments.
They are indebted to Professor Hiroyuki Ito for helpful advice on rational quasi-elliptic surfaces.
They would like to thank the referee for valuable advice which improved the paper.
Discussions with Teppei Takamatsu, Yuya Matsumoto, and Takeru Fukuoka on the automorphism groups of Du Val del Pezzo surfaces have been insightful.
They are grateful to Jakub Witaszek for letting them know about Remark \ref{lift remark}.
The authors would like to thank Fabio Bernasconi for telling them the paper \cite{AZ}.
They also wish to express their gratitude to Shou Yoshikawa, Yohsuke Matsuzawa, and Naoki Koseki for helpful discussions and comments.
The authors are supported by JSPS KAKENHI Grant Number JP19J21085 and JP19J14397.

\newcommand{\etalchar}[1]{$^{#1}$}

%\bibliography{hoge.bib}
%\bibliographystyle{alpha}

\end{document}